\documentclass[11pt, reqno]{amsart}
\usepackage[margin=3.6cm]{geometry}

\usepackage{array,booktabs,tabularx}
\usepackage{graphicx}
\usepackage{amssymb,amsmath, amsthm}
\usepackage{enumerate}
\usepackage{wrapfig}
\usepackage{xfrac} 
\usepackage{nicefrac}
\usepackage{color}
\usepackage{hyperref}
\usepackage{esint}
\usepackage{mathtools}
\usepackage{dsfont}
\usepackage{ esint }
\usepackage[show]{ed}
\usepackage{comment}
\usepackage{enumitem}
\usepackage{soul}
\usepackage{needspace}
\usepackage[normalem]{ulem}
\usepackage{tikz}
\usetikzlibrary{
  decorations.pathmorphing,%
  decorations.markings,%
  decorations.pathreplacing,%
  decorations.text,%
  calc,%
  patterns,%
  shapes.geometric,%
  arrows,
  decorations.shapes,
  positioning,
  plotmarks,
  shadings,
  math,
  intersections
}
\tikzset{>=latex} % sets default arrow type in tikz
\tikzset{font=\small}
\tikzset{mark size=1.5pt, mark options=thin}
\tikzset{pin distance=4pt,
  every pin edge/.style={<-, thin, shorten <= -2pt}}
\usepackage{float}

\usepackage{multicol}
\usepackage{tabto}

\newcommand{\TM}{T^*\!M}

\newcommand{\T}{\mathcal{T}}

\newtheorem{theorem}{Theorem}
\newtheorem{corollary}[theorem]{Corollary}

\newtheorem{lemma}{Lemma}
\newtheorem{proposition}[lemma]{Proposition}

\theoremstyle{definition}
\newtheorem{definition}{Definition}
\newtheorem{remark}{Remark}

\newcommand{\FR}{\mathfrak{I}_{_{\!0}}}
\newcommand{\KR}{\mathcal{K}}

\newcommand{\R}{{\mathbb R}}

\newcommand{\G}{{\mathcal G}}

\newcommand{\ep}{\varepsilon}
\newcommand{\Hs}[1]{\text{\raisebox{.2cm}{${_{\!H_{\scl}^{#1}\!\!(M)}}$}}}
\newcommand{\sob}[1]{H_{\scl}^{#1}(M)}
\newcommand{\Hm}{\Hs{\!\!\frac{k-2m+1}{2}}}
\newcommand{\Hl}{\Hs{\frac{k-3}{2}}}
\newcommand{\Hln}{\Hs{\!\!{\frac{n-3}{2}}}}

\newcommand{\x}{\ensuremath{\times}}

\newcommand{\mc}[1]{\mathcal{#1}}
\newcommand{\re}{\mathbb{R}}

\newcommand{\WFh}{\operatorname{WF_h}}
\newcommand{\MSh}{\operatorname{MS_h}}

\newcommand{\sigFlow}{ \mc{L}_{h} }
\newcommand{\smalleq}[1]{\scalebox{.9}{$#1$}}%

\newcommand{\red}[1]{{\color{purple}{#1}}}

\usepackage{xcolor}
\definecolor{light-gray}{RGB}{168,168,168}
\definecolor{dark-gray}{RGB}{110,110,110}
\definecolor{dark-green}{RGB}{34,149,34}
\definecolor{jeffColor}{RGB}{102, 0, 204}

\newcommand{\SNH}{S\!N^*\!H}

\newcommand{\SM}{S^*\!M}

\newcommand{\SigH}{\Sigma_{_{\!H\!,p}}}

\newcommand{\comp}{\operatorname{comp}}

\newcommand{\Tj}{\Lambda_{\rho_j}^\tau(R(h))}
\newcommand{\LambdaH}{\Lambda^\tau_{_{\!\Sigma_{H\!,p}}}}

\newcommand{\LM}{{_{\!L^2(M)}}}
\newcommand{\LMo}{{_{\!L^1({\tilde{H}})}}}
\newcommand{\class}{\delta}

\newcommand{\U}[2]{ U(#1,#2)  }%e^{-\frac{i}{h}A(#2,#1,\tilde{x},hD_{\tilde{x}})}
\renewcommand{\P}{{\scriptscriptstyle P}}

\def\XXint#1#2#3{{\setbox0=\hbox{$#1{#2#3}{\int}$} \vcenter{\hbox{$#2#3$}}\kern-.5\wd0}}

\DeclareMathOperator{\vol}{vol}

\DeclareMathOperator{\supp}{supp}
\DeclareMathOperator{\inj}{inj}

\DeclareMathOperator{\scl}{scl}

\newcommand{\e}{\varepsilon}
\newcommand{\Tinj}{\tau_{_{\!\text{inj}}}}

\DeclareMathOperator{\WF}{WF}
\numberwithin{equation}{section}
\numberwithin{lemma}{section}

\title[]{Eigenfunction concentration via geodesic beams}

\author{Yaiza Canzani}
\address{Department of Mathematics, University of North Carolina at Chapel Hill, Chapel Hill, NC, USA}
\email{canzani@email.unc.edu}

\author{Jeffrey Galkowski}
\address{Department of Mathematics, University College London, London, UK}
\email{j.galkowski@ucl.ac.uk}

\date{}

\begin{document}

\begin{abstract}
In this article we develop new techniques for studying  concentration of Laplace eigenfunctions $\phi_\lambda$ as their frequency, $\lambda$, grows.
The method consists of controlling $\phi_\lambda(x)$ by decomposing $\phi_\lambda$ into a superposition of geodesic beams that run through the point $x$.   Each beam is localized in phase-space on a tube centered around a geodesic whose radius shrinks slightly slower  than $\lambda^{-\frac{1}{2}}$. We control $\phi_\lambda(x)$ by the $L^2$-mass of $\phi_\lambda$ on each geodesic tube and derive a purely dynamical statement through which  $\phi_\lambda(x)$ can be studied. {In particular, we obtain estimates on  $\phi_\lambda(x)$} by decomposing the set of geodesic tubes into those that are non self-looping for time $T$ and those that are. This approach allows for quantitative improvements, in terms of $T$, on the available bounds for $L^\infty$ norms, $L^p$ norms, pointwise Weyl laws, and averages over submanifolds. 
\end{abstract}

\maketitle

%\tableofcontents
%\newpage
%%%%%%%%%%%%%%%%%%%%%%%%%%%%%%%%%%%%%%%%%%%%%%%%%%%%%%%%%%%%%%%%%%%%%%%%%%%%%%%%
%%%%%%%%%%%%%%%%%%%%%%%%%%%%%%%%%%%%%%%%%%%%%%%%%%%%%%%%%%%%%%%%%%%%%%%%%%%%%%%%
%%%%%%%%%%%%%%%%%%%%%%%%%%%%%%%%%%%%%%%%%%%%%%%%%%%%%%%%%%%%%%%%%%%%%%%%%%%%%%%%
%%%%%%%%%%%%%%%%%%%%%%%%%%%%%%%%%%%%%%%%%%%%%%%%%%%%%%%%%%%%%%%%%%%%%%%%%%%%%%%%
\section{Introduction}
%%%%%%%%%%%%%%%%%%%%%%%%%%%%%%%%%%%%%%%%%%%%%%%%%%%%%%%%%%%%%%%%%%%%%%%%%%%%%%%%
%%%%%%%%%%%%%%%%%%%%%%%%%%%%%%%%%%%%%%%%%%%%%%%%%%%%%%%%%%%%%%%%%%%%%%%%%%%%%%%%
%%%%%%%%%%%%%%%%%%%%%%%%%%%%%%%%%%%%%%%%%%%%%%%%%%%%%%%%%%%%%%%%%%%%%%%%%%%%%%%%
%%%%%%%%%%%%%%%%%%%%%%%%%%%%%%%%%%%%%%%%%%%%%%%%%%%%%%%%%%%%%%%%%%%%%%%%%%%%%%%%

On a smooth, compact, Riemannian manifold  $(M^n,g)$ with no boundary, we consider sequences of Laplace eigenfunctions $\{\phi_\lambda\}$ solving 
\begin{equation}
\label{e:eigsYesa}
(-\Delta_g-\lambda^2)\phi_\lambda=0,\qquad{\|\phi_\lambda\|_{_{\!L^2(M)}}=1.}
\end{equation}
 From a quantum mechanics point of view, {\small $|\phi_\lambda(x)|^2$} represents the probability density for finding a quantum particle of energy $\lambda^{2}$ at the point $x\in M$. As a result, understanding how $\phi_\lambda$ concentrates across $M$ is an important problem in the mathematical physics community.

In this article, we construct tools to examine the behavior of $\phi_\lambda$ by decomposing it into geodesic beams.  To study how $\phi_\lambda$ concentrates near $x\in M$, we rewrite $\phi_\lambda$ as a sum of functions, each of which is microlocalized to a shrinking neighborhood of a geodesic that runs through $x$. The {analysis} of this decomposition, including a precise description of the $L^\infty$ behavior of each geodesic beam, yields a bound on $\phi_\lambda(x)$ in terms of the local structure of the $L^2$-mass of $\phi_\lambda$ along each of the geodesic tubes  starting at $x$. In addition, through an application of Egorov's theorem, we obtain estimates on the growth of $\phi_\lambda(x)$ that rely only on the dynamical behavior of geodesics emanating from $x$, and not on  {any other geometric} structure of $(M,g)$. Throughout the article, we refer to the tools developed here as \emph{geodesic beam techniques}.

The term geodesic beam is inspired by Gaussian beams. Recall that, on the round sphere, these are eigenfunctions {that concentrate} in a {$\lambda^{-1/2}$} neighborhood of a closed geodesic that have a Gaussian profile transverse to the geodesic. Gaussian beams have been extensively studied in the math and physics literature (see e.g.~\cite{Laz67,Arn73,Keller71,babich1991short,Dim06,ZeGaussian15,We75,Ra77,Ra82}).  Notably, Ralston~\cite{Ra76} constructed quasimodes associated to stable periodic orbits modelled on Gaussian beams. These references concern modes associated to a single closed geodesic. In contrast, the methods developed here decompose functions into linear combinations of what we call geodesic beams. Each building block is similar to a Gaussian beam in that it is associated to a geodesic and concentrates in a small neighborhood thereof. However, {three facts crucial} to our construction are: that geodesic beams are only locally defined, %and need not resemble Gaussian beams along the entire geodesic,
that the geodesic need not close, {and that they do {not} need to have a Gaussian profile transverse to the geodesic}.

In this article we build the geodesic beam tools and illustrate their application by obtaining quantitative improvements to $L^\infty$ norms for eigenfunctions on certain integrable geometries (see \S\ref{s:spheres}).

In addition, the techniques developed in this paper have remarkable implications in the study of $L^\infty$ norms and averages of eigenfunctions, $L^p$ norms, and pointwise Weyl Laws.  (See \S \ref{s:no conjugate},  \S\ref{s:Lp}, \S \ref{s:Weyl} respectively.) However, all of these applications require some additional non-trivial input e.g. {controlling looping behavior of geodesics}  in~\cite{CG19dyn}, understanding the local geometry {of overlapping tubes} in~\cite{CG19Lp}, and reduction of Weyl remainders to quasimode estimates in~\cite{CG19Weyl}. \emph{We stress that the crucial technique in each application is that of geodesic beams, which are developed in this article}.  We briefly describe the applications to $L^\infty$ norms, averages, $L^p$ norms, and Weyl Laws now:  \smallskip

{\bf{$L^\infty$ norms:}} Beginning in the 1950's, the works of Levitan, Avakumovi\'c, and H\"ormander~\cite{Lev,Ava,Ho68} prove the estimate {\small $\|{\phi_\lambda}\|_{_{\!L^\infty(M)}}=O(\lambda^{\frac{n-1}{2}})$} as $\lambda \to \infty$; known to be saturated on the round sphere. This bound was improved to {\small${\small o(\lambda^{\frac{n-1}{2}})}$} by Sogge, Toth, Zelditch and the second author~\cite{SZ02,SoggeTothZelditch, SZ16I,SZ16II, GT, Gdefect} under various dynamical assumptions at $x$. Notably,~\cite{SZ02} was the first to study $L^\infty$ bounds under purely local dynamical assumptions. When $(M,g)$ has no conjugate points, a quantitative improvement of the form {\small $\|{\phi_\lambda}\|_{_{\! L^\infty}}=O(\lambda^{\frac{n-1}{2}}/\sqrt{\log \lambda})$} has been known since the classical work of B\'erard~\cite{Berard77,Bo16,Randol}. However, until the present time, no quantitative improvements were available without \emph{global geometric} assumptions on $(M,g)$.  In \S\ref{s:no conjugate} we present applications of our geodesic beam techniques giving such improvements.  \medskip

{\bf{Averages:}} Another measure of eigenfunction concentration is the average over a submanifold $H\subset M$ of codimesion $k$. 
In this case, the general bound ~{\small $\int_H{\phi_\lambda}d\sigma_H=O(\lambda^{\frac{k-1}{2}})$} was proved by Zelditch \cite{Zel} and is saturated on the round sphere. This generalized the work of Good and Hejhal~\cite{Good,Hej}. Chen--Sogge~\cite{CS} were the first to obtain {a refinement} on the standard bounds. This work has since been improved under various assumptions by Sogge, Xi, Zhang, Wyman, Toth, and the authors~\cite{SXZ,Wym,Wym2,Wym3,Wym18,CGT,CG17}. As before, none of these results obtain quantitative improvements without {global geometric} assumptions on $(M,g)$.  In \S\ref{s:no conjugate} we present applications of our geodesic beam techniques giving such improvements. \medskip

{\bf{$L^p$ norms:}} 
 Since the seminal work of Sogge~\cite{So88}, it has been known that {\small $\|\phi_\lambda\|_{_{L^p(M)}}={\small O(\lambda^{\delta(p,n)})}$} where  {\small $\delta(p,n)$} depends on how $p$ compares  to the critical exponent {\small ${\small p_c=\tfrac{2(n+1)}{n-1}}$}. Namely,   {\small ${\small \delta(p,n)=\tfrac{n-1}{2} -\tfrac{n}{p}}$} if   {\small $p \geq p_c$} and  {\small ${\small \delta(p,n)=\tfrac{n-1}{4} -\tfrac{n-1}{2p}}$} if  {\small $2\leq p \leq p_c$}. 
When $(M,g)$ has non-positive sectional curvature,  Hassell and Tacy~\cite{HT15} gave quantitative gains over this estimate of the form {\small ${\small O(\lambda^{\delta(p,n)}/ (\log \lambda)^{\sigma(p,n)})}$} when $p> p_c$ and with {\small$\sigma(p,n)=\tfrac{1}{2}$}.   Blair and Sogge~\cite{BS17,BS18} also obtained an improvement when {\small$2<p\leq p_c$} for some {\small $\sigma(p,n)>0$}  smaller than $\tfrac{1}{2}$. In \S\ref{s:Lp} we present applications of our geodesic beam techniques which yield {\small $\sqrt{\log \lambda}$} improvements for $L^p$ norms with $p>p_c$, generalizing those of~\cite{HT15}.  \medskip

{\bf Weyl Laws:}
 Let $\{\lambda_j^{2}\}_j$ be the Laplace eigenvalues of $(M,g)$.  It is well known that 
{\small ${\small \#\{j:\; \lambda_j \leq \lambda\}=\tfrac{\vol(B^{n})\vol(M)}{(2\pi)^n}\lambda^n + E(\lambda)}$} with {\small $E(\lambda)=O(\lambda^{n-1})$} as $\lambda \to \infty$, {where $B^n \subset \mathbb R^n$} is the unit ball. Indeed, this is the integrated version of the more refined statement proved by H\"ormander  in \cite{Ho68} which says that {\small  ${\small \sum_{\lambda_j \leq \lambda}|\phi_{\lambda_j}(x)|^2=\tfrac{\vol(B^{n})}{(2\pi)^n}\lambda^n + E(\lambda,x)}$}  for all $x \in M$, with {\small $E(\lambda,x)=O(\lambda^{n-1})$} uniform for $x \in M$. This estimate has been improved by Sogge--Zelditch~\cite{SZ02} and B\'erard~\cite{Berard77} under various dynamical assumptions. In \S\ref{s:Weyl} we present improvements of these results based on geodesic beam techniques. \ 

\medskip

%%%%%%%%%%%%%%%%%%%%%%%%%%%%%%%%%%%%%%%%%%%%%%%%%%%%%%%%%%%%%%%%%%%%%%%%%%%%%%%%
%%%%%%%%%%%%%%%%%%%%%%%%%%%%%%%%%%%%%%%%%%%%%%%%%%%%%%%%%%%%%%%%%%%%%%%%%%%%%%%%
\subsection{Main results: Localizing eigenfunctions near geodesic tubes}
%%%%%%%%%%%%%%%%%%%%%%%%%%%%%%%%%%%%%%%%%%%%%%%%%%%%%%%%%%%%%%%%%%%%%%%%%%%%%%%%
%%%%%%%%%%%%%%%%%%%%%%%%%%%%%%%%%%%%%%%%%%%%%%%%%%%%%%%%%%%%%%%%%%%%%%%%%%%%%%%%
In this section we present Theorems \ref{t:porcupine efxs} and  \ref{t:coverToEstimate efxs}, which are our main estimates for Laplace eigenfunctions. In \S\ref{s:general} we present much more general versions of these two results, Theorems \ref{t:porcupine} and  \ref{t:coverToEstimate}, that hold for quasimodes of more general operators.

In fact, we  work in the semiclassical framework, writing $\lambda =h^{-1}$ and {letting} $h\to 0^+$. Then, relabeling $\phi_{\lambda}=\phi_h$, we {study} 
\begin{equation}
\label{e:eigsYes}
(-h^2\Delta_g-1)\phi_h=0,\qquad{\|\phi_h\|_{_{\!L^2(M)}}=1.}
\end{equation} 
This rescaling is useful because it allows us to work in compact subsets of phase space, and in particular, near the cosphere bundle $\SM$ where geodesic dynamics naturally take place.

Our main results give an estimate for $\phi_h$ near a point $x\in M$. We now introduce the necessary objects to state these estimates.   We will work with a cover of $S_x^*M$ by {short} geodesic tubes $\Lambda_{\rho}^\tau(R(h))\subset \TM$. This notation {roughly} means that the geodesic tube, $\Lambda_{\rho}^\tau(R(h))$, is the {flowout of a ball of radius $R(h)$ around $\rho$ for times $t\in[-\tau-R(h),\tau+R(h)]$. {We will, in fact, take $\tau>0$ small}. This is similar to} an $R(h)$ thickening (with respect to the Sasaki metric on $T^*M$) of the geodesic of length $2\tau$ centered at {$\rho \in S_x^*M$} (see \eqref{e:tube2} for a precise definition). 
We say that $\{\Lambda_{\rho_j}^\tau(R(h))\}_{j=1}^{N_h}$ is a  \emph{$(\tau, R(h))$-cover} of $S_x^*M$ if it covers { $\Lambda_{\!S_x^*M}^{\tau}( \tfrac{1}{2}R(h))$} {(see Definition~\ref{d: cover} for the definition of a cover and~\eqref{e:tube} for the definition of $\Lambda_{S^*_xM}^\tau(\tfrac{1}{2}R(h))$)}. 

In addition,  a \emph{$\delta$-partition} of $S_x^*M$ associated to the $(\tau, R(h))$-cover  is a collection of functions $\{\chi_j\}_{j=1}^{N_h}\subset S_\delta(\TM;[0,1])$ so that each $\chi_j$ is supported in the tube $\Lambda_{\rho_j}^\tau(R(h))$ and with the property that $\sum_{j=1}^{N_h}\chi_j\geq 1$ on $\Lambda_{S^*_xM}^{\tau}(\tfrac{1}{2}R(h)).$ {(See Appendix~\ref{s:SemiNote} for a description the symbol class $S_\delta$, {and Definition~\ref{d: cover} for the definition of a $\delta$-partition}.)}

The functions $\chi_j$ are used to microlocalize $\phi_h$ to the tubes $\Lambda_{\rho_j}^\tau(R(h))$. We refer to $Op_h(\chi_j)\phi_h$ as a \emph{geodesic beam through $x$}. They  are constructed in Proposition~\ref{l:nicePartition} and have the additional property that $Op_h(\chi_j)$ nearly commutes with $(-h^2\Delta_g-1)$ near $x$ {(so that these localizers do not destroy the property of being a quasimode locally near $x$)}. (See also Step 2 in the proof of Theorem~\ref{t:porcupine}.)
{The fact that $Op_h(\chi_j)$ nearly commutes with $(-h^2\Delta_g-1)$ requires that we work with geodesic tubes of positive length, $\tau$, independent of $h$ rather than localizing to balls of radius $R(h)$ centered in $S^*_xM$.}

In the following result, we control $\phi_h(x)$ by the $L^2$-mass of the geodesic beams through $x$. 

%%%%%%%%%%%%%%%%%%%%%%%%%%%%%%%%%%%%%%%%%%%%%%%%%%%%%%%%%%%%%%%%%%%%%%%%%%%%%%%
\begin{theorem}
\label{t:porcupine efxs} 
 Let $x \in M$. There exist 
$\tau_0=\tau_0(M,g)>0$, $R_0=R_0(M,g)>0,$
 $C_{n}>0$ depending only on $n$, so that the following holds. 

Let $0<\tau\leq \tau_0$,  $0\leq \class<\frac{1}{2}$,  and ${8}h^\delta\leq R(h) \leq  R_0$. Let $\{\chi_j\}_{j=1}^{N_h}$ be a $\delta$-partition for $S_x^*M$ associated to a $(\tau, R(h))$-cover. Let $N>0$.

Then, there are $h_0=h_0(M,g,\{\chi_j\},\delta)>0$ and $C_{_{\! N}}>0$ with the property that for any $0<h<h_0$ and $\phi_h$ satisfying~\eqref{e:eigsYes},
\begin{align*}
\|\phi_h\|_{L^\infty(B(x,h^\delta))}
&\leq C_{n} \tau^{-\frac{1}{2}}{h^{\frac{1-n}{2}}R(h)^{\frac{n-1}{2}}}\sum_{j=1}^{N_h}\|Op_h(\chi_j)\phi_h\|_{\LM} +C_{_{\!N}}h^N\|\phi_h\|_{\LM}.
\end{align*}
{Moreover, the constants ${h_0}$ and $C_{_{\!N}}$ are uniform for $\chi_j$ in bounded subsets of $S_\delta$.}
\end{theorem}

%%%%%%%%%%%%%%%%%%%%%%%%%%%%%%%%%%%%%%%%%%%%%%%%%%%%%%%%%%%%%%%%%%%%%%%%%%%%%%%% 
 {Crucially, this estimate makes no assumptions on the geometry of $M$ or the dynamics of the geodesic flow. Information on the dynamics of the geodesic flow will later allow us to control the $L^2$ mass of the geodesic beams (see Theorem~\ref{t:coverToEstimate efxs}).}
 
This result is a consequence of the more general and stronger result given in Theorem~\ref{t:porcupine} below. {(See Remark \ref{r:how to prove results for efxs} for the proof.)} Indeed, the latter is stated as a bound for $\int_H u_h d\sigma_H$, where $H\subset M$ is a general submanifold and  $u_h$ is a quasimode for a pseudodifferential operator with a real, classically elliptic symbol with respect to which $H$ is conormally transverse. Note that when $H=\{x\}$ we have $\int_H u_h d\sigma_H=u_h(x)$. See \S\ref{s:general} for a detailed description.\smallskip

One can conclude from  Theorem~\ref{t:porcupine efxs} that, in order to have maximal sup-norm growth at a point, an eigenfunction must have a component with $L^2$ norm bounded from below that is distributed in the same way as the canonical example on the sphere (up to scale $h^\delta$ for all $\delta<\frac{1}{2}$). Indeed, {if one restricts attention to $(\tau,r)$ covers of $S_x^*M$ without too many overlaps (see Definition~\ref{d:good cover})}
it follows from Theorem~\ref{t:porcupine efxs} that there exists $C_n>0$, so that for all $\e>0$, if 
\[\#\Big\{j:\;  { { \e^2}} \, R(h)^{n-1} \leq \|Op_h(\chi_j)\phi_h\|^2_{_{\!L^2(M)}} \leq \small \frac{R(h)^{n-1}}{{ \e^2}} \Big\} \leq {\e^2} N_h,\] then
$\|\phi_h\|_{L^\infty(B(x,h^\delta))}
\leq \e C_{n} \tau^{-\frac{1}{2}}h^{\frac{1-n}{2}}$.

%\red{Indeed, let 
%\begin{gather*}
%\mc{Z}_<:=\{j\mid  \|Op_h(\chi_j)\phi_h\|^2_{_{\!L^2(M)}}<{ \e^2} \, R(h)^{n-1}\},\qquad\mc{Z}_{>}:=\{j\mid \small \frac{R(h)^{n-1}}{ \e^2}<  \|Op_h(\chi_j)\phi_h\|^2_{_{\!L^2(M)}} \},\\
%\mc{Z}:=\{j\mid {  \e^2} \, R(h)^{n-1}\leq \|Op_h(\chi_j)\phi_h\|^2_{_{\!L^2(M)}}\leq  \small \frac{R(h)^{n-1}}{ \e^2} \}
%\end{gather*}
%Then, since there is $C_n>0$ such that $|\mc{Z}_>|\leq C_n\e^2 R(h)^{1-n}$, and Definition~\ref{d:good cover} implies there is $C_n>0$ such that $N_h\leq C_nR(h)^{1-n}$, applying Cauchy--Schwarz in the last term, we have
%\begin{align*}
%R(h)^{\frac{n-1}{2}}\sum_j\|Op_h(\chi_j)\phi_h\|_{_{\!L^2(M)}}&\leq C_n\e +R(h)^{\frac{n-1}{2}}\sum_{\mc{Z}}\|Op_h(\chi_j)\phi_h\|_{_{\!L^2(M)}}+ C_n\e\\
%&\leq 2C_n\e +C_n |\mc{Z}|\e^{-1}\leq C_n\e.
%\end{align*}
%Together with Theorem~\ref{t:porcupine efxs} this implies the estimate.}
%Therefore, for the sup-norm bound to be saturated, there must be a positive proportion of the $N_h$ tubes for which 
%the mass $\|Op_h(\chi_j)\phi_h\|_{_{\!L^2(M)}}$ is comparable to $R(h)^{\frac{n-1}{2}}$. This is the behavior displayed by
%the the zonal harmonic at the poles.
%}

To understand Theorem~\ref{t:porcupine efxs} heuristically, one should think of {\small $\|Op_h(\chi_j)\phi_h\|_{L^2(M)}$} as measuring the $L^2$ mass of $\phi_h$ on the tube of radius $R(h)$ around a geodesic that runs through the point $x$. Since {\small $\vol (\supp \chi_j)\asymp R(h)^{n-1} $}, an individual term in the sum in Theorem~\ref{t:porcupine efxs} is then
\begin{align*}
R(h)^{\frac{n-1}{2}}{\|Op_h(\chi_j)\phi_h\|_{L^2(M)} } %&=\Bigg(\frac{\|Op_h(\chi_j)u\|_{L^2(M)}^2}{R(h)^{n-1}}\Bigg)^{\frac{1}{2}}R(h)^{n-1}\\
&\asymp {\small \Bigg(\frac{\|Op_h(\chi_j)\phi_h\|_{L^2(M)}^2}{\vol(\supp \chi_j)}\Bigg)^{\frac{1}{2}} }\vol(\supp \chi_j),
\end{align*}
where $\vol$ is the volume measure on $S^*_xM$ induced by the Sasaki metric on $T^*M$. In particular,  the  sum on the right of the estimate in Theorem~\ref{t:porcupine efxs} can be interpreted  as
$
\int_{S^*_xM} \big|\frac{d\mu}{d\vol}\big|^{\frac{1}{2}}d\vol,
$
where $\mu$ is the measure giving the distribution of the mass squared of $\phi_h$ on ${{\!S^*_xM}}$. 
This statement can be made precise by using defect measures (see~\cite[Theorem 6]{CG17}), but the results using defect measures can only be used to obtain $o(1)$ improvements on eigenfunction bounds.

We emphasize now that Theorem~\ref{t:porcupine efxs} is the key estimate for the proofs of all the applications to $L^\infty$-norms, $L^p$-norms, and Weyl Laws stated in  \S \ref{s:no conjugate},  \ref{s:Lp}, \ref{s:Weyl}, respectively.

At first sight it may seem that it is not easy to extract information from the upper bound provided in Theorem~\ref{t:porcupine efxs}. However, the strength of this bound is showcased in our next result, Theorem \ref{t:coverToEstimate efxs}. The latter combines  the  analytical bound of Theorem~\ref{t:porcupine efxs} together with Egorov's Theorem to obtain a purely dynamical statement.  Indeed, $\phi_h(x)$ is controlled by covers of $\Lambda^\tau_{S^*_xM}({\tfrac{1}{2}R(h)})$ by  ``good" tubes that are non self-looping under the geodesic flow, {$\varphi_t:=\exp(tH_{|\xi|_g})$ (where $H_{|\xi|_g}$ is the Hamiltonian vector field of $|\xi|_g$)}, and ``bad" tubes whose {number} is small. %In fact, Theorems~\ref{t:noConj2},~\ref{t:noConj1}, and~\ref{T:applications}  are reduced to a purely dynamical argument together with an application of Theorem~\ref{t:coverToEstimate}.}

\begin{definition}
(non-self looping sets) For $0<t_0<T_0$, we say that $A\subset T^*\!M$ is \emph{$[t_0,T_0]$ non-self looping} if 
\begin{equation}\label{e:nonsl}
\bigcup_{t=t_0}^{T_0}\varphi_t(A)\cap A=\emptyset\qquad \text{ or }\qquad  \bigcup_{t=-T_0}^{-t_0}\varphi_t(A)\cap A=\emptyset.
\end{equation}
\end{definition}

The goal of our next result is to obtain quantitative control of $\phi_h(x)$ by splitting the geodesic tubes into ``good" tubes $\{\Lambda_{\rho_j}^\tau(R(h))\}_{j \in \mc{G}_\ell}$ that are {$[t_\ell,T_\ell]$} non self-looping and ``bad" tubes {$\{\Lambda_{\rho_j}^\tau(R(h))\}_{j \in \mc{B}}$} {that may be {self-}looping}. The quantitative control is then given in terms of {$t_\ell$}, $T_\ell$, $|\mc{G_\ell}|$, and $|\mc{B}|$. {Recall that $\tau>0$ is a small parameter so the tubes $\Lambda_{\rho}^\tau(R(h))$ do not see the global dynamical structure of the geodesic flow. It is only when $T_\ell\gg \tau$ that one encounters this information.} 

{It is convenient to work with covers by tubes for which the number of overlaps is controlled. Indeed, we say that a $(\tau, {R(h)})$- covering by tubes is a  $(\mathfrak{D}, \tau, {R(h)})$-good covering, if it can be split into $\mathfrak{D}>0$ families of disjoint tubes. See Definition \ref{d:good cover} for a precise definition. In Proposition \ref{l:cover} we prove that one can always work with $(\mathfrak{D}_n, \tau, {R(h)})$-good coverings, where $\mathfrak{D}_n$ only depends on $n$.}

In what follows we write $\Lambda_{\max}$ for the maximal expansion rate {of the flow}
and $T_e(h)$ for the  Ehrenfest time $T_e(h):=\frac{\log h^{-1}}{2\Lambda_{\max}}$ (see \eqref{e:Lmax}).
%Note that $\Lambda_{\max}\in[0,\infty)$ and if $\Lambda_{\max}=0$, we may replace it by an arbitrarily small positive constant. 

%%%%%%%%%%%%%%%%%%%%%%%%%%%%%%%%%%%%%%%%%%%%%%%%%%%%%%%%%%%%%%%%%%%%%%%%%%%%%%%%

\begin{theorem}
\label{t:coverToEstimate efxs}
 Let $x\in M$, $0<\delta<\frac{1}{2}$. There exist positive constants $h_0=h_0(M,g,\delta)$, $\tau_0=\tau_0(M,g)$, $R_0=R_0(M,g)$, and $C_{n}$ depending only on $n$, so that for all $0<\tau {\leq}\tau_0$ and $0<h<h_0$ the following holds.

Let ${8}h^\delta\leq R(h){\leq}R_0$, and  $\{\Lambda_{\rho_j}^\tau(R(h))\}_{j=1}^{N_h}$ be a {$(\mathfrak{D},\tau, R(h))$-good cover for $S^*_xM$ for some $\mathfrak{D}>0$}. Let $0\leq\alpha< 1-2{\limsup_{h\to 0}\frac{\log R(h)}{\log h}}$ and suppose there exists a partition of $\{1,\dots, N_h\}$ into $\mc{B}$ and  $\{\mc{G}_\ell\}_{\ell \in \mathcal L}$ such that for every $\ell \in \mathcal L$ there exist $T_\ell=T_\ell(h)>0$ {and $t_\ell=t_\ell(h)>0$} with $t_\ell(h)\leq T_\ell(h)\leq {2} \alpha T_e(h)$   {such that}  
$$
\bigcup_{j\in \mc{G}_\ell}\Lambda_{\rho_j}^\tau(R(h))\;\;\text{ is }\;\;[t_\ell,T_{\ell}]\text{ non-self looping}.
$$
Then, {for all $N>0$ there exists $C_{_{\!N}}=C_{_{\!N}}(M,{g}, N,\tau,\delta)>0$} so that for $\phi_h$ solving~\eqref{e:eigsYes}
\begin{align*}
\|\phi_h\|_{L^\infty(B(x,h^\delta))}\!
\leq C_{n}{\mathfrak{D}}\tau^{-\frac{1}{2}}h^{\frac{1-n}{2}}R(h)^{\frac{n-1}{2}}\!\!
\Bigg(|\mc{B}|^{\frac{1}{2}}\!+\!\sum_{\ell \in \mathcal L }\frac{|\mc{G}_\ell|^{\frac{1}{2}} {t_\ell^{\frac{1}{2}}}}{T_\ell^{\frac{1}{2}}}\Bigg)\|\phi_h\|_{\LM} \!\!\!+\! C_{_{\!N}}h^N \|\phi_h\|_{\LM}.
\end{align*}
\end{theorem}
%%%%%%%%%%%%%%%%%%%%%%%%%%%%%%%%%%%%%%%%%%%%%%%%%%%%%%%%%%%%%%%%%%%%%%%%%%%%%%%%

\begin{remark}
{Note that, since the tubes $\Lambda_{\rho_j}^\tau(R(h))$ are essentially time $\tau$ flowouts of balls around $\rho_j$ with radius $R(h)$,  if the ball of radius $R(h)$ around $\rho_j$ is $[t-\tau,T+\tau]$ non-self looping then $\Lambda_{\rho_j}^\tau(R(h))$ is $[t,T]$ non-self looping. Therefore, we could replace the non-self looping assumption on $\Lambda_{\rho_j}^\tau(R(h))$ in Theorem~\ref{t:coverToEstimate efxs} by an analogous non-self looping assumption on $B(\rho_j,R(h))$. Note, however, that these balls \emph{cannot} be replaced by balls inside $S^*_xM$. We need them to have full dimension so that smooth cutoffs can be supported inside $\Lambda^\tau_\rho(R(h))$. Moreover, it is necessary that they encode quantitative information on how geodesics near the center of $\Lambda_{\rho}^\tau(R(h))$ return close to $x$.}
\end{remark}

 This result is a consequence of the more general and stronger result given in Theorem~\ref{t:coverToEstimate}. {See Remark \ref{r:how to prove results for cover efxs} for the proof}. {As with the previous theorem, the generalization is stated for averages over submanifolds of quasimodes of general operators.}  % Indeed, the latter is stated as a bound for $\int_H u_h d\sigma_H$, where $H\subset M$ is a general submanifold and  $u_h$ is a quasimode for a self-adjoint pseudodifferential operator with a  classically elliptic symbol with respect to which $H$ is conormally transverse. 
{See \S\ref{s:general} for a detailed explanation. For examples where Theorem~\ref{t:coverToEstimate efxs} is applicable see \S\ref{s:app with integrable} and \S\ref{s:examples}.}
\smallskip

{We note that Theorem~\ref{t:coverToEstimate efxs}  distinguishes much finer features than that of self-conjugacy with maximal multiplicity. Indeed, the theorem can be used to obtain estimates at points \emph{all} of whose geodesics return; provided the geodesics through the point have some additional non-recurrent structure (e.g. the umbillic points on the triaxial ellipsoid; see \S\ref{s:examples}). In particular, this estimate distinguishes recurrent structure and non-recurrent structure {as in Definition~\ref{d:non-rec}}. At this point, we do not know to what extent it distinguishes periodic structure from recurrent structure.}

Theorem~\ref{t:coverToEstimate efxs} reduces estimates on $\phi_h(x)$ to the construction of covers of $\Lambda_{_{\!S^*_xM}}^\tau({\tfrac{1}{2}R(h)})$ by sets with appropriate structure. Here $\Lambda_{_{\!S^*_xM}}^\tau({\tfrac{1}{2}R(h)})$ denotes a $ \tfrac{1}{2}R(h)$ thickening of the set of {geodesics} through $x$, {see \eqref{e:tube}}.  If there is a cover of $\Lambda^\tau_{_{\!S^*_xM}}({\tfrac{1}{2}R(h)})$ by ``good" sets $\{G_\ell\}_{\ell\in L}$ and a ``bad" set $B$, with  every $G_\ell$ being $[t_\ell(h),T_\ell(h)]$ non-self looping, the estimate reads
\[
\|\phi_h\|_{L^\infty(B(x,h^\delta))}
\leq {C_n \mathfrak{D}}\tau^{-\frac{1}{2}}h^{\frac{1-n}{2}}
\left(
\![\vol(B)]^{\frac{1}{2}}
+\sum_{\ell \in \mathcal L }\frac{[\vol(G_\ell)]^{\frac{1}{2}} {t_\ell^{\frac{1}{2}}}}{T^{\frac{1}{2}}_\ell}\right)\!\!\|\phi_h\|_{\LM},
%\left.\qquad\qquad\qquad\qquad\qquad\qquad\!+\!\sum_{\ell \in L}\frac{[\Sig(G_\ell)]^{\frac{1}{2}}t_\ell^{\frac{1}{2}}(h)T_\ell^{\frac{1}{2}}(h)}{h}\;\|Pu\|_{\LM}\!\right).
\]
where $\vol$ denotes the volume induced on $S_x^*M$ by the Sasaki metric on $T^*\!M$, and where  we write $\vol(A)=\vol(A\cap S^*_xM)$ for $A\subset T^*\!M$. The additional structure required on the sets $G_\ell$ and $B$ is that they consist of a union {of} tubes $\Lambda_{{\rho_i}}^\tau({R(h)})$  and that $T_\ell(h)<2(1-2\delta)T_e(h)$. 

\begin{figure}%{l}{0.55\textwidth}
 \begin{tikzpicture}
    \def \T{2.7};
    \def \t{1.6};
    \def \rone{1};
    \def \steptheta{2.5};
    \def \x{3};
    \def \y{\x};
    \def \dotrad{.75pt};
    \def \tau{.3};
    \def \R{.01};
    \begin{scope}[scale={2.7/\T}]
    \draw[blue,ultra thick] (0,0) circle (\T);
    \draw[jeffColor,thick] (0,0) circle (\t);

    \draw[gray] (0,0) circle (\rone);
    \begin{scope}
    \clip (-\x,-\y) rectangle (\x,\y);
    \foreach \i in {-\x,...,\x}
    {
    \foreach \j in {-\y,...,\y}
    {
    \draw[light-gray,dashed] (0,0)--(\x*\i,\x*\j);
    \node at (\i,\j)[circle,fill,inner sep=\dotrad]{};
    };
    };
    \end{scope}
    \begin{scope}
    \foreach \i in {-\x,...,\x}
    {
    \foreach \j in {-\y,...,\y}
    {
    \node at (\i,\j)[circle,fill,inner sep=\dotrad]{};
    };
    };
    \end{scope}
    \foreach \t in {0,\steptheta,...,359.9}
    {
    \filldraw[rotate=\t,dark-green] (\rone-\tau-\R,-\R) rectangle (\rone+\tau+\R,\R);
    };
    \foreach \i in {-\x,...,\x}
    {
    \foreach \j in {-\y,...,\y}
    {
     \pgfmathparse{(abs(\j)^2 +abs(\i)^2)<= (\T)^2 ? int(1) : int(0)};
    \ifnum\pgfmathresult=1 
        \pgfmathparse{abs(\i)^2+abs(\j)^2>(\t)^2 ?int(1):int(0)};
        \ifnum\pgfmathresult=1
          \filldraw[rotate={atan2(\j,\i)},orange] (\rone-\tau-\R,-\R) rectangle (\rone+\tau+\R,\R);
        \fi
    \fi

    };
    };
    
        \draw[->] (-10:\T+1) node[right]{$S^*_{{x}}M$}--(-10:\rone);
        \draw[blue,->](80:\T+1)node[blue,above]{$\varphi_T(S^*_{{x}}M)$}--(80:\T);
     \draw[->,jeffColor] (39:\T+1.8)node[jeffColor,right]  {$\smalleq{\varphi_{t}(S^*_{{x}}M)}$}--(39:\t);
     \node at (0,0)[circle,fill,red,inner sep=2*\dotrad]{};
     \draw (0,0) node[right,red]{${\bf{x}}$};
     \end{scope}
    \end{tikzpicture}
\label{f:cover}
\end{figure}
With this in mind, Theorem~\ref{t:coverToEstimate efxs} should be thought of as giving a non-recurrent condition on $S_x^*M$ which guarantees quantitative improvements over the standard bounds {(see Definition~\ref{d:non-rec} for a precise explanation of what we mean by non-recurrent structure)}. In particular, taking $T_\ell$, {$t_\ell$}, $G_\ell$ and $B$ to be $h$-independent can be used to recover the dynamical consequences in~\cite{CG17,Gdefect} (see~\cite{GJEDP} and Section~\ref{s:relateOld}).

In \S \ref{s:spheres} we illustrate how to build covers by good and bad tubes in some integrable geometries, and how to use them to obtain quantitative improvements over the known $L^\infty$-bounds.  {In the figure we illustrate how to cover $S_x^*M$ with  ``good" tubes  (green) and  ``bad" tubes (orange) for a point $x$ on the square flat torus. The grid represents the integer lattice on the universal cover of the torus. In the figure,  there is \emph{only one index} i.e. $\ell=1$, and we chose {$t_\ell=t=1.6$}, $T_\ell=T{=2.7}$, {$\tau=0.2$, and $R=0.01$}. {In the figure, the length of the green/orange tubes is $2(\tau+R)$. Note that some of the green tubes \emph{are not} $[3\tau,T]$ non-self looping but are $[t,T]$ non-self looping e.g. the tube at angle $\pi/4$.} {In practice, to obtain quantitative gains, one needs to work with $T\to \infty$. The figure is drawn for \emph{one} relatively small $T$ because choosing a larger $T$ makes the figure illegible.} A tube is ``bad" if the geodesic generated by it returns to $x$ in time between $t$ and $T$.} Note, in addition, that $t_\ell$ must be positive since our tubes have finite, positive width in the flow direction. Also, a set \emph{may} be  $[t_0,T]$ self-looping, {but not $[\tilde{t}_0,T]$ self-looping for some $\tilde{t}_0>t_0$ e.g. a {neighborhood, $U\setminus V\subset T^*M$, where $U$ is a neighborhood} around an unstable hyperbolic closed geodesic }{in phase space and $V$ is a slightly smaller neighborhood}. While, at the moment we do not have examples where it is necessary to send $t_\ell\to \infty$ with $h$, we anticipate this will be useful in the future. 

{
To understand why it is in general useful to have families of tubes $\mc{G}_\ell$ with different looping times, $[t_\ell,T_\ell]$, we consider the following setup. We assume that the geodesic flow is exponentially contracting in the sense that 
$$
\|d\varphi_t|_{S^*_xM}\|\leq Ce^{-Ct}.
$$
For simplicity, let $\dim M=2$. The way in which we work with the assumption  on the geodesic flow is that the flow out of an arc of length $R$ in $S_x^*M$ will have length $e^{-CT}R$ upon return to $S_x^*M$ at time $T$. We, in general, do not have information about the place to which the arc returns. Suppose we want to cover $S^*_xM$ with tubes of radius $R$ and divide them into $[t_0,T(h)]$ non-self looping collections $\mc{G}_\ell$ such that Theorem~\ref{t:coverToEstimate efxs} gives a $\log h^{-1}$ gain.
Note that, for simplicity, we identify each tube with  the arc of length $R$ that is formed by its intersection with $S_x^*M$. Since $R\geq h^\delta$, and, in order to get a $\log h^{-1}$ improvement, we must take $T(h)\sim \log h^{-1}$, we have $R\geq e^{-CT(h)}$. 

To simplify the situation further, we discretize the time and imagine that the return map, $\Phi$, has the properties above. 
To produce a non-self looping collection, we start with an arc $A_0$ of length $\sim 1$.
To construct a $[t_0,T(h)]$ non-self looping set, $G_0$, we let
$$
A_1:=\bigcup_{t_0\leq k\leq T(h) }\Phi^k(A_0)\cap A_0,\qquad G_0:=A_0\setminus A_1.
$$
Since we do not know the directions in which $A_0$ returns, $A_1$ apriori consists of intervals of size $e^{-C},e^{-2C},\dots,e^{-CT(h)}$. {Hence, $A_1$} has volume $\sim e^{-C}$ and is $[t_0,T(h)]$ self-looping. In order to get a $T(h)^{-1}$ improvement with only one {$T_\ell(h)=T(h)$}, any set which is $[t_0,T(h)]$ self looping must have volume $\leq CT(h)^{-1}$. 
Since $A_1$'s volume is $\gg T(h)^{-1}$, we must iterate this process by putting 
$$
A_\ell:=\bigcup_{t_0\leq k\leq T(h)}\Phi^k(A_{\ell-1})\cap A_{\ell-1},\qquad G_{\ell-1}:=A_{\ell-1}\setminus A_\ell.
$$
Apriori, $A_\ell$ has volume $\sim e^{-C\ell}$, is $[t_0,T(h)]$ self-looping, and consists of intervals of size $e^{-C\ell},e^{-C(\ell+1)},\dots e^{-C(T(h)+\ell)}$.
Therefore, in order to gain $T(h)^{-1}$ in our estimates, we must iterate until $e^{-C\ell}\sim T(h)^{-1}$. That is, $\ell(h)\sim \log {T(h)}$. Note that in this case the smallest arc in $A_{\ell(h)}$ has length $$e^{-C(T(h)+\ell(h))}\sim {h^CT(h)^{-C}}.$$
Now, depending on $C$, this may be $\ll h^\delta$, which is the scale of our cover. There are a two ways around this. We could shrink $T(h)$ so that this scale is above $R$. However, this would be somewhat unnatural since then our dynamical gain would necessarily depend on the contraction rate. So that we may use our original $T(h)$, while still having a scale above $h^\delta$, 
we shrink the non-self looping times at each step so that $G_\ell$ is $e^{-{C\ell}/2}T(h)$ non-self looping. In doing this, we have that $G_\ell$ is $[t_0,e^{-C\ell/2}T(h)]$ non-self looping and has volume $\sim e^{-C\ell}$. In addition, the minimum size of an interval in $A_\ell$ is $e^{-\sum_{j=0}^\ell e^{-Cj/2}T(h)}$. Iterating until $\ell\sim \log T(h)$, then enables us to obtain our estimates.}

%We refer the reader to the discussion of the triaxial ellipsoid \S\ref{s:triaxial ellipsoid}.

In the following sections, \S \ref{s:no conjugate}, \S \ref{s:Lp}, \S \ref{s:Weyl}, we showcase a {few of the many} applications of Theorem \ref{t:coverToEstimate efxs} in obtaining quantitative improvements for $L^\infty$ norms, $L^p$ norms, pointwise Weyl laws, and averages over submanifolds.

%%%%%%%%%%%%%%%%%%%%%%%%%%%%%%%%%%%%%%%%%%%%%%%%%%%%%%%%%%%%%%%%%%%%%%%%%%%%%%%%
%%%%%%%%%%%%%%%%%%%%%%%%%%%%%%%%%%%%%%%%%%%%%%%%%%%%%%%%%%%%%%%%%%%%%%%%%%%%%%%%
%%%%%%%%%%%%%%%%%%%%%%%%%%%%%%%%%%%%%%%%%%%%%%%%%%%%%%%%%%%%%%%%%%%%%%%%%%%%%%%%
%\section{Applications}
%%%%%%%%%%%%%%%%%%%%%%%%%%%%%%%%%%%%%%%%%%%%%%%%%%%%%%%%%%%%%%%%%%%%%%%%%%%%%%%%
%%%%%%%%%%%%%%%%%%%%%%%%%%%%%%%%%%%%%%%%%%%%%%%%%%%%%%%%%%%%%%%%%%%%%%%%%%%%%%%%
%%%%%%%%%%%%%%%%%%%%%%%%%%%%%%%%%%%%%%%%%%%%%%%%%%%%%%%%%%%%%%%%%%%%%%%%%%%%%%%%
\subsection{Improvements to $L^\infty$-norms and averages}\label{s:no conjugate}
In this section we introduce some of the applications of geodesic beam techniques to the study of the $L^\infty$ norms of $\phi_h$, and of the averages $\int_H \phi_h d\sigma_H$ over a submanifold $H \subset M$. The goal is  to obtain quantitative improvements on the known bounds \cite{Ho68,Zel}
\begin{equation}\label{e:Hor bound}
\phi_h(x)=O(h^{\frac{1-n}{2}}) \qquad \text{and} \qquad \int_H \phi_h(x) d\sigma_H = O(h^{\frac{1-k}{2}}),
\end{equation}
where $k$ is the codimension of $H$. These bounds are sharp since they are, for example, saturated on the round sphere. Note that the right hand estimate includes the left if we take $H=\{x\}$.
In \S\ref{s:app with no conjugate} we present applications of our geodesic beam techniques to studying eigenfunction growth on manifolds with no conjugate points, or whose geometries satisfy a weaker condition. These results, and many more, can be found in \cite{CG19dyn}. In \S\ref{s:app with integrable}  we present applications to obtaining quantitative improvements of $L^\infty$ norms in integrable geometries. The proofs of these and more general results are presented in \S\ref{s:spheres}.

\subsubsection{{Results under conjugate point assumptions}} \label{s:app with no conjugate} %{\cs reference paper, edit flashiness\cs} 
 It is well known that the $L^\infty$ bound in \eqref{e:Hor bound} is saturated on the round sphere if one chooses $\phi_h$ to be a zonal harmonic that peaks at the given point $x \in S^n$. {This phenomenon is possible since} all geodesics through $x$ are closed. {In addition, on the sphere every point is maximally self-conjugate. } In general, a point $x \in M$ is said to be conjugate to $y \in M$ if there exists a unit speed geodesic $\gamma$ joining $x$ and $y$, together with   a non-trivial   Jacobi field along $\gamma$ that vanishes at $x$ and $y$. The number of such Jacobi fields that are linearly independent is called the multiplicity of $x$ with respect to $y$ and is always bounded by $n-1$. When the multiplicity equals $n-1$ the point $x$ is said to be maximally conjugate to $y$. {As a consequence of our geodesic beam techniques, we obtain quantitative improvements on the $L^\infty$ norm of an eigenfunction near a point $x$ that, loosely speaking, is not maximally self-conjugate. }
% \sout{For sup-norm estimates the relevant condition is that of maximal self conjugacy, i.e. $x$ being maximally conjugate to $x$.}} {\sout{For example, on $S^n$ every point is maximally self-conjugate. We now present an application of {our} {geodesic beams techniques} that gives a quantitative improvement on the $L^\infty$-norm of an eigenfunction near a point $x$ that, loosely speaking, is not maximally self-conjugate.} }

Consider the set $\Xi$ of unit speed geodesics {on $(M,g)$} and define
\begin{equation}\label{e:conjugate set}
\mc{C}_x^{r,t}:=\big\{ \gamma(t) \big| \, \gamma\in \Xi, \, \gamma(0)=x,\,\exists\, n-1\text{ conjugate points to } x \text{ in }\gamma(t-r,t+r)\big\},
\end{equation}
where we count conjugate points with multiplicity.
 %Next, for  a set $V \subset M$ define
 %$
%\mc{C}_{_{\!V}}^{m,r,t}:=\bigcup_{x\in V}\{\gamma(t): \gamma\in \Xi_x^{m,r,t}\}.%\qquad% \mc{C}_{_{\!V}}^{m}:=\bigcap_{r>0}\,\bigcap_{T>0}\,\overline{\bigcup_{t>T}\mc{C}_{_{\!V}}^{m,r,t}}.
%$
Note that if $r_t \to 0^+$ as $|t|\to \infty$, then saying that $x \in \mc{C}_x^{r_t,t}$ for $t$ large indicates that $x$ behaves like  a point that is maximally self-conjugate. This is the case for every point on the sphere.  The following result applies under the assumption that this does not happen and obtains quantitative improvements in that setting.  {The obvious case where our next theorem applies is that of manifolds without conjugate points, where $\mc{C}_x^{r,t}=\emptyset$ for $0<r<|t|$. In addition, the theorem applies to \emph{all} non-trivial product manifolds $M=M_1\times M_2$ (see \S~\ref{s:examples})}. 

%%%%%%%%%%%%%%%%%%%%%%%%%%%%%%%%%%%%%%%%%%%%%%%%%%%%%%%%%%%%%%%%%%%%%%%%%%%%%%%%
\begin{theorem}[{\cite[Theorem 1]{CG19dyn}}]
\label{t:noConj2}
Let $V \subset M$ and assume that there exist $t_0>0$ and $a>0$  so that 
$$
\inf_{x\in V}d\big(x, \mc{C}_{x}^{r_t,t}\big)\geq r_t,\qquad\text{ for } t\geq t_0,
$$
with $r_t=\frac{1}{a}e^{-at}.$ 
Then, there exist $C>0$ and $h_0>0$ so that for $0<h<h_0$ and $u \in {\mc{D}'}(M)$
$$
\|u\|_{L^\infty(V)}\leq Ch^{\frac{1-n}{2}}\left(\frac{\|u\|_{\LM}}{\sqrt{\log h^{-1}}}\;+\; \frac{\sqrt{\log h^{-1}}}{h}\big\|(-h^2\Delta_g-1)u\big\|_{\Hln}\right).
$$
\end{theorem}
%%%%%%%%%%%%%%%%%%%%%%%%%%%%%%%%%%%%%%%%%%%%%%%%%%%%%%%%%%%%%%%%%%%%%%%%%%%%%%%%
\noindent{For a definition of the semiclassical Sobolev spaces $H^s_{\text{scl}}$ see \eqref{e:sobolev}.} {Here and below, when we write $\|v\|_{H_{\scl}^s}$ for some $v\in \mc{D}'$ {with} $v\notin H_{\scl}^s$, we define $\|v\|_{H_{\scl}^s}=\infty$.}

Before stating our next theorem, we recall that if $(M,g)$ has strictly negative sectional curvature, then it also has  Anosov geodesic flow~\cite{Anosov}.  Also, both Anosov geodesic flow~\cite{Kling} and non-positive sectional curvature imply that $(M,g)$ has no conjugate points.
%%%%%%%%%%%%%%%%%%%%%%%%%%%%%%%%%%%%%%%%%%%%%%%%%%%%%%%%%%%%%%%%%%%%%%%%%%%%%%%%
\begin{theorem}[{\cite[Theorems 3 and 4]{CG19dyn}}]\label{T:applications}
 Let $(M,g)$ be a smooth, compact Riemannian manifold of dimension $n$. Let $H\subset M$ be a closed embedded submanifold of codimension $k$. Suppose one of the following assumptions holds{:}
\begin{enumerate}[label=\textbf{\Alph*.},ref=\textbf{\Alph*}]
\item  \label{a1} $(M,g)$ has no conjugate points and $H$ has codimension $k>\frac{n +1}{2}$. \smallskip %C
\item  \label{a2}$(M,g)$ has no conjugate points and $H$ is a geodesic sphere. \smallskip %F
\item  \label{a3}$(M,g)$ is a surface with Anosov geodesic flow.  \smallskip %A
\item  \label{a6}$(M,g)$ is non-positively curved and has Anosov geodesic flow, and $H$ has codimension $k>1$.  \smallskip %E
\item  \label{a4}$(M,g)$ is non-positively curved and has Anosov geodesic flow, and $H$ is totally geodesic.  \smallskip %E
\item  \label{a5} $(M,g)$ has {Anosov geodesic flow} and $H$ is a subset of $M$ that lifts to a horosphere in the universal cover. %G
\end{enumerate}
Then, there exists $C>0$ so that for all $w\in C_c^\infty(H)$ the following holds. There is $h_0>0$ so that for $0<h<h_0$ and $u\in \mc{D}'(M)$
\begin{equation}
\label{e:subEst}
\Big|\int_Hwud\sigma_H\Big|\leq Ch^{\frac{1-k}{2}}\|w\|_{\infty}\Big(\frac{\|u\|_{\LM}}{\sqrt{\log h^{-1}}}+\frac{\sqrt{\log h^{-1}}}{h}\|(-h^2\Delta_g-1)u\|_{\Hl}\Big).
\end{equation}
\end{theorem}
%%%%%%%%%%%%%%%%%%%%%%%%%%%%%%%%%%%%%%%%%%%%%%%%%%%%%%%%%%%%%%%%%%%%%%%%%%%%%%%%

\begin{remark}
{Note that while $C>0$ in~\eqref{e:subEst} is independent of $w$, the choice of $h_0>0$ depends on high order derivatives of $w$. }
\end{remark}
To the authors' knowledge, the results in \cite{CG19dyn} improve and extend \emph{all} existing bounds on averages over submanifolds for eigenfunctions of the Laplacian, including those on $L^\infty$ norms (without additional assumptions on the eigenfunctions; see Remark~\ref{r:extra} for more detail on other types of assumptions).  Our estimates imply those of~\cite{CG17} and therefore give all previously known improvements of the form {\small $\int_H ud\sigma_H=o(h^{\frac{1-k}{2}})$}. Moreover, we are able to improve upon the results of~\cite{Wym18,Wym2,SXZ,Berard77,Bo16,Randol}.

%%%%%%%%%%%%%%%%%%%%%%%%%%%%%%%%%%%%%%%%%%%%%%%%%%%%%%%%%%%%%%%%%%%%%%%%%%%%%%%%
\subsubsection{Integrable geometries} \label{s:app with integrable}
%Not only do we prove Theorem~\ref{t:noConj2, but we study certain integrable situations on the sphere and prove the bound $o(\frac{h^{-\frac{1}{2}}}{\sqrt{\log h^{-1}}})$ for eigenfunctions away from the poles. For example, we find a rotationally symmetric metric on the sphere for which this bound holds (see Theorems~\ref{t:spherePend},~\ref{t:sphere} and Remark~\ref{r:sphere}). As far as the authors are aware, these are the first bounds with better than $\sqrt{\log h^{-1}}$ improvements in any geometry without extra assumptions on the eigenfunctions (see Remark~\ref{r:extra}). Moreover, they are the first quantitative bounds available in geometries without assumptions on the set of conjugate points.
%%%%%%%%%%%%%%%%%%%%%%%%%%%%%%%%%%%%%%%%%%%%%%%%%%%%%%%%%%%%%%%%%%%%%%%%%%%%%%%%
Next, we present a class of integrable geometries for which $\log h^{-1}$ improvements over the standard bounds {are} a consequence of Theorem \ref{t:coverToEstimate efxs} and its generalization, Theorem \ref{t:coverToEstimate}. We apply Theorem~\ref{t:coverToEstimate} to the case of Schr\"o\-dinger operators, {$-h^2\Delta_g+V$}, {acting on} spheres of revolution where the bicharacteristic flow is integrable. When $V=0$, these examples {give manifolds with} many conjugate points {where} we are able to obtain {quantitatively} improved $L^\infty$ bounds away from the poles of $S^2$.

{To state our results, we identify the surface of revolution $M$ with $[0,\pi]\times S^1$ endowed with the metric
$
g(r,\theta)=dr^2+\alpha(r)^2d\theta^2.
$
We then consider operators of the form 
$
P(h)=-h^2\Delta_g-V
$
with $V>0$. The Hamiltonian for this problem is then
$$
p(\theta,r,\xi_{\theta},\xi_r)=\xi_r^2+\tfrac{1}{\alpha(r)^2}\, {\xi_\theta^2}-V(r)
$$ 
and we assume that the map $r\mapsto \alpha(r)\sqrt{V(r)}$ has a single {critical point at $r=r_s$} which is a non-degenerate maximum. In order that $M$ be equivalent to a sphere, $\alpha(r)$  must satisfy $\alpha^{(2k)}(0)=0$ and $\alpha^{(2k)}(\pi)=0$ {for all non-negative integers $k$}.}

{Since $\{p, \xi_\theta\}=0$,  the pair $(M,p)$ yields  an integrable system on $T^*M$.   
Let $(\Theta,I)\in \mathbb T^2\times \re^2$ be action-angle coordinates so that
$
\TM = \bigsqcup_{I\in \re^2}\mathbb{T}_I
$
is the foliation by Liouville tori (possibly with some singular elements). That is, in the $(\Theta,I)$ coordinates    $p=p(I)$ and hence the {Hamiltonian} flow is given by 
$$ 
\varphi_{t}(\Theta,I)=(\Theta+t\partial_I p(I),I).
$$
There is a single singular torus corresponding to the closed Hamiltonian bicharacteristic $\gamma_s:=\{r=r_s\}$. In addition, we make the following assumption
\begin{enumerate}
\item The map $\{p=0\}\ni I \mapsto \partial_Ip(I)\in \mathbb{RP}^2$ is a diffeomorphism. When this is the case at $I_0$, we say $p$ is iso-energetically non-degenerate at $I_0$ on $\{p=0\}$.
\end{enumerate}
%%%%%%%%%%%%%%%%%%%%%%%%%%%%%%%%%%%%%%%%%%%%%%%%%%%%%%%%%%%%%%%%%%%%%%%%%%%%%%%
\begin{theorem}
\label{t:sphere}
Let $\alpha$ and $V$ satisfy the assumptions above. %$M=[0,2\pi]_\theta\times[0,\pi]_r$ and 
Then, for
\begin{equation}
    \label{e:sphereHamilton}
P=-h^2\Delta_g-V(r)+hQ
\end{equation}
with $Q\in \Psi^2(M)$ self-adjoint, and $K\subset [0,2\pi]\times (0,\pi)$ compact, there exists $C>0$ with the following properties. For all $L>0$ there exists $h_0>0$ so that for $0<h<h_0$, and $u \in \mathcal D'(M)$
$$
\|u\|_{L^\infty(K)}\leq Ch^{-\frac{1}{2}}\Bigg(\frac{\|u\|_{\LM}}{L\sqrt{\log h^{-1}}}+\frac{L\sqrt{\log h^{-1}}\|Pu\|_{\Hs{-\frac{1}{2}}}}{h}\Bigg).
$$
In particular, if $\|Pu\|_{\Hs{-\frac{1}{2}}}=o\Big(\frac{h\|u\|_{\LM}}{{{\log h^{-1}}}}\Big)$,
 then
$$
\|u\|_{L^\infty(K)}=o\Bigg(\frac{h^{-\frac{1}{2}}}{\sqrt{\log h^{-1}}}\|u\|_{\LM}\Bigg).
$$
\end{theorem}
%%%%%%%%%%%%%%%%%%%%%%%%%%%%%%%%%%%%%%%%%%%%%%%%%%%%%%%%%%%%%%%%%%%%%%%%%%%%%%%

}

{\begin{remark}
Note that we make no assumptions on $u$. In particular, $u$ need not be a joint eigenfunction of the quantum completely integrable system. Furthermore, the addition of the perturbation $hQ$ (for $Q$ general) destroys the quantum complete integrability of the operator.
\end{remark}}

\subsection{Logarithmic improvements for $L^p$-norms}\label{s:Lp}
 Since the work of Sogge \cite{So88} it has been known that 
 \[
 \|\phi_h\|_{_{L^p(M)}}={\small O(h^{-\delta(p,n)})}, \qquad \qquad 
 \delta(p,n)=
\begin{cases}
\tfrac{n-1}{2} -\tfrac{n}{p}&  \quad p \geq p_c,\\ 
\tfrac{n-1}{4} -\tfrac{n-1}{2p}& \quad 2\leq p \leq p_c,
\end{cases}
\]
where $ p_c=\tfrac{2(n+1)}{n-1}$. 
This bound is saturated on the sphere by zonal harmonics  when $p\geq p_c$ and by highest weight spherical harmonics (a.k.a Gaussian beams)  when $p\leq p_c$. (See e.g~\cite{Ta18a} for a description of extremizing quasimodes.)

It is then natural to look for quantitative improvements on this bound under different geometric assumptions. 
When $(M,g)$ has non-positive sectional curvature,  a bound  of the form
\[
 \|\phi_h\|_{_{L^p(M)}}= O\Big(\frac{h^{-\delta(p,n)}} {(\log h^{-1})^{\sigma(p,n)}}\Big)
 \]
 was proved by Hassell-Tacy \cite{HT15},  with {\small$\sigma(p,n)=\tfrac{1}{2}$}, for the case $p> p_c$. In the same setting, Blair-Sogge \cite{BS17,BS18} studied the {\small$2<p\leq p_c$} case and obtained a logarithmic improvement for some $\sigma(p,n)$ that is smaller than $\tfrac{1}{2}$. 
 
 An application of Theorem \ref{t:coverToEstimate efxs} gives $(\log{h^{-1}})^{\frac{1}{2}}$ improvement when $p > p_c$ under very weak assumptions on the set of conjugate points of $(M,g)$. Indeed, given $x\in M$, $r>0$, and $t>0$, we continue to write 
$\mc{C}_x^{r,t}$ for the set of points defined in \eqref{e:conjugate set}. Note that if $r_t \to 0^+$ as $|t|\to \infty$, then saying that $y \in \mc{C}_x^{r_t,t}$ for $t$ large indicates that $y$ behaves like  point that is maximally conjugate to $x$.

%%%%%%%%%%%%%%%%%%%%%%%%%%%%%%%%%%%%%%%%%%%%%%%%%%%%%%%%%%%%%%%%%%%%%%%%%%%%%%%%
\begin{theorem}[\cite{CG19Lp}]
\label{t:Lp}
{Let $p>p_c$.} Let $V \subset M$ and assume that there exist $t_0>0$ and $a>0$  so that 
$$
\inf_{x,y\in V}d\big(y, \mc{C}_{x}^{r_t,t}\big)\geq r_t,\qquad\text{ for } t\geq t_0,
$$
with $r_t=\frac{1}{a}e^{-at}.$ 
Then, there exist $C>0$ and $h_0>0$ so that for $0<h<h_0$, and $\phi_h$ satisfying \eqref{e:eigsYes},
$$
\|\phi_h\|_{L^p(V)}\leq C  \frac{h^{-\delta(p,n)}}{\sqrt{\log h^{-1}}}.
$$
\end{theorem}

One should think of the assumption in Theorem~\ref{t:Lp} as ruling out maximal conjugacy of the points $x$ and $y$ uniformly up to time $\infty$.

\begin{remark}
There are estimates in terms of the dynamical properties of covers by tubes similar to Theorem~\ref{t:coverToEstimate efxs} for each of the bounds in
Theorems~\ref{t:noConj2},~\ref{T:applications}, and~\ref{t:Lp}. In particular, these estimates do \emph{not} require global {geometric} assumptions on $(M,g)$, instead only using dynamical properties near $S^*_xM$ or $\SNH$.
\end{remark}
%%%%%%%%%%%%%%%%%%%%%%%%%%%%%%%%%%%%%%%%%%%%%%%%%%%%%%%%%%%%%%%%%%%%%%%%%%%%%%%%

\subsection{Logarithmic improvements for pointwise Weyl Laws}\label{s:Weyl}
 Let $\{h_j^{-2}\}_j$ be the eigenvalues of $(M,g)$.  It is well known that 
{\small ${\small \#\{j:\; h_j^{-1} \leq h^{-1}\}=\tfrac{\vol(B^{n})\vol(M)}{(2\pi)^n}h^{-n} + E(h)}$} with {\small $E(h)=O(h^{1-n})$}. Indeed, this result is the integrated version of the more refined statement proved by H\"ormander  in \cite{Ho68} which says that for all $x \in M$
\begin{equation}\label{e:weyl error defn}
    \sum_{h_j^{-1} \leq h^{-1}}|\phi_{h_j}(x)|^2=\frac{\vol(B^{n})}{(2\pi)^n}h^{-n} + E(h,x),
\end{equation}  
 with {\small $E(h,x)=O(h^{1-n})$} uniformly for $x \in M$. When the set of looping directions {over} $x$ has measure zero \cite{SZ02} proved that $E(h, x)=o(h^{1-n})$. Also, Duistermaat-Guillemin \cite{DG75} proved an integrated version of this result by showing that $E(h)=o(h^{1-n})$ if the set of closed geodesics in $M$ has measure zero. In terms of quantitative improvements, \cite{Berard77,Bo16} prove that {\small $E(h,x)={\small O(h^{1-n}/{\small \log h^{-1}})}$} if $(M,g)$ has no conjugate points. As before, another application of geodesic beam techniques is that $\log h^{-1}$ improvements can be obtained under weaker assumptions than having no conjugate points.

\begin{theorem}[{\cite{CG19Weyl}}]
\label{t:weyl} 
%\red{\cs Option 1 \cs
% Let $x\in M$, $0<\delta<\frac{1}{2}$. There exist positive constants $h_0=h_0(M,g,\delta)$, $\tau_0=\tau_0(M,g)$, $R_0=R_0(M,g)$, and $C_{n}$ depending only on $n$, so that for all $0<\tau <\tau_0$ and $0<h<h_0$ the following holds.
%
%Let ${8}h^\delta\leq R(h){<R_0}$, and  $\{\Lambda_{\rho_j}^\tau(R(h))\}_{j=1}^{N_h}$ be a {$(\mathfrak D,\tau, R(h))$-good cover of $S^*_xM$ for some $\mathfrak{D}>0$}. Let $0\leq\alpha< 1-2{\limsup_{h\to 0}\frac{\log R(h)}{\log h}}$ and suppose there exists a partition of $\{1,\dots, N_h\}$ into $\mc{B}$ and  $\mc{G}$ such that there exist ${1\leq T(h)}\leq {2} \alpha T_e(h)$ and   $t(h)<\inj(M)$ so that
%$$
%\bigcup_{j\in \mc{G}}\Lambda_{\rho_j}^\tau(R(h))\;\;\text{ is }\;\;[t(h),T(h)]\text{ non-self looping}.
%$$
%% 
%Then, with $E(h,x)$ as in \eqref{e:weyl error defn},
%\begin{align*}
%E(h,x)
%\leq  C_{n}{\mathfrak{D}}\tau^{-1}h^{{1-n}}R(h)^{{n-1}}
%\Bigg(|\mc{B}|\frac{T}{t}+ |\mc{G}|\frac{ t  }{T}\Bigg)
%\end{align*}}
Let $V \subset M$ and assume that there exist $t_0>0$ and $a>0$  so that 
$$
\inf_{x\in V}d\big(x, \mc{C}_{x}^{r_t,t}\big)\geq r_t,\qquad\text{ for } t\geq t_0,
$$
with $r_t=\frac{1}{a}e^{-at}.$ 
Then, there exist $C>0$ and $h_0>0$ so that for $0<h<h_0$ and $E(h,x)$ as in \eqref{e:weyl error defn},
$$
\sup_{x\in V}E(h,x)\leq \frac{Ch^{1-n}}{\log h^{-1}}.
$$
\end{theorem}

We remark that there are generalizations of this result to Kuznecov sums estimates, where evaluation at $x$ is replaced by an integral average over a submanifold $H$ (see \cite{Zel} for the first results in this direction). In addition, in the same way that Theorem~\ref{t:coverToEstimate efxs} can be used to obtain quantitative improvements in $L^\infty$ bounds in concrete geometric settings, the {dynamical version of the} estimate in Theorem~\ref{t:weyl} can be used to obtain improved remainder estimates for pointwise Weyl laws. {We show, for example, that all non-trivial product manifolds satisfy the assumptions of Theorem~\ref{t:weyl} at \emph{every} point in \S~\ref{s:examples}.}

{\subsection{Examples}\label{s:examples} We now record some examples to which our theorems apply. We refer the reader to~\cite{CG19dyn} for many more examples. First, note that Theorem~\ref{t:noConj2} applies when $M$ is a manifold without conjugate points. The following examples may (and typically do) have conjugate points.}

{
\subsubsection{Product manifolds}
\begin{lemma}
Let $(M_i,g_i)$, $i=1,2$, be two compact Riemannian manifolds. Let $M=M_1\times M_2$ endowed with the product metric  $g=g_1\oplus g_2$. Then, $\mc{C}_x^{r,t}=\emptyset$ for all $x\in M$, $|t|> 0$, and $0<r<t$.
\end{lemma}
\begin{proof}
Let $x=(x_1,x_2)\in M$ and $\gamma(t)$ be a unit speed geodesic on $M$ with $\gamma(0)=0$. Then, there are unit speed geodesics $\gamma_1$ and $\gamma_2$ in $M_1$ and $M_2$ respectively such that  $\gamma_1(0)=x_1$, $\gamma_2(0)=x_2$, and there exists $\theta_0 \in \mathbb{R}$ such that $$\gamma(t)=(\gamma_1(t\cos \theta_0 ),\gamma_2(t\sin \theta_0 ))\in M_1\times M_2.$$ Moreover, for every $\theta\in \re$, the curve $\gamma_\theta:=(\gamma_1(t\cos \theta ),\gamma_2(t\sin\theta ))$ is a unit speed geodesic. In particular, one perpendicular Jacobi field along $\gamma=\gamma_{\theta_0}$ is given by 
$$
J(t)=\partial_\theta \gamma_\theta\big|_{\theta=\theta_0}= t(-\sin \theta_0 \dot \gamma_1(t\cos \theta_0 ),  \cos \theta_0 \dot \gamma_2(t\sin \theta_0)).
$$
Thus, $\|J(t)\|=t,$
and hence $J$ vanishes only at $t=0$. In particular, since there exists a Jacobi field vanishing only at $t=0$, $\mc{C}^{r,t}_x=\emptyset$ for all  $0<r<|t|$. \end{proof}
}

{We point out that although $\mc{C}^{r,t}_x$ is empty for $0<r<|t|$, $M$ may, and often does, have self conjugate points. For example, this is the case if $M_1=S^{n_1}$ for $n_1\geq 2$. 
}

{
\begin{corollary}
Let $(M_i,g_i)$, $i=1,2$, be two compact Riemannian manifolds of dimension $n_i>0$. Let $M=M_1\times M_2$ endowed with the metric $g= g_1\oplus g_2$. Then, there is $C>0$ such that for all $x\in M$ and $u\in \mc{D}'(M)$,
$$
|u(x)|\leq Ch^{\frac{1-{(n_1+n_2)}}{2}}\Big(\frac{\|u\|_{L^2(M)}}{\sqrt{\log h^{-1}}}+\frac{\sqrt{\log h^{-1}}}{h}\big\|(-h^2\Delta_g-1)u\big\|_{\Hs{\frac{n_1+n_2-3}{2}}}\Big).
$$
\end{corollary}
\subsubsection{The triaxial ellipsoid}\label{s:triaxial ellipsoid}
We consider the \emph{triaxial ellipsoid}
$$
M:=\{ x\in \re^3:\; a^2x^2+b^2y^2+c^2z^2=1\}
$$
with $0<a<b<c$. It is well known that the four umbillic points (i.e. points at which the normal curvatures are equal in all directions)
 on $M$ are maximally self-conjugate. In fact, for an umbillic point $x_0$, there is $T>0$ such that every geodesic through $x_0$ returns to $x_0$ at time $T$. Nevertheless, Theorem~\ref{t:coverToEstimate efxs} and its generalization, Theorem~\ref{t:coverToEstimate}, are useful at these points. The reason for this is the presence of {a hyperbolic closed geodesic through $x_0$} to which every other geodesic through $x_0$ exponentially converges forward and backward in time (up to reversal of the parametrization).  In particular, letting $(x_0,\xi_+)$ and $(x_0,\xi_-)$ be the initial points of the {hyperbolic} geodesic, we have that the stable direction for $\xi_+$ is given by $T_{\xi_+}S^*_xM$ and the unstable direction for $\xi_-$ is given by $T_{\xi_-}S^*_xM$~\cite[Theorem 3.5.16]{Kling95}. Thus, for each $\delta>0$ there is $C>0$ such that {if $d(\xi,\xi^{\pm})>\delta$}, then in {for all} $\mp t>0$ one has that  
 $$
\big \|d\varphi_t\big|_{_{\!T_{\xi}S^*_{x_0}M}}\big\|\leq Ce^{\pm Ct}.
 $$
 This type of exponential convergence can be used (see~\cite{GT18a},~\cite[Lemmas 3.1-3.2]{CG19dyn}) to generate covers and obtain
 $$
 |u(x_0)|\leq Ch^{-\frac{1}{2}}\Bigg(\frac{\|u\|_{L^2(M)}}{\sqrt{\log h^{-1}}}+\frac{\sqrt{\log h^{-1}}}{h}\big\|(-h^2\Delta_g-1)u\big\|_{ \text{\raisebox{.2cm}{${_{\!H_{scl}^{-\frac{1}{2}}}}$}}(M)}
\Bigg). $$
}

\subsubsection{The spherical pendulum}
{One example to which Theorem~\ref{t:sphere} applies is that of $S^2=\{x\in \re^3: |x|=1\}$ the standard sphere equipped with the round metric, {$g$}, and $V\in C^\infty(S^2)$ given by $V(x_1,x_2,x_3)=2x_3$. The quantum spherical pendulum is then the operator
$$
P=-h^2\Delta_g+V.
$$
}
{
Identifying the sphere with $M= [0,\pi]_r\times [0,2\pi]_\theta$. The Hamiltonian is given by
$$
p(\theta, r, \xi_{\theta}, \xi_r)=\xi_r^2+\tfrac{1}{\sin^2r}\,{\xi_\theta^2}+2\cos r -E,
$$
 {with $E \in \re$. This Hamiltonian describes the movement of a pendulum of mass $1$ moving without friction on the surface of a sphere of radius $1$.}}
 
{
By~\cite{Horozov} for $E\geq \frac{14}{\sqrt{17}}$, $p$ is iso-energetically non-degenerate for all $I_0$ on $\{p=0\}$. It is easy to check by explicit computations that  $E-2\cos r>0$ for $E>2$ and $r \mapsto \sin r\sqrt{E-2\cos r}$ has a single non-degenerate maximum on $[0,\pi]$.  Therefore, taking $E=E_0\geq\frac{14}{\sqrt{17}}$ and $Q=h^{-1}(E_0-E_h)$ in Theorem~\ref{t:sphere} yields the following Corollary~\ref{t:spherePend}.
\begin{corollary}
\label{t:spherePend}
Let $B>0$, $E_0\geq\frac{14}{\sqrt{17}}$ {and $\delta>0$}. There exists $C>0$ such that for all $L>0$ there exists $h_0>0$ {so that the following holds.  For all $u\in \mc{D}'(S^2)$, $0<h<h_0$ and {$E_h \in (E_0-Bh,E_0+Bh)$},} 
$$
\|u\|_{L^\infty(|x_3|<1-\delta)}\leq Ch^{-\frac{1}{2}}\Bigg(\frac{\|u\|_{_{L^2(S^2)}}}{L\sqrt{\log h^{-1}}}+\frac{L\sqrt{\log h^{-1}}\|(P-E_h)u\|_{\text{\raisebox{.2cm}{${_{\!H_{scl}^{-\frac{1}{2}}\!\!(S^2)}}$}}}}{h}\Bigg).
$$
{In particular, if 
$\|u\|_{L^2(S^2)}=1$ and $ Pu=o\big(h/\log h^{-1}\big)_{L^2}$
then
\begin{equation}\label{e:pendulum bound}
\|u\|_{L^\infty(|x_3|<1-\delta)}=o\Big(\frac{h^{-\frac{1}{2}}}{\sqrt{\log h^{-1}}}
\Big).
\end{equation}}
\end{corollary}}

{\noindent Note that if we define $\tilde{g}=g/{\small \sqrt{E_0-2x_3}}$ with $E_0\geq \frac{14}{\sqrt{17}}$, then Theorem~\ref{t:sphere} shows that {the eigenfunctions $\phi_h$ for $(-h^2\Delta_{\tilde{g}}-1)\phi_h=0$ satisfy the bound
\[
\|\phi_h\|_{L^\infty(|x_3|<1-\delta)}=o\Big(\frac{h^{-\frac{1}{2}}}{\sqrt{\log h^{-1}}}\Big).
\] 
for any $\delta>0$}.   }
%\red{\begin{remark}
%The example is given by a Schr\"odinger operator on $(M,g)=(S^2,g_{\text{round}})$. In particular, the potential is given from the embedding in $\re^3$ as
%$$S^2:=\{(x_1,x_2,x_3)\in \re^3:\; x_1^2+x_2^2+x_3^2=1\}\qquad V(x)=2x_3$$
%and the relevant operator is
%$$P=-h^2\Delta_g+V-E+hQ.$$
%
%\end{remark}}

\subsection{Relations with previous dynamical conditions on pointwise estimates}
\label{s:relateOld}

{In this section, we recall the previous dynamical conditions guaranteeing improved pointwise estimates~\cite{Saf88,VaSa:92,SZ02,SoggeTothZelditch,SZ16I,SZ16I, Gdefect}.
We first define the \emph{loop set at $x$} by
$$
\mc{L}_x:=\{\rho\in S^*_xM \mid \,\exists t\in \mathbb{R} \text{ s.t.}\; \varphi_t(\rho)\in S^*_xM\},
$$ 
and recall that a point $x$ is said to be \emph{non-self focal} if $\vol_{S^*_xM}(\mc{L}_x)=0$. It is proved in~\cite{Saf88,SZ02} that if $x$ is non-self focal, then 
\begin{equation}
\label{e:smallo}
|\phi_h(x)|=o(h^{\frac{1-n}{2}}).
\end{equation}

Next, define
$T_{\pm}:S^*_xM\to [0,\infty]$ by
$
T_\pm(\rho):=\pm \inf\{\pm t>0\mid \varphi_t(\rho)\in S^*_xM\}
$
and $\Phi_\pm:T_\pm^{-1}(0,\infty)\to S^*_xM$ by 
$$
\Phi_\pm(\rho)=\varphi_{_{T_\pm(\rho)}}(\rho).
$$
We then define $\mc{R}_x$ as the recurrent set for $\Phi$. In~\cite{VaSa:92,SoggeTothZelditch,Gdefect}, it is shown that if $\vol_{S^*_xM}(\mc{R}_x)=0$, then~\eqref{e:smallo} continues to hold. In that case $x$ is called \emph{non-recurrent}. Finally, in~\cite{SZ16I,VaSa:92,Gdefect} it is shown that there need only be no invariant $L^2(\vol_{S^*_xM})$ function for~\eqref{e:smallo} to hold.

\begin{definition}
\label{d:non-rec}
For the purposes of this section, we will say that a point $x$ is \emph{$(t_0,T(h))$ non-looping via covers} if there is a $(\tau, R(h))$ cover for $S^*_xM$, $\{\Lambda_{\rho_j}^\tau(R(h))\}_{j=1}^{N_h}$, and $\mc{B}\sqcup \mc{G}=\{1,\dots N_h\}$, such that 
$$
\bigcup_{j\in \mc{G}}\Lambda_{\rho_j}^\tau(R(h)) \text{ is }[t_0,T(h)]\text{ non-self looping} 
\qquad
\text{and} 
\qquad
|\mc{B}|\leq \frac{R(h)^{1-n}}{T(h)}.
$$
(See also~\cite[Definition 2.1]{CG19Weyl}.) We will say that  \emph{$x$ is $T(h)$ non-recurrent via covers} if there are sets of indices $\mc{G}_\ell\subset\{1,\dots N_h\}$ and pairs of times $(t_\ell, T_\ell)$ such that $\{1,\dots N_h\}=\cup_\ell \mc{G}_\ell$ and 

 $${\bigcup_{j\in \mc {G}_\ell}\Lambda_{\rho_j}^\tau(R(h))} \text{ isn} \;[t_\ell,T_\ell]\;\text{ non-self looping} 
    \qquad
    \text{and}
    \qquad \sum_{\ell}\frac{|\mc{G}_\ell|^{1/2}t_\ell^{1/2}}{T_\ell^{1/2}}\leq \frac{R(h)^{\frac{1-n}{2}}}{T(h)^{1/2}}.
    $$
    (See also~\cite[Definition 2.2]{CG19Weyl}.) 
    \end{definition}
    
    First of all, we point out that $x$ being $T(h)$ non-looping via covers implies that it is $T(h)$ non-recurrent via covers and that Theorem~\ref{t:coverToEstimate efxs} states that if $x$ is $T(h)$ non-recurrent via covers for some $T(h)\ll T_{e}(h)$,  then there is $C>0$ such that 
    \begin{equation}
    \label{e:Tnonrec}
    |\phi_h(x)|\leq \frac{Ch^{\frac{1-n}{2}}}{T(h)^{1/2}}.
    \end{equation}

%Note that the estimate~\eqref{e:smallo} under the condition that $x$ is non-recurrent follows from~\eqref{e:Tnonrec} by allowing $T(h)$ to grow arbitrarily slowly with $h$ (see~\cite{GJEDP}).

In order to relate these two concepts to the concept of a non-self focal point and a non-recurrent point respectively, we prove the following two Lemmas in Appendix~\ref{a:oldrelations}

\begin{lemma}
\label{l:loopsAreFun}
Suppose that $x$ is non-self focal. Then there are $t_0>0$ and $T:(0,1)\to (0,\infty)$ such that $\lim_{h\downarrow 0}T(h)=\infty$ and $x$ is $(t_0,T(h))$ non-looping via covers.
\end{lemma}

\begin{lemma}
\label{l:dejaVu}
Suppose that $x$ is non-recurrent. Then there is $T:(0,1)\to (0,\infty)$ such that $\lim_{h\downarrow 0} T(h)=\infty$ and $x$ is $T(h)$ non-recurrent via covers.
\end{lemma}
In particular, lemmas~\ref{l:loopsAreFun} and~\ref{l:dejaVu} recover the fact that $x$ being non-recurrent implies~\eqref{e:smallo}.

}

%%%%%%%%%%%%%%%%%%%%%%%%%%%%%%%%%%%%%%%%%%%%%%%%%%%%%%%%%%%%%%%%%%%%%%%%%%%%%%%%
\subsection{Outline of the paper}
In \S\ref{s:general} we present Theorems \ref{t:porcupine} and \ref{t:coverToEstimate} which are the generalization of Theorems \ref{t:porcupine efxs} and \ref{t:coverToEstimate efxs} to quasimodes of general pseudo-differential operators $P$.
In \S\ref{s:loc}, we perform the analysis of quasimodes for $P$ and in particular prove Theorem~\ref{t:porcupine}. In \S\ref{s:Egorov} we give the proof of Theorem~\ref{t:coverToEstimate}. 
In \S\ref{s:spheres} we construct non-self looping covers on spheres of revolution and {prove Corollary~\ref{t:spherePend}.} {Finally, in \S\ref{S:change of hamiltonian}, we prove that the Hamiltonian flow for $|\xi|_g^2-1$ can be replaced by that for $|\xi|_g-1$.} 
{In Appendix~\ref{s:prelim} we present an index of notation and background on semiclassical analysis. }
%%%%%%%%%%%%%%%%%%%%%%%%%%%%%%%%%%%%%%%%%%%%%%%%%%%%%%%%%%%%%%%%%%%%%%%%%%%%%%%%

\ \\
\noindent {\sc Acknowledgements.} Thanks to Pat Eberlein, John Toth, Andras Vasy, and Maciej Zworski for many helpful conversations and comments on the manuscript. {Thanks also to the anonymous referees for many suggestions which improved the exposition.}
J.G. is grateful to the National Science Foundation for support under the Mathematical Sciences Postdoctoral Research Fellowship  DMS-1502661. {Y.C. is grateful to the Alfred P. Sloan Foundation. }

%%%%%%%%%%%%%%%%%%%%%%%%%%%%%%%%%%%%%%%%%%%%%%%%%%%%%%%%%%%%%%%%%%%%%%%%%%%%%%%%
%%%%%%%%%%%%%%%%%%%%%%%%%%%%%%%%%%%%%%%%%%%%%%%%%%%%%%%%%%%%%%%%%%%%%%%%%%%%%%%%
\section{General results: Bicharacteristic beams}\label{s:general}

%%%%%%%%%%%%%%%%%%%%%%%%%%%%%%%%%%%%%%%%%%%%%%%%%%%%%%%%%%%%%%%%%%%%%%%%%%%%%%%%
%%%%%%%%%%%%%%%%%%%%%%%%%%%%%%%%%%%%%%%%%%%%%%%%%%%%%%%%%%%%%%%%%%%%%%%%%%%%%%%%

Our main estimate gives control on eigenfunction averages in terms of microlocal data. {The ideas leading to the estimate build on the tools first constructed in~\cite{Gdefect} for sup-norms and generalized for use on submanifolds in~\cite{CG17}.}  

Since it entails little extra difficulty, we work in the general setup of semiclassical pseudodifferential operators (see e.g.~\cite{EZB} or ~\cite[Appendix E]{ZwScat} for a treatment of semiclassical analysis, see \S\ref{s:SemiNote} for a brief description of notation). Indeed, instead of only working with Laplace eigenfunctions, all our results can be proved for quasimodes of  a pseudodifferential operator of any order that has real, classically elliptic symbol. We now introduce the necessary objects to state this estimate.

Let $H \subset M$ be a submanifold. For $p\in S^m(\TM )$ define
\begin{equation}\label{e:SigH}
 {\SigH}=\{p=0\} \cap N^*\!H, 
\end{equation} 
where $N^*\!H$ is the conormal bundle to $H$
and {consider the Hamiltonian flow}
\begin{equation}\label{e:varphi}
\varphi_t:=\exp(tH_p).
\end{equation}
{Here, and in what follows, $H_p$ is the Hamiltonian {vector field} {generated by $p$}.}
In practice, we will prove our main result with $H$ replaced by a family of submanifolds $\{H_h\}_h$ such that for all $\alpha$ multiindex  there exists {$\KR_{_{\alpha}}>0$} such that for all $h>0$
\begin{equation}
\label{e:curvature}
|\partial_x^\alpha R_{_{H_h}}|+|\partial_x^\alpha \Pi_{_{H_h}}|\leq \KR_{_{\alpha}}
\end{equation}
where $R_{_{H_h}}$ and $\Pi_{_{H_h}}$ denote the sectional curvature and the second fundamental form of  $H_h$. Next, we assume that there is $\e>0$ so that for all $h>0$, the map $(-\e,\e)\times {\SigH}\to M$,
\begin{equation}
    \label{e:nonSelf}
    (t,\rho)\mapsto\pi({\varphi_t(\rho)})\;\;\;\text{ is a diffeomorphism}.
\end{equation} 
 We will say that a family of submanifolds $\{H_h\}_{h}$ is \emph{regular} if it satisfies~\eqref{e:curvature} and~\eqref{e:nonSelf}.  In addition, we will prove uniform statements in a shrinking neighborhood of $H_h$. In particular, {we prove stimates on $\tilde{H}_h$ where} $\tilde{H}_h$ is another family of submanifolds such that
\begin{equation}
\label{e:conormalClose}
{\sup_{\rho \in \Sigma_{_{H_h,p}}} d(\rho ,\Sigma_{_{\tilde{H}_h,p}})}\leq  h^\delta, \qquad \qquad |\partial_x^\alpha R_{_{\tilde H_h}}|+|\partial_x^\alpha \Pi_{_{\tilde H_h}}|\leq 2\KR_{_{\alpha}}
\end{equation}
for all $h>0$.
 Note that when $H_h$ is a family of points, the curvature bounds  become trivial, {and so in place of ~\eqref{e:conormalClose} we work with}  $d(x_h,\tilde{x}_h)<h^\delta$ 
and we may take ${\KR_{_{0}}}$ to be arbitrarily close to $0$. {It will often happen that the constants involved in our estimates depend on $\{H_h\}$ only through finitely many of the $\KR_{_{\alpha}}$ constants.} %When this happens we will write 
%\begin{equation}\label{e:KR}
%C=C(\KR_{_{\!*}}).
%\end{equation}}

For $p\in S^m(\TM )$, we say that $p$ is \emph{classically elliptic} if there exists {$K_p>0$} so that 
\begin{equation}\label{e:classellipt}
|p(x,\xi)|\geq |\xi|^m/{K_p},\qquad |\xi|\geq {K_p},\quad x\in M.
\end{equation}
In addition, for $p\in S^\infty(\TM ;\R)$, we say that a submanifold $H\subset M$ of codimension $k$ is \emph{conormally transverse for $p$} if given $f_1,\dots, f_{k}\in C_c^\infty(M;\R)$ {locally defining $H$ i.e.} with
\[
H= \bigcap_{i=1}^k\{f_i=0\}\quad\text{and}\quad \{df_i\}\text{ linearly independent on }H,
\] 
we have
\begin{equation}
\label{e:transverse} N^*\!H\; \subset \; \{p\neq0\} \cup \bigcup_{i=1}^k\{ H_pf_i \neq 0\},
\end{equation}
where $H_p$ is the Hamiltonian vector field associated to $p$, {and $N^*\!H$ is the set of conormal directions to $H$.} {Here, we interpret $f_i$ as a function on the cotangent bundle by pulling it back through the canonical projection map.}
In addition, let $r_H:M\to \re$ be the geodesic distance to $H$;  $r_{_{\!H}}(x)=d(x,H)$. Then, define $|H_pr_{_{\!H}}|:\SigH\to \re$ by 
\begin{equation}\label{e:HprH}
|H_pr_{_{\!H}}|(\rho):=\lim_{t\to 0}|H_pr_{_{\!H}}(\varphi_t(\rho))|.
\end{equation}
A family of submanifolds $\{H_h\}_{h}$ is said to be \emph{uniformly conormally transverse for $p$} if $H_h$ is conormally transverse {for $p$} for all $h$ and there exists {$\FR>0$} so that for all $h>0$
\begin{equation}\label{e:FR}
\inf_{\rho \in \SigH} |H_pr_{_{\!H_{\!h}}}|(\rho) \geq \FR.
\end{equation}
{Note that when $p(x, \xi)=|\xi|^2_{g(x)}-1$ then $\SigH=\SNH$ and $|H_pr_{_{\!H}}|(\rho)=2$ for all $\rho \in \SNH$.}

{Let $\{H_h\}_h{\subset M}$ be a regular and uniformly conormally transverse family of submanifolds}. Then, we may fix a family of regular hypersurfaces {depending on $h$,} $\sigFlow\subset \TM$ such that
\begin{equation}
\label{e:Hsig}
\sigFlow \text{ is uniformly transverse to }H_p\;\text{ with }\Sigma_{_{\!H_h, p}}\subset \sigFlow 
\end{equation}
and so that with $\Psi:\re\times {\TM}\to \TM $ defined by $\Psi(t,q)=\varphi_t(q)$, there is $0<\Tinj \leq 1$ (independent of $h$) so that 
\begin{equation}\label{e:Tinj}
%\Tinj:=\sup\{{\tau \leq 1}:\; \Psi|_{(-\tau,\tau)\times\sigFlow }\text{ is injective}\}.
 \Psi|_{(-\Tinj,\Tinj)\times\sigFlow }\text{ \;\;is injective}
\end{equation}
for all $h>0$.

\begin{remark}
{Working with a family $\{\tilde{H}_h\}_h$, and obtaining  uniform estimates for it, is needed in Theorem~\ref{t:porcupine efxs}. In this case, $H_h=\{x\}$ for every $h$ and $\tilde{H}_h$ is a point $\tilde{x}_h\in B(x,h^\delta)$.} Moreover, it is often useful to allow $H_h$ itself to vary with $h$ (see e.g.~\cite{CG19Lp}). Note that any $h$-independent submanifold $H\subset M$ that is conormally transverse is automatically regular and uniformly conormally transverse. {While in some applications it is useful to have $h$-dependent submanifolds $H_h$, as well as uniform estimates in a neighborhood of $H_h$, the reader may wish to ignore the dependence of $H_h$ on $h$ as well as letting $\tilde{H}=H$ for simplicity of reading.}
\end{remark}

Given $A \subset \TM $ define
$$\Lambda_{_{\!A}}^\tau:=\bigcup_{|t|\leq \tau}\varphi_t(A).$$
For $R>0$ and $A\subset \SigH$ we define
\begin{equation}\label{e:tube}
{\Lambda_{_{\!A}}^\tau(r):=\Lambda_{A_{_R}}^{\tau+r},\qquad A_{_r}:=\{\rho\in \sigFlow :d(\rho,A)<r\}.}
\end{equation}
where $d$ denotes the distance induced by the Sasaki metric on $T^*M$ (see e.g. \cite[Chapter 9]{BlairSasaki} for an explanation of the Sasaki metric). {In particular, the tube 
\begin{equation}
\label{e:tube2}
\Lambda_{\rho}^\tau(r):=\bigcup_{|t|\leq \tau +r}\varphi_t\big(\mc{L}_h\cap B(\rho,r)\big).
\end{equation}}
%%%%%%%%%%%%%%%%%%%%%%%%%%%%%%%%%%%%%%%%%%%%%%%%%%%%%%%%%%%%%%%%%%%%%%%%%%%%%%%%
\begin{figure}[htbp]
 \begin{center}
\includegraphics[height=5cm]{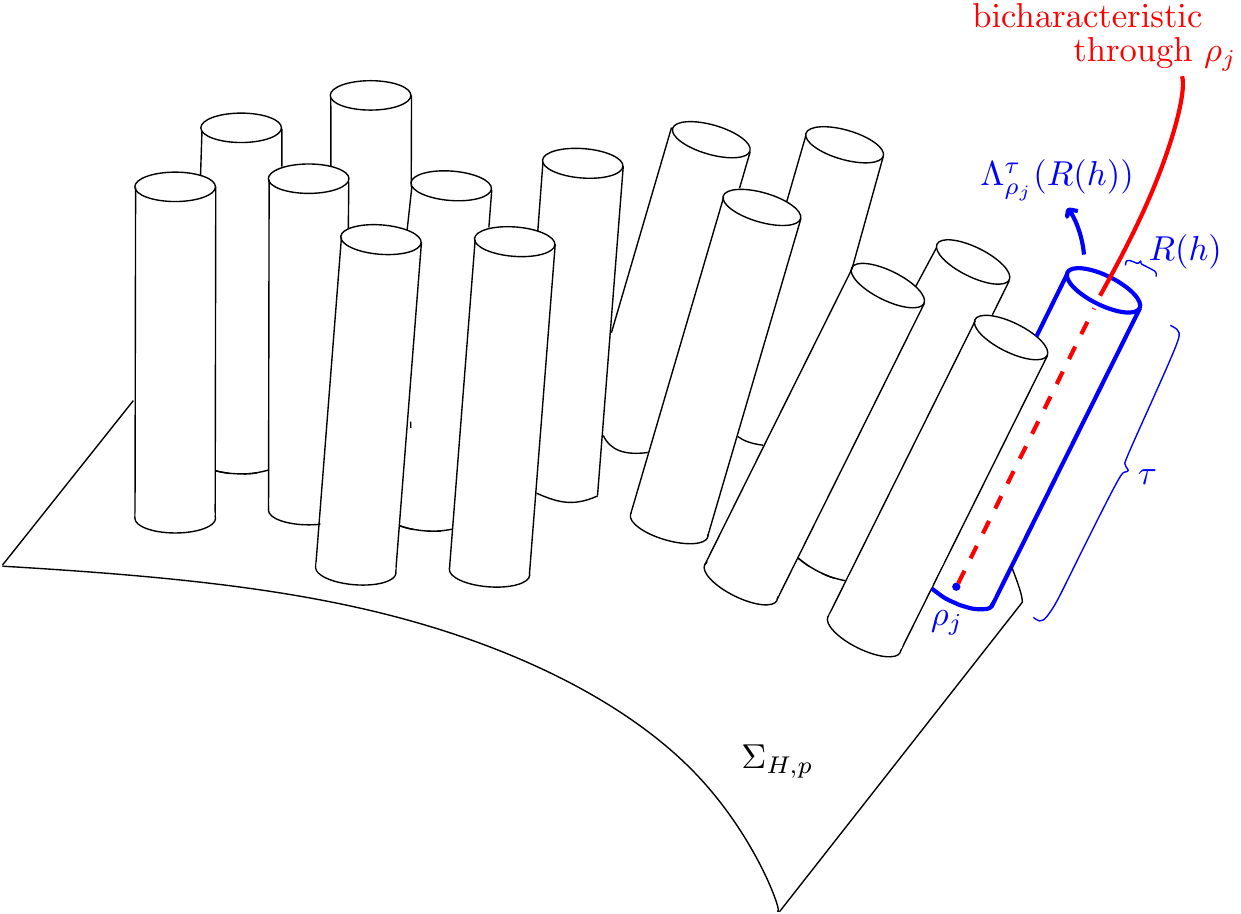}
\end{center}
\caption{\label{f:1} The tubes $\Tj$ through $\SigH$.}
\end{figure}
%%%%%%%%%%%%%%%%%%%%%%%%%%%%%%%%%%%%%%%%%%%%%%%%%%%%%%%%%%%%%%%%%%%%%%%%%%%%%%%% 

%%%%%%%%%%%%%%%%%%%%%%%%%%
\begin{definition} \label{d: cover} Let $A\subset \SigH$, $r>0$, and  $\{\rho_j(r)\}_{j=1}^{N_r} \subset A$. We say that the collection of tubes $\{\Lambda_{\rho_j}^\tau(r)\}_{j=1}^{N_h}$ is a \emph{$(\tau, r)$-cover} of a set $A\subset  \SigH$ provided
 $$\Lambda_A^\tau(\tfrac{1}{2}r) \subset\bigcup_{j=1}^{N_r}\Lambda_{\rho_j}^{\tau}(r).$$
In addition, {for $0\leq \delta\leq \frac{1}{2}$ and} $R(h)\geq {8}h^\delta$, we say that a collection $\{\chi_j\}_{j=1}^{N_h}\subset S_\delta(\TM;[0,1])$ is a \emph{$\delta$-partition for $A$   
associated to the $(\tau, R(h))$-cover}
if  $\{\chi_j\}_{j=1}^{N_h}$ is bounded in $S_\delta$ and
\begin{enumerate}
    \item $\supp \chi_j \subset {\Lambda_{\rho_j}^\tau(R(h))}$,\smallskip
    \item $\sum_{j=1}^{N_h}\chi_j\geq 1$ \;on\;  {$\Lambda_{A}^{{\tau/2}}(\tfrac{1}{2}R(h)).$}
\end{enumerate}
\end{definition}
%%%%%%%%%%%%%%%%%%%%%%%%%%

The main estimate is the following. 
%%%%%%%%%%%%%%%%%%%%%%%%%%%%%%%%%%%%%%%%%%%%%%%%%%%%%%%%%%%%%%%%%%%%%%%%%%%%%%%
\begin{theorem}
\label{t:porcupine} 
Let $P\in \Psi^m(M)$ have real, classically elliptic symbol $p\in S^m(\TM;{\R})$. 
{Let $\{H_h\}_h \subset M$ be a {regular} family of submanifolds of codimension $k$ that is uniformly conormally transverse for $p$.}
There exist
$$\tau_0=\tau_0(M,p,\Tinj, \FR ,{\{H_h\}_h}){>0},\qquad R_0=R_0(M,p,{k}, {\KR_{_{0}}}, {\Tinj}, \FR ){>0},$$
 $C_{n,k}>0$ depending only on $(n,k)$, {and $C_0>0$ depending only on $(M,p)$,} so that the following holds. 

Let $0<\tau\leq \tau_0$,  $0\leq \class<\frac{1}{2}$, and ${8}h^\delta\leq R(h)\leq R_0$.
{Let  $\{\chi_j\}_{j=1}^{N_h}$ be a $\delta$-partition for $\SigH$ associated to a $(\tau, R(h))$-cover.} {Let $N>0$ and $\{\tilde H_h\}_h\subset M$ be a family of submanifolds of codimension $k$  satisfying \eqref{e:conormalClose}.}

There exist $C>0$, so that for every family{ $\{w_h\}_h$ with $w_h\in S_\class \cap C_c^\infty(\tilde H_h)$} there are $C_{_{\!N}}>0$ and 
$$h_0=h_0{(M,P,\{\chi_j\},\delta, \FR, {\{H_h\}_h} )}>0$$  
with the property that for any $0<h<h_0$ and $u\in \mc{D}'(M)$,
\begin{align*}
h^{\frac{k-1}{2}}\Big|\int_{\tilde H_h}w_h u\, d\sigma_{_{\!\tilde H_h}}\Big|
&\leq \frac{C_{n,k} }{\tau^{\frac{1}{2}} {\FR^{\frac{1}{2}}}} \|w_h\|_{_{\!\infty}} \, R(h)^{\frac{n-1}{2}}\sum_{j\in \mathcal J_h(w_h)}{\|Op_h(\chi_j)u\|_{\LM}}\\
&+Ch^{-1}{\|w_h\|_\infty}\|Pu\|_{\Hm} +C_{_{\!N}}h^N\big(\|u\|_{\LM}+{\|Pu\|_{\Hm}}\big),
\end{align*}
where
\begin{equation}\label{e:i}
\mathcal J_h(w_h):=\{j:\; \Lambda_{\rho_j}^\tau (2R(h))\cap \pi^{-1}(\supp w_h)\neq \emptyset\},
\end{equation}
and  {$\pi: \Sigma_{_{\!\tilde H_h, p}} \to \tilde H_h$} is  the canonical projection. {Moreover, the constants $C,  C_{_{\!N}},{h_0}$ are uniform for $\chi_j$ in bounded subsets of $S_\delta$. The constants ${\tau_0},C,C_{_{\!N}},{h_0}$ depend on $\{H_h\}_h$ only through finitely many of the constants $\KR_{_{\alpha}}$ in \eqref{e:curvature}.  The constant $C_{_{\!N}}$ is uniform for $ \{w_h\}_h$ in bounded subsets of $S_\delta$.}
\end{theorem}
%{\cs explain how constants depend on $\KR_{_{\alpha}}$\cs}
%%%%%%%%%%%%%%%%%%%%%%%%%%%%%%%%%%%%%%%%%%%%%%%%%%%%%%%%%%%%%%%%%%%%%%%%%%%%%%%% 

\begin{remark}[{Proof of Theorem~\ref{t:porcupine efxs}}]\label{r:how to prove results for efxs}
We emphasize now that Theorem~\ref{t:porcupine} is the key analytical estimate of this article. In particular, Theorem \ref{t:porcupine efxs} is a direct consequence of it. Indeed, we work with $P=-h^2\Delta_g -I$,  $Pu=0$. {Let $H_h=\{x\}$ and $\tilde H_h=\{x_h\}$ with $x_h \in B(x, h^\delta)$. Let $w_h=1$ for all $h$.} In particular, ${\mc{J}_h}(w_h)=\{1, \dots, N_h\}$. Note that since  $H_h=\{x\}$, then $\SNH=S_x^*M$. Also, in this case  $\Tinj(\{x\})$ can be chosen uniform on $M$, and we have $H_pr_H=2$ and $\FR=2$. Moreover, ${\KR_{_{\alpha}}}$ can be taken arbitrarily small. {This yields $\tau_0=\tau_0(M,g)$, $R_0=R_0(M,g)$ and $h_0=h_0(M,g,\{\chi_j\}, \delta)$.}  Theorem~\ref{t:porcupine efxs} follows.
\end{remark}

We next present Theorem \ref{t:coverToEstimate} which combines Theorem~\ref{t:porcupine} with an application of Egorov's theorem  to control eigenfunction averages using dynamical information at $\SigH$.  In fact, all the applications to obtaining quantitative improvements {for $L^\infty$ bounds and averages} described in the introduction are reduced to a purely dynamical argument together with an application of Theorem~\ref{t:coverToEstimate}.

{As explained before Theorem \ref{t:coverToEstimate efxs}, it will be convenient for us to work with covers by tubes without too much redundancy. We therefore introduce the following definition.
\begin{definition}
\label{d:good cover}
Let $A\subset \SigH$, $r,\,\mathfrak{D}>0$, and  $\{\rho_j(r)\}_{j=1}^{N_r} \subset A$. We say that the collection of tubes $\{\Lambda_{\rho_j}^\tau(r)\}_{j=1}^{N_r}$ is a \emph{$(\mathfrak{D},\tau, r)$-good cover} of a set $A\subset  \SigH$ provided that it is a $(\tau,r)$-cover for $A$ and there exists a partition $\{\mathcal{J}_\ell\}_{\ell=1}^{\mathfrak{D}}$ of $\{1, \dots, N_r\}$ so that for every $\ell\in \{1, \dots, \mathfrak{D}\}$
    \[
    \Lambda_{\rho_j}^\tau (3r)\cap \Lambda_{\rho_i}^\tau(3r)=\emptyset\qquad i,j\in \mathcal{J}_\ell, \quad  i\neq j.
    \]
\end{definition}
\noindent In Proposition \ref{l:cover} we prove that {there exists a} $(\mathfrak{D}_n,\tau, r)$-good cover {for $\SigH$} where $\mathfrak{D_n}$ only depends on $n$. {Thus, one can always work with such a cover}.}

We define the \emph{maximal expansion rate } and the \emph{Ehrenfest time} at frequency $h^{-1}$ respectively:
\begin{equation}\label{e:Lmax}
\Lambda_{\max}:=\limsup_{|t|\to \infty}\frac{1}{|t|}{\log} \sup_{{\{|p|\leq \frac{1}{2}\}}}\|d\varphi_t(x,\xi)\|, \qquad 
T_e(h):=\frac{\log h^{-1}}{2\Lambda_{\max}}.
\end{equation}
Note that $\Lambda_{\max}\in[0,\infty)$ and if $\Lambda_{\max}=0$, we may replace it by an arbitrarily small positive constant.

The next theorem involves many parameters; their role is to provide flexibility when applying the theorem. {This theorem controls averages over uniformly conormally transverse families of submanifolds in terms of families $\{\G_\ell\}_\ell$ of  tubes that run conormally to the submanifolds and are $[t_\ell, T_\ell]$ non self-looping. For an explanation on the roles of these tubes and non-looping times, see the text after Theorem~\ref{t:coverToEstimate efxs}}.
%%%%%%%%%%%%%% %%%%%%%%%%%%%%%%%%%%%%%%%%%%%%%%%%%%%%%%%%%%%%%%%%%%%%%%%%%%%%%%%%

\begin{theorem}
\label{t:coverToEstimate}
Let $P\in \Psi^m(M)$ be a self-adjoint operator with classically elliptic symbol $p$.
Let $\{H_h\}_h \subset M$ be a regular family of submanifolds of codimension $k$ that is uniformly conormally transverse for $p$.
Let $\{\tilde H_h\}_h$ be a family of submanifolds of codimension $k$  satisfying \eqref{e:conormalClose}.
 Let $0<\delta<\frac{1}{2}$, $N>0$ and ${\{w_h\}_h}$ with {$w_h\in S_\class \cap C_c^\infty(\tilde H_h)$}. There exist positive constants $\tau_0=\tau_0(M,p,\Tinj,\FR,{\{H_h\}_h})$, {$R_0=R_0(M,p, {\KR_{_{0}}},{k,\Tinj,\FR })$,} and $C_{n,k}$ depending only on $n$ and $k$, $h_0=h_0(M,P,{\delta,\FR, {\{H_h\}_h}})$, and for each $0<\tau\leq \tau_0$  there are 
 $$C=C(M,p,\tau,\delta,{\FR }, {\{H_h\}_h}),\qquad C_{_{\!N}}=C_{_{\!N}}(M, {P}, N,\tau,\delta,{\{w_h\}_h},{\FR },{\{H_h\}_h}),$$   
 so that the following holds.

Let ${8}h^\delta\leq R(h){<R_0}$,{ $0\leq\alpha< 1-2{\limsup_{h\to 0}\frac{\log R(h)}{\log h}}$,} and suppose $\{\Lambda_{_{\rho_j}}^\tau(R(h))\}_{j=1}^{N_h}$  is a {$(\mathfrak{D},\tau, R(h))$-{good }cover of $\SigH$ for some $\mathfrak{D}>0$}. In addition, suppose there exist $\mc{B}\subset \{1,\dots, N_h\}$ and a finite collection  $\{\mc{G}_\ell\}_{\ell \in \mathcal L} \subset \{1,\dots, N_h\}$ with 
$$
\mathcal J_h(w_h)\;\subset\;  \mc{B} \cup \bigcup_{\ell \in \mathcal L}\mc{G}_\ell,
$$
where $\mathcal J_h(w_h)$ is defined in \eqref{e:i},
 and so that for every $\ell \in \mathcal L$ there exist  $t_\ell=t_\ell(h)>0$ and ${T_\ell=T_\ell(h)}$ {with $t_\ell(h)\leq T_\ell \leq {2} \alpha T_e(h)$}  so that
$$
\bigcup_{j\in \mc{G}_\ell}\Lambda_{_{\rho_j}}^\tau(R(h))\;\;\text{ is }\;\;[t_\ell,T_{\ell}]\text{ non-self looping}.
$$
Then, for $u\in \mc{D}'(M)$ and $0<h<h_0$,
\begin{align*}
h^{\frac{k-1}{2}}\Big|\int_{\tilde H_h} w_h u\, d\sigma_{\tilde H_h}\Big|
&\leq\frac{C_{n,k}{\mathfrak{D}}\|w_h\|_{_{\!\infty}}R(h)^{\frac{n-1}{2}}}{\tau^{\frac{1}{2}}\FR^{\frac{1}{2}}}
\Bigg(|\mc{B}|^{\frac{1}{2}}+\sum_{\ell \in \mathcal L }\frac{(|\mc{G}_\ell|t_\ell)^{\frac{1}{2}}}{T^{\frac{1}{2}}_\ell}\Bigg)\|u\|_{\LM} \\
&+\frac{C_{n,k}{\mathfrak{D}}\|w_h\|_{_{\!\infty}}R(h)^{\frac{n-1}{2}}}{\tau^{\frac{1}{2}}\FR^{\frac{1}{2}}} \sum_{\ell \in \mathcal L}\frac{(|\mc{G}_\ell|t_\ell T_\ell)^{\frac{1}{2}}}{h}\;\|Pu\|_{\LM}\! \\
&+Ch^{-1}{\|w_h\|_\infty}\|Pu\|_{\Hm}+C_{_{\!N}}h^N\big(\|u\|_{\LM}+{\|Pu\|_{\Hm}}\big).
\end{align*}
Here, the constant $C_{_{\!N}}$ depends on $\{w_h\}_h$  only through finitely many $S_\delta$ seminorms of $w_h$. The constants ${\tau_0},C,C_{_{\!N}}, {h_0}$ depend on $\{H_h\}_h$ only through finitely many of the constants $\KR_{_{\alpha}}$ in \eqref{e:curvature}.
\end{theorem}

%%%%%%%%%%%%%%%%%%%%%%%%%%%%%%%%%%%%%%%%%%%%%%%%%%%%%%%%%%%%%%%%%%%%%%%%%%%%%%%%
\medskip
\begin{remark}[{Proof of Theorem~\ref{t:coverToEstimate efxs}}]\label{r:how to prove results for cover efxs}
{Note that making the same observations in Remark \ref{r:how to prove results for efxs} it is straightforward to see that Theorem \ref{t:coverToEstimate efxs} is a generalization of Theorem \ref{t:coverToEstimate}. The only consideration is that the tubes are built using the geodesic flow, which is generated by the symbol $p(x,\xi)=|\xi|_{g(x)}-1$ instead of $p_0(x,\xi)=|\xi|^2_{g(x)}-1$. We explain how to pass from one flow to the other in \S\ref{S:change of hamiltonian}.} 
\end{remark}

\begin{remark}
\label{r:extra}
Note that in this paper we study averages of relatively weak quasimodes for the Laplacian with no additional assumptions on the functions. This is in contrast with results which impose additional conditions on the functions such as: 
that they be Laplace eigenfunctions that simultaneously satisfy additional equations~\cite{I-S,GT18a,Ta18,ToZe03}; that they be eigenfunctions in the very rigid case of the flat torus~\cite{B93,Gros}; or that they form a density one subsequence of Laplace eigenfunctions~\cite{JZ}.
\end{remark}

\begin{remark}
We also note that the norm $C\|Pu\|_{\Hm}$ in Theorems~\ref{t:coverToEstimate} and~\ref{t:porcupine} may be replaced by $C_\e\|Pu\|_{\Hs{\!\!\frac{k-2m+\e}{2}}}$ for any $\e>0$. However, for notational convenience  we have chosen to use a sub-optimal Sobolev embedding to produce the $\|Pu\|_{\Hm}$ term.
\end{remark}

%%%%%%%%%%%%%%%%%%%%%%%%%%%%%%%%%%%%%%%%%%%%%%%%%%%%%%%%%%%%%%%%%%%%%%%%%%%%%%%%
\addcontentsline{toc}{section}{\quad\, \bf{Analysis of Quasimodes}}
%%%%%%%%%%%%%%%%%%%%%%%%%%%%%%%%%%%%%%%%%%%%%%%%%%%%%%%%%%%%%%%%%%%%%%%%%%%%%%%%
%%%%%%%%%%%%%%%%%%%%%%%%%%%%%%%%%%%%%%%%%%%%%%%%%%%%%%%%%%%%%%%%%%%%%%%%%%%%%%%%
%%%%%%%%%%%%%%%%%%%%%%%%%%%%%%%%%%%%%%%%%%%%%%%%%%%%%%%%%%%%%%%%%%%%%%%%%%%%%%%%
%%%%%%%%%%%%%%%%%%%%%%%%%%%%%%%%%%%%%%%%%%%%%%%%%%%%%%%%%%%%%%%%%%%%%%%%%%%%%%%%
%%%%%%%%%%%%%%%%%%%%%%%%%%%%%%%%%%%%%%%%%%%%%%%%%%%%%%%%%%%%%%%%%%%%%%%%%%%%%%%%
\section{Estimates near bicharacteristics: Proof of Theorem  \ref{t:porcupine}}
\label{s:loc}
%%%%%%%%%%%%%%%%%%%%%%%%%%%%%%%%%%%%%%%%%%%%%%%%%%%%%%%%%%%%%%%%%%%%%%%%%%%%%%%%
%%%%%%%%%%%%%%%%%%%%%%%%%%%%%%%%%%%%%%%%%%%%%%%%%%%%%%%%%%%%%%%%%%%%%%%%%%%%%%%%
%%%%%%%%%%%%%%%%%%%%%%%%%%%%%%%%%%%%%%%%%%%%%%%%%%%%%%%%%%%%%%%%%%%%%%%%%%%%%%%%
%%%%%%%%%%%%%%%%%%%%%%%%%%%%%%%%%%%%%%%%%%%%%%%%%%%%%%%%%%%%%%%%%%%%%%%%%%%%%%%%
%%%%%%%%%%%%%%%%%%%%%%%%%%%%%%%%%%%%%%%%%%%%%%%%%%%%%%%%%%%%%%%%%%%%%%%%%%%%%%%%

The proof of Theorem \ref{t:porcupine} relies on several estimates. In what follows we give an outline of the proof to motivate three propositions that together yield the proof of Theorem~\ref{t:porcupine}.  \medskip

{\noindent {\bf A note on notation.} Throughout this section to ease notation we write 
\[H, \,\tilde{H},\, w, \qquad \text{ instead of}\qquad H_h,\, \tilde{H}_h,\, w_h.\]}
%%%%%%%%%%%%%%%%%%%%%%%%%%%%%%%%%%%%%%%%%%%%%%%%%%%%%%%%%%%%%%%%%%%%%%%%%%%%%%%%

\noindent{\bf Proof Theorem \ref{t:porcupine}.} {Let $0<\delta<\tfrac{1}{2}$.}
{In what follows $\tau_0$, $R_0$, $\e_0$ {and $h_0$} are the constants given by Proposition \ref{l:manyTubes}. Let  $ {{8}}h^\delta\leq R(h){\leq} R_0$, {and $N>0$}.}

 Let $0<\tau\leq \tau_0$ and $\{\rho_j\}_{j=1}^{N_h} \subset \SigH$ be so that the tubes 
{$\{\Lambda_{_{\!\rho_{j}}}^\tau (R(h))\}_{j=1}^{N_h}$ form a $(\tau, R(h))$- covering of $\SigH$.}
%  This is possible by {setting $r_0=h^\delta$ and ${r_1}=R(h)$ in Proposition~\ref{l:cover} below, together with the assumption that $h^\delta\leq \frac{1}{2}R(h)$. } 
 We divide the proof into three steps, each of which relies on a  proposition.\\

  %%%%%%%%%%%%%%%%%%%%%%
  %%%%%%%%%%%%%%%%%%%%%%
\noindent{{\bf Step 1} {\emph{(Localization near conormal directions).}}}
Let $\chi_{_{0}} \in C_{{c}}^\infty(\R;[0,1])$ be a smooth cut-off function with $\chi_{_{0}}(t)=1$ for $t \leq \tfrac{1}{2}$ and $\chi_{_{0}}(t)=0$ for $t \geq {1}$. {Let $K>0$ be defined as in \eqref{e:Kdef} below} and define  
\begin{equation}\label{e:beta}
\beta_\class{(x',\xi')}:=\chi_{_{0}}\!\left(\frac{K|\xi'|_{_{{{\tilde{H}}}}}}{h^{\class}}\right),
\end{equation}
{where $|\xi'|_{_{{{\tilde{H}}}}}$ denotes the length of $\xi'$ as an element of $T^*_{x'}{\tilde{H}}$ with respect to the Riemannian metric induced on ${\tilde{H}}$.}  
% \begin{center}
%\includegraphics[height=3cm]{beta.pdf}
%\end{center}
In Proposition~\ref{l:intByParts} we prove that {for $w\in S_\delta\cap C_c^\infty({\tilde{H}})$} there exists $C_{_{\!N}}>0$, {depending on $P$}, {finitely many seminorms of } $w$, and finitely many of the constants $\KR_{_{\alpha}}$ in \eqref{e:curvature}, so that {for all $h>0$}
\begin{equation}\label{e:seaweed}
\Big|\int_{{\tilde{H}}} wud\sigma_{{\tilde{H}}}\Big|\leq  \|wOp_h(\beta_\class)u\|_{\LMo}+{C_{_{\!N}} h^N}\big(\|u\|_{\LM}{+\|Pu\|_{\Hm}}\big).
\end{equation}

%%%%%%%%%%%%%%%%%%%%%%
%%%%%%%%%%%%%%%%%%%%%%
\noindent{{\bf Step 2} {\emph{(Coverings by bicharacteristic beams).}}}
Let $\tilde{R}(h)=\tfrac{1}{2}R(h)$, $\tilde{\tau}=\tfrac{\tau}{4}$.

In Proposition \ref{l:cover} we prove that there exist a constant $\mathfrak{D}_n$, depending only on $n$,  points $\{\tilde \rho_j\}_{j=1}^{\tilde N_h}\subset \SigH$, 
and a partition $\{\mathcal J_{i}\}_{i=1}^{\mathfrak{D}_{n}}$ of $\{1,\dots, \tilde N_h\}$, so that 
\begin{itemize}
    \item $\Lambda_{_{\!\SigH}}^{{\tilde{\tau}}}(\tfrac{1}{2}\tilde{R}(h))\subset \bigcup_{j=1}^{\tilde N_{h}}\Lambda_{_{\tilde \rho_j}}^{\tilde \tau}(\tilde{R}(h)),$ \medskip
    \item $\Lambda_{_{\tilde \rho_j}}^{\tilde \tau}(3\tilde{R}(h))\cap \Lambda_{_{\tilde \rho_\ell}}^{\tilde \tau}(3\tilde{R}(h)))=\emptyset, \qquad j,\ell\in \mathcal J_i,\quad j\neq \ell.$
\end{itemize}
{That is, we work with a  $(\mathfrak{D}_{n}, \tilde \tau, \tilde R(h))$-good cover.}

In Proposition \ref{l:nicePartition} we prove that  {there exists $C_0>0$ so that for $0<\e<\e_0$ and $0<h \leq h_0$ there is a} partition of unity $\{\chi_j^\P\}_j$  for  $\Lambda_{_{\!\SigH}}^{{\tilde{\tau}}}(\tfrac{1}{2}\tilde{R}(h))$ with
\begin{itemize}
\item $\chi_j^\P\in S_\class\cap  {C^\infty_c(\TM ;[-C_0h^{1-2\delta},1+C_0h^{1-2\delta}])}$,\medskip
\item $\supp \chi_j^\P\subset {\Lambda^{\tilde{\tau}+\e}_{\tilde{\rho}_j}(\tilde R(h))},$\medskip
\item $\MSh([P,Op_h(\chi_j^\P)])\cap \Lambda_{_{\!\SigH}}^{\tilde{\tau}}({\e})=\emptyset.$
\end{itemize} 
{
Indeed, this follows from applying Proposition \ref{l:nicePartition} since $\tilde R(h)=\tfrac{1}{2}R(h)\geq \tfrac{1}{2} {8}h^\delta \geq 2h^\delta$.
From now on we fix $\e>0$ so that $\e<\e_0$ and $\e<\tfrac{\tau}{4}$.}
{See Appendix \ref{s:micro} for background on microsupports.} \smallskip

%%%%%%%%%%%%%%%%%%%%%%
%%%%%%%%%%%%%%%%%%%%%% \ 
\noindent{{\bf Step 3} {\emph{(Estimates near bicharacteristics).}}} In Proposition~\ref{l:manyTubes} we prove that there exist $C_{n,k}>0$, $C_{_{\!N}}>0$, $h_0>0$, and $C>0$ so that for all $w \in S_\delta\cap C^\infty_c({\tilde{H}})$ {and $0<h<h_0$}, if $\{\chi_j^\P\}$ is as before, then  
\begin{align}\label{e:goooool}
 h^{\frac{k-1}{2}}\|w Op_h(\beta_\class)u\|_{\LMo}
 &\leq C_{n,k}{\|w\|_{_{\!\infty}}}R(h)^{\frac{n-1}{2}}\sum_{j\in \tilde{\mathcal I}_h(w)} \frac{\|Op_h(\chi_j^\P)u\|_{\LM}}{\tau^{\frac{1}{2}}|H_pr_{_{\!H}}(\tilde \rho_j)|^{\frac{1}{2}}} \notag \\
&+{Ch^{-1}\|w\|_{\infty}\|Pu\|_{\Hm}}  +C_{_{\!N}}h^N{\|w\|_{\infty}}\|u\|_{\LM},
\end{align}
where $\tilde{\mathcal I}_h(w)=\{j:\;  \Lambda_{\tilde{\rho}_j}^{\tilde{\tau}}(\tilde{R}(h))\cap\pi^{-1}(\supp(w))\neq \emptyset\}$.

{\begin{remark} It is crucial that the cutoffs $\chi_j$ supported in disjoint tubes act almost orthogonally. This allows for efficient decomposition and recombination of estimates based on tubes and we use this fact throughout the text.
\end{remark}}

Next, let $\{\chi_\ell\}_{\ell=1}^{N_h}$ be a \emph{$\delta$-partition} associated to the  $(\tau, R(h))$-cover {$\{\Lambda_{\rho_\ell}^\tau(R(h))\}_{\ell=1}^{N_h}$ of $\SigH$}. We claim that {for each $j\in \tilde{\mc{I}}_h(w)$}
\begin{equation}\label{e:relating partitions}
    \chi_j^\P \leq 2\sum_{\ell \in \mc{A}_j} \chi_\ell,
\end{equation}
where 
$$\mc{A}_j=\{\ell :\;\Lambda_{\tilde{\rho}_j}^{{\tau/2}}(\tilde R(h)) \cap \Lambda_{\rho_\ell}^\tau (R(h)) \neq \emptyset\}.$$
Indeed, this follows from two observations. The first one is that $\supp \chi_j^\P \subset \Lambda_{\tilde\rho_j}^{{\tau}/{2}}(\tilde R(h))$ since $\e<\tfrac{\tau}{4}$. The second observation is that {on $ \Lambda_{\tilde\rho_j}^{\tau/2}(\tilde R(h))$} we have
$\sum_{\ell=1}^{N_h} \chi_\ell  =\sum_{\ell \in \mc{A}_j} \chi_\ell \geq 1$ since
$\sum_{\ell=1}^{N_h} \chi_\ell \geq 1$ on $\Lambda^{{\tau/2}}_{S^*_xM}(\tilde R(h))$ and {$\supp \chi_\ell \subset \Lambda_{\rho_\ell}^\tau(R(h))$}.
 Combining this with the fact that $\chi_j^\P\leq 1+C_0h^{1-2\delta}$ yields the claim in \eqref{e:relating partitions}.
 
{Next, note that if $j\in \tilde{\mc{I}}_h(w)$, then $\mathcal{A}_j\subset \mathcal{J}_h(w)$ where $\mathcal{J}_h(w)=\{\ell:\;  \Lambda_{\rho_\ell}^{{\tau}}(2R(h))\cap\pi^{-1}(\supp(w))\neq \emptyset\}.$ 
{This follows from the fact} that if  $\ell \in \mathcal{A}_j$, then $\Lambda_{\tilde{\rho}_j}^{{\tau/2}}(\tilde R(h)) \subset \Lambda_{\rho_\ell}^\tau (2R(h))$.}

%Indeed, note that for every $j \in \tilde{\mc{I}}_h(w)$ we have   $d(\tilde{\rho}_j,\pi^{-1}(\supp w))<\tfrac{1}{2}R(h)$. {\cs Not quite. note that $\pi^{-1}(\supp w) \subset \Sigma_{\tilde H,p}$, and by \eqref{e:conormalClose} we only have  $\sup_{\rho \in \Sigma_{_{H_h,p}}} d(\rho ,\Sigma_{_{\tilde{H}_h,p}})\leq  h^\delta$. I think we need to change the definition of $\mathcal{J}_h(w)$ to say  $\Lambda_{\rho_\ell}^{{\tau}}(3R(h))$. Then, everything would be fine since $d(\tilde{\rho}_j,\pi^{-1}(\supp w))<R(h)$  for $h$ small enough   \cs} Also, if $\ell \in \mathcal{A}_j$, then $d(\rho_\ell,\tilde{\rho}_j)<\tfrac{3}{2}R(h)$.  Then, $d(\rho_\ell, \pi^{-1}(\supp w))<2R(h)$ for all $\ell \in \mathcal{A}_j$, and this implies $\ell\in \mathcal{J}_h(w).$

To complete the proof we claim that there exists $C_n>0$ depending only on $n$ so that {for every $\ell \in \{1, \dots, N_h\}$,}
\begin{equation}
\label{e:countingMe}
\#\{j{\in \tilde{\mc{I}}_h(w)}:\; \ell\in \mathcal{A}_j\}\leq C_n.
\end{equation}

Assuming the claim for now, we conclude from \eqref{e:relating partitions} that 
\begin{align*}
\sum_{j\in \tilde{\mathcal I}_h(w)} \frac{\|Op_h(\chi_j^\P)u\|_{\LM}}{|H_pr_{_{\!H}}({\tilde \rho_j})|^{\frac{1}{2}}}
&\leq 4\FR^{-\frac{1}{2}} \sum_{j\in \tilde{\mathcal I}_h(w)} \sum_{\ell \in \mc{A}_j}\|Op_h(\chi_\ell)u\|_{\LM}\\
&\leq 4C_n\FR^{-\frac{1}{2}} \sum_{j\in \mathcal J_h(w)} \|Op_h(\chi_j)u\|_{\LM}.
\end{align*}

{Combining this with \eqref{e:goooool} and \eqref{e:seaweed} finishes the proof of Theorem \ref{t:porcupine}.} 

We now prove~\eqref{e:countingMe}. Suppose that $\ell\in \mathcal{A}_j$. Then, 
$$B({\rho}_\ell,R(h))\cap B(\tilde\rho_j,\tilde R(h))\cap\sigFlow \neq \emptyset.$$
In particular, 
$$
B(\tilde\rho_j, \tilde R(h))\cap \sigFlow \;\subset \;\, B({\rho_\ell},2R(h))\cap \sigFlow .
$$
{Therefore, $\Lambda_{\tilde{\rho_j}}^{\tilde{\tau}}(\tilde{R}(h))\subset \Lambda_{\rho_\ell}^{\tilde{\tau}}(2R(h))$.}
Thus, {since the tubes $\{\Lambda_{\tilde \rho_j}^{\tilde\tau}(3\tilde R(h))\}_{j \in \mc{J}_i}$ are disjoint for each $i=1, \dots, \mathfrak{D}_n$}, there exists $C_n>0$, depending only on $n$, such that for every $\ell \in \{1, \dots, N_h\}$
$$
\#\{j:\;  \ell\in \mathcal{A}_j\}\leq \mathfrak{D}_n\frac{\sup_{\ell}\vol(\Lambda_{{\rho}_{\ell}}^{{\tilde \tau}} (2R(h))}{\inf_{j}\vol(\Lambda_{\tilde \rho_j}^{{\tilde \tau}}(\tilde R(h)))}\leq C_n.
$$
\qed
%%%%%%%%%%%%%%%%%%%%%%%%%%%%%%%%%%%%%%%%%%%%%%%%%%%%%%%%%%%%%%%%%%%%%%%%%%%%

We proceed to state and prove all the propositions needed in the proof of Theorem~\ref{t:porcupine}.
%%%%%%%%%%%%%%%%%%%%%%%%%%%%%%%%%%%%%%%%%%%%%%%%%%%%%%%%%%%%%%%%%%%%%%%%%%%%%%%% 
\subsection{Step 1: Localization near conormal directions}

{Our first result is quite general, and it shows that in order to study integral averages over {$\tilde H$} of a function $v$ it suffices to restrict ourselves to studying the conormal behavior of $v$. That is, the {non-}oscillatory behavior of $v$ along  {$\tilde H$} is encoded in $Op_h(\beta_\delta)v$.}
\begin{lemma}
\label{l:intByPartsa}
Let $0\leq \class<\frac{1}{2}$, $N>0$, and $w\in S_\class \cap C_c^\infty({\tilde{H}})$. Then, there is {$C_{_{\!N}}>0$, depending on finitely many seminorms of $w\in S_\delta$ {and finitely many of the constants $\KR_{_{\alpha}}$ in~\eqref{e:curvature}}, so that for all} $v\in {\mc{D}'}({\tilde{H}})$
$$
 \Big|\int_{{\tilde{H}}} w(1-Op_h(\beta_\class)){( v)}d\sigma_{{\tilde{H}}}\Big|\leq C_{_{\!N}}h^N{\|v\|_{L^2({{\tilde{H}}})}}.
$$
\end{lemma}
%%%%%%%%%%%%%%%%%%%%%%%%%%%%%%%%%%%%%%%%%%%%%%%%%%%%%%%%%%%%%%%%%%%%%%%%%%%%%%%%

%%%%%%%%%%%%%%%%%%%%%%%%%%%%%%%%%%%%%%%%%%%%%%%%%%%%%%%%%%%%%%%%%%%%%%%%%%%%%%%%
\begin{proof}
Let $h>0$. Here, we work in coordinates $(\bar{x},x')\in \re^{k}\times \re^{n-k}$ where $\tilde{H}={\tilde{H_h}}=\{\bar{x}=0\}$. 
{Let $\tilde N$ be so that $N{<} k-n+\tilde N(1-2\delta)$}. Let {$g_{_{\!{\tilde H}}}$} denote the metric induced on {$\tilde H$}. Then, integrating by parts with 
$
L:= \frac{1}{|\xi'|^2} \left( \sum_{j=1}^{n-k} \xi_j'  hD_{x_j} \right),
$
 gives
\begin{align*}
&{\int_{{\tilde{H}}} w(x) \, (1-Op_h(\beta_\class)){v}(x)d\sigma_{_{\!\!{\tilde{H}}}}(x)}=\\
&\qquad=\frac{1}{(2\pi h)^{n-k}}\iiint e^{\frac{i}{h}\langle x-x',\xi'\rangle} w(x)(1-\beta_\class(x,\xi')){v}(x'){\sqrt{|g_{_{\!{\tilde{H}}}}(x')||g_{_{\!{\tilde{H}}}}(x)|}}dxdx'd\xi'\\
&\qquad= \frac{1}{(2\pi h)^{n-k}}\iiint e^{\frac{i}{h}\langle x-x',\xi'\rangle}(L^*)^{\tilde N}\Big[w(x)(1-\beta_\class(x,\xi')){v}(x'){\sqrt{|g_{_{\!{\tilde{H}}}}(x')||g_{_{\!{\tilde{H}}}}(x)|}}\,\Big]dxdx'd\xi'\\
&\qquad\leq C_{_{\!N}}h^{k-n+{\tilde N}(1-2\class)}{\|v\|_{L^2({{\tilde{H}}})}}.
\end{align*}
{Here, $C_{_{\!N}}$ depends on the $C^{\tilde{N}}$ norm of $w$ as well as finitely many of the constants $\KR_{_{\alpha}}$. The second fact follows since the transition maps for the coordinate change which flattens $\tilde{H}$ have $C^{\tilde{N}}$ norm bounded by finitely many of the constants $\KR_{\alpha}$.} 
%{\cs what did you want to say here?\red{That all derivatives of this `flattening' are uniformly bounded in $h$}{Yes, but what about the sentence ``and uniformly bounded over all choices of $\tilde{H}$"?}\red{Meaning that the constants in the bounds do not depend on which of the possible $\tilde{H}$'s we are talking about}}\cs }
\end{proof}
{We next apply Lemma \ref{l:intByPartsa} to the setup of Theorem \ref{t:porcupine}.}
{
\begin{proposition}
\label{l:intByParts}
{Let $P$ be as in Theorem \ref{t:porcupine}.} Let $0\leq \class<\frac{1}{2}$, $N>0$, and $w\in S_\delta\cap C_c^\infty({\tilde{H}})$. Then, there exists $C_{_{\!N}}>0$, depending on $P$, finitely many seminorms of $w\in S_\delta$, {and finitely many of the constants $\KR_{_{\alpha}}$ in~\eqref{e:curvature}}, so that for all $u\in \mc{D}'(M)$ {and all $h>0$}
$$
 \Big|\int_{{\tilde{H}}} w(1-Op_h(\beta_\class)){( u)}d\sigma_{{\tilde{H}}}\Big|\leq C_{_{\!N}}h^N(\|u\|_{\LM}+\|Pu\|_{\Hm}).
$$
\end{proposition}}
\begin{proof}
In order to use Lemma~\ref{l:intByPartsa}, we first {bound $\|u\|_{L^2(\tilde H)}$.} For this, observe that since $p$ is classically elliptic, {by a standard elliptic parametrix construction (see e.g~\cite[Appenix E]{ZwScat}) }
$$
\|u\|_{\Hs{\frac{k+1}{2}}}\leq C(\|u\|_{\LM}+\|Pu\|_{\Hm})
$$
where $C$ depends only on $P$.
In particular, {the semiclassical Sobolev estimates (see e.g.~\cite[{Lemma 6.1}]{Gdefect})  imply that }
$$
\|u\|_{L^2({\tilde{H}})}\leq Ch^{-\frac{k}{2}}(\|u\|_{\LM}+\|Pu\|_{\Hm}).
$$
Using Lemma~\ref{l:intByPartsa} then gives 
$$
 \Big|\int_{{\tilde{H}}} w(1-Op_h(\beta_\class)){( u)}d\sigma_{{\tilde{H}}}\Big|\leq C_{_{\!N}}h^N(\|u\|_{\LM}+\|Pu\|_{\Hm}).
$$
\end{proof}
%%%%%%%%%%%%%%%%%%%%%%%%%%%%%%%%%%%%%%%%%%%%%%%%%%%%%%%%%%%%%%%%%%%%%%%%%%%%%%%% 

%%%%%%%%%%%%%%%%%%%%%%%%%%%%%%%%%%%%%%%%%%%%%%%%%%%%%%%%%%%%%%%%%%%%%%%%%%%%%%%%
\subsection{Step 2: Coverings by bicharacteristic beams}
%%%%%%%%%%%%%%%%%%%%%%%%%%%%%%%%%%%%%%%%%%%%%%%%%%%%%%%%%%%%%%%%%%%%%%%%%%%%%%%%
{We first prove that there is $\mathfrak{D}_n>0$, depending only on $n$, so that for $\tau,r$ small enough, there is a  $(\mathfrak{D}_n,\tau,r)$-good cover of $\SigH$}. We adapt the proof of~\cite[Lemma 2]{Col11} to our purposes.
%We first prove the existence of a covering of {$\Lambda_{_{\!\SigH}}^{{\frac{\tau}{4}}}(\red{2}h^\class)$} by bicharacteristic tubes. In particular, we construct at $(\tau,r)$ good cover of {$\LambdaH(h^\delta)$} for all $r$ small enough. We adapt the proof of~\cite[Lemma 2]{Col11} to our purposes.
%%%%%%%%%%%%%%%%%%%%%%%%%%%%%%%%%%%%%%%%%%%%%%%%%%%%%%%%%%%%%%%%%%%%%%%%%%%%%%%%

\begin{proposition}
\label{l:cover}
There exist $\mathfrak{D}_{n}>0$ depending only on $n$,  ${R_0=R_0(n,k,{\KR_{_{0}}})}>0$, and $0<\tau_{_{\!\SigH}}<\tfrac{\Tinj}{2}$ depending only on $\Tinj$, such that for $0<{r_1}<{R_0}$, $0<r_0\leq \frac{{r_1}}{2} $, and $0<\tau<\tau_{_{\!\SigH}}$ there exist $\{\rho_j\}_{j=1}^{N_{r_1}}\subset \SigH$ {and a partition $\{\mathcal J_{i}\}_{i=1}^{\mathfrak{D}_{n}}$ of $\{1,\dots, N_{r_1}\}$ so that 
\begin{itemize}
    \item $\LambdaH(r_0)\subset \bigcup_{j=1}^{N_{r_1}}\Lambda_{_{\rho_j}}^\tau({r_1}),$ \medskip
    \item $\Lambda_{_{\rho_j}}^\tau(3{r_1})\cap \Lambda_{_{\rho_\ell}}^\tau(3{r_1})=\emptyset, \qquad j,\ell\in \mathcal J_i,\quad j\neq \ell.$
\end{itemize}}
\end{proposition}

%%%%%%%%%%%%%%%%%%%%%%%%%%%%%%%%%%%%%%%%%%%%%%%%%%%%%%%%%%%%%%%%%%%%%%%%%%%%%%%%
\begin{proof}
Let $\{\rho_j\}_{j=1}^{N_{r_1}}$ be a maximal $\frac{{r_1}}{2}$ separated set in $\SigH$. Fix $i_0 \in \{1, \dots,N_{r_1}\} $ and suppose that $B(\rho_{i_0},3{r_1})\cap B(\rho_\ell,3{r_1})\neq \emptyset$ for all {$\ell\in \mathcal{L}_{{{i_0}}} \subset \{1,\dots, N_{r_1}\}$}. Then for all ${\ell\in \mathcal{L}_{{{i_0}}}}$, $B(\rho_\ell,\tfrac{r_1}{2})\subset B(\rho_{i_0},8{r_1}).$ In particular,
$$
\sum_{{\ell\in \mathcal{L}_{{{i_0}}}}}\vol(B(\rho_\ell,\tfrac{{r_1}}{2}))\leq \vol(B(\rho_{i_0},8{r_1})). 
$$
Now, there exist {$\mathfrak{D}_{n}>0$} and  ${R_0}>0$ depending on {$(n,k)$ and} a lower bound on the Ricci curvature of $\SigH $, and hence on only $(n,k,\KR_{_{0}})$, so that for ${r_1}<{R_0}$,
$$\vol(B(\rho_{i_0},8{r_1}))\leq \vol(B(\rho_\ell,14{r_1}))\leq \mathfrak{D}_{n}\vol(B(\rho_\ell,\tfrac{{r_1}}{2})).$$
Hence, 
$$
\sum_{{\ell\in \mathcal{L}_{{{i_0}}}}}\vol(B(\rho_\ell,\tfrac{{r_1}}{2}))\leq \vol(B(\rho_{i_0},8{r_1}))\leq \frac{\mathfrak{D}_n}{{| \mathcal{L}_{{{i_0}}}|}}\sum_{{\ell\in \mathcal{L}_{{{i_0}}}}}\vol(B(\rho_\ell,\tfrac{{r_1}}{2}))
$$
and in particular, $|\mathcal{L}_{{{i_0}}}|\leq \mathfrak{D}_{n}$.

Now, {suppose that}
$$
\Lambda_{_{\rho_k}}^\tau(3{r_1})\cap \Lambda_{_{\rho_{{i_0}}}}^\tau(3{r_1})\neq \emptyset.
$$
Then, there exists $q_k\in B(\rho_k,3{r_1})\cap\sigFlow $, $q_{{i_0}}\in B(\rho_{{i_0}},3{r_1})\cap\sigFlow $ and {$ t_k,t_{{i_0}}\in[-\tau, \tau]$}  so that 
$$
\varphi_{_{\!t_k-t_{{i_0}}}}(q_k)=q_{{i_0}}.
$$
{Here, $\sigFlow $ is the hypersurface defined in \eqref{e:Hsig}}.
In particular, choosing $\tau_{_{\!\SigH}}<\Tinj/2$, this implies that $q_k=q_{{i_0}}$, $t_k=t_{{i_0}}$ and hence $B(\rho_\ell,3{r_1})\cap B(\rho_{{i_0}},3{r_1})\neq \emptyset$. {This implies that $j\in \mc{L}_{i_0}$ and hence that there are at most $\mathfrak{D}_n$ such distinct $j$ (including $i_0$).}

At this point we have proved that each of the tubes $\Lambda_{_{\rho_j}}^\tau({r_1})$ intersects at most $\mathfrak{D}_{n}-1$ other tubes. We now construct the sets $\mathcal J_1,\dots, \mathcal J_{\mathfrak{D}_{n}}$ using a greedy algorithm. We will say that $i$ intersects $j$ if 
$$
\Lambda_{_{\rho_i}}^\tau({r_1})\cap \Lambda_{_{\rho_j}}^\tau({r_1})\neq \emptyset.
$$
First place $1\in \mathcal J_1$. Then suppose we have placed $j=1,\dots, \ell$ in $\mathcal J_1,\dots, \mathcal J_{\mathfrak{D}_{n}}$ so that each of the $\mathcal J_i$'s consists of disjoint indices. Then, since $\ell+1$ intersects at most $\mathfrak{D}_{n}-1$ indices, it is disjoint from $\mathcal J_i$ for some $i$. We add $\ell$ to $\mathcal J_i$. By induction we obtain the partition $\mathcal J_1,\dots ,\mathcal J_{\mathfrak{D}_{n}}$.

Now, suppose $r_0\leq {r_1}$ and that there exists $\rho \in \Lambda_{_{\SigH}}^\tau(r_0)$ so that {$\rho \notin \bigcup_i \Lambda_{_{\rho_i}}^\tau(r_1)$}. Then, there are $|t|<\tau+r_0$ and $q\in \sigFlow $ so that
$$
\rho=\varphi_t(q),\qquad d(q,\SigH)<r_0,\qquad \min_id(q,\rho_i)\geq {r_1}.
$$
In particular, by the triangle inequality, there exists $\tilde{\rho}\in \SigH$, 
$$
d(\tilde{\rho},\rho_i)\geq d(q,\rho_i)-d(q,\tilde{\rho})>{r_1}-r_0.
$$
This contradicts the maximality of $\{\rho_j\}_{j=1}^{N_{r_1}}$ if $r_0\leq {r_1}/2$.

\end{proof}
%%%%%%%%%%%%%%%%%%%%%%%%%%%%%%%%%%%%%%%%%%%%%%%%%%%%%%%%%%%%%%%%%%%%%%%%%%%%%%%%
%%%%%%%%%%%%%%%%%%%%%%%%%%%%%%%%%%%%%%%%%%%%%%%%%%%%%%%%%%%%%%%%%%%%%%%%%%%%%%%%
We proceed to build a $\delta$-partition of unity associated to the cover we constructed in Proposition~\ref{l:cover}. {The key feature in this partition will be that it is invariant under the bicharacteristic flow. Indeed, the partition is built so that its quantization commutes with the operator $P$ in a neighborhood of $\SigH$.}
\begin{proposition}
\label{l:nicePartition}There exist $\tau_1=\tau_1(\Tinj)>0$ and  $\e_1=\e_1(\tau_1)>0$, and given $0<\delta<\tfrac{1}{2}$,{ $0<\e\leq \e_1$} there exists $h_1>0$,
so that for any $0<\tau\leq \tau_1$, and $R(h)\geq 2h^\delta$, the following holds.

There exist $C_1>0$  so that  for all $0<h\leq h_1$ and all $(\tau,R(h))$-covers of $\SigH$
 there exists a partition of unity $\chi_j\in S_\class\cap  C^\infty_c(\TM ;[-C_1h^{1-2\delta},1+C_1h^{1-2\delta}])$ on $\LambdaH({\frac{1}{2}R(h)})$
for which
\begin{itemize}
\item$\supp \chi_j\subset \Lambda_{\rho_j}^{\tau+{\e}}(R(h))$, \medskip
\item $\MSh([P,Op_h(\chi_j)])\cap \LambdaH({\e})=\emptyset,$
\end{itemize}
{and the $\chi_j$ are uniformly bounded in $S_\delta$.}
\end{proposition}
%%%%%%%%%%%%%%%%%%%%%%%%%%%%%%%%%%%%%%%%%%%%%%%%%%%%%%%%%%%%%%%%%%%%%%%%%%%%%%%%

%%%%%%%%%%%%%%%%%%%%%%%%%%%%%%%%%%%%%%%%%%%%%%%%%%%%%%%%%%%%%%%%%%%%%%%%%%%%%%%%
\begin{proof}
Let $\sigFlow $ be as in~\eqref{e:Hsig} $\tau_1<\frac{1}{2}\Tinj$ {and fix $0<\tau \leq \tau_1$}. Then let $\e_1>0$ be so small that $\Lambda_{\SigH}^{\tau_1}(\e_1)\subset\Lambda^{2\tau_1}_{\sigFlow }(0)$, {fix $0<\e<\e_1$} {and let $h_1$ be so small that $h^\delta \leq {\e}$ for all $0<h\leq h_1$}. {For each $j\in \{1, \dots, N_h\}$ let $\mc{H}_j=\sigFlow \cap \Tj$}. Let $\{\psi_j\}\subset C_c^\infty(\sigFlow ;[0,1]){\cap S_\delta}$ be a partition of unity on $\sigFlow \cap \LambdaH({\frac{1}{2}R(h)})$ subordinate to $\{\mc{H}_j\}_{j=1}^{N_h}$ {that is uniformly bounded in $S_\delta$}. Then, define  {  $a_{j,0} \in S_\delta$ on $\LambdaH({\e})$} by solving 
$$
a_{j,0}|_{\sigFlow }={\psi_j},\qquad H_pa_{j,0}=0\;\; \text{ on }\;\;\LambdaH({\e}).
$$
Clearly, $a_{j,0}$ defined in this way is a partition of unity {for $\LambdaH({\frac{1}{2}R(h)})$}. Furthermore, we can extend $a_{j,0}$ to $T^*\!M$ as an element of $S_\delta$ so that 
$$
\supp a_{j,0}\subset \bigcup_{|t|\leq{\tau+\e}+R(h)}\varphi_t(\mc{H}_j)\subset \Lambda_{\rho_j}^{{\tau+\e}}(R(h)),\qquad 0\leq a_{j,0}\leq 1
$$ 
Note also that since $P\in \Psi^m(M)$ {and $H_p a_{j,0}=0$}, for $b\in S_\delta$ with $\supp b\subset \LambdaH({\e})$,
$$
Op_h(b)[P,Op_h(a_{j,0})]\in h^{2-2\delta}\Psi_\delta(M).
$$

We define $a_{j,k}$ by induction. Suppose we have $a_{j,\ell}$, $\ell=0,\dots, k-1$, so that if we set  $\chi_{j,k-1}:=\sum_{\ell=0}^{k-1} h^{\ell (1-2\class)}a_{j,\ell}$, then
\begin{enumerate}
\item[A)]\; ${\displaystyle \sum_{j=1}^{N_h}\chi_{j,k-1}\equiv 1}$\;\; on  $\LambdaH({\frac{1}{2}R(h)})$,\\
 \item[B)] \; $e_{j,k}:=\sigma\Big(h^{-1-k(1-2\class)}[P,Op_h(\chi_{j,k-1})]\Big) \in S_\class$ \;\; on $\LambdaH({\e}).$
\end{enumerate}
Then, {for every $k\geq 1$} define $a_{j,k}\in  S_\class$ by 
\begin{equation}\label{E:defn of a}
a_{j,k}|_{\sigFlow }=0,\quad H_pa_{j,k}=-i{e_{j,k}}\;\; \text{ on }\;\;\LambdaH({\e}).
\end{equation}
Next extend $a_{j,k}$ to $\TM $ as an element of $S_\delta$ so that 
$$
\supp a_{j,k}\subset \bigcup_{|t|\leq{\tau+\e}+R(h)}\varphi_t(\mc{H}_j)\subset \Lambda_{\rho_j}^{{\tau+\e}}(R(h)).
$$
Now, since $\sum_{j=1}^{N_h} \chi_{j,k-1}\equiv 1$ on $\LambdaH({\frac{1}{2}R(h)})$, by  (B) we see that for $\rho\in \LambdaH({\frac{1}{2}R(h)})$,
\begin{align*}
\sum_{j=1}^{N_h} e_{j,k}(\rho)
%&=\sigma\Big(\sum_{j=1}^{N_h} h^{-1-k(1-2\delta)}[P,Op_h(\chi_{j,k-1})]\Big)(\rho)\\
%&
=\sigma\Big(h^{-1-k(1-2\delta)}\Big[P,Op_h\Big(\sum_{j=1}^{N_h}\chi_{j,k-1}\Big)\Big]\Big)(\rho)=0.
\end{align*}
 In particular, \eqref{E:defn of a} gives that  $\sum_{j=1}^{N_h} a_{j,k}=0$ on $\LambdaH({\frac{1}{2}R(h)})$. Therefore, since $\chi_{j,k}= \chi_{j,k-1}+h^{k(1-2\delta)}a_{j,k}$,   we conclude that 
\[\sum_{j=1}^{N_h} \chi_{j,k}=1 \qquad  \text{on} \;\; \LambdaH({\tfrac{1}{2}R(h)}),\] 
and hence (A) is satisfied for  $a_{j,\ell}$ with  $\ell=0,\dots, k$. To show that (B) is also satisfied, let $b\in S_\delta$ with $\supp b\subset \LambdaH({\e})$. By assumption, we have
$$
Op_h(b)[P,Op_h(\chi_{j,k-1})]\in h^{1+k(1-2\delta)}\Psi_\delta(M).
$$
Also, using once again that $P\in \Psi^m(M)$ {and that $H_pa_{j,k}=-ie_{j,k}$}
$$
Op_h(b)[P,Op_h(a_{j,k})]\in h\Psi_\delta(M)+ h^{2-2\delta}\Psi_\delta(M).
$$
Hence,
$$
Op_h(b)[P,Op_h(\chi_{j,k})]={Op_h(b)\big[P,Op_h( \chi_{j,k-1}+h^{k(1-2\delta)}a_{j,k})\big]}\in h^{1+k(1-2\delta)}\Psi_\delta(M),
$$
and so, on $\LambdaH({\e})$,
\begin{align*}&\sigma(h^{-1-k(1-2\delta)}Op_h(b)[P,Op_h(\chi_{j,k})])=\\
&\qquad\qquad =\sigma\! \Big(h^{-1-k(1-2\class)}\!Op_h(b)\Big([P,Op_h(\chi_{j,k-1})]+h^{k(1-2\class)}[P,Op_h(a_{j,k})]\Big)\Big)\\
&\qquad\qquad =b(e_{j,k}-e_{j,k})=0.
\end{align*}
In particular, 
\begin{equation}
\label{e:howdy}
    Op_h(b)[P,Op_h(\chi_{j,k})]\in h^{1+(k+1)(1-2\delta)}\Psi_\delta(M),
\end{equation}
and $ e_{j,k+1} \in S_\delta$ on $\LambdaH({\e})$ as claimed.

Finally, let 
$$\chi_j\sim \sum_{\ell=0}^\infty h^{\ell(1-2\class)}a_{j,\ell}.$$
Then, using~\eqref{e:howdy},
$$
\MSh([P,Op_h(\chi_j)])\cap \LambdaH({\e})=\emptyset.
$$
%{To see this note that the previous assertion is equivalent to showing that for all $\alpha, N$ there exists $C_{\alpha,N}$ so that 
%\[\sup_{(x,\xi)\in \LambdaH(\e_1)} |H_p \chi_j(x,\xi)| \leq C_{\alpha, N} h^{N-1}.\]
%Also, within $ \LambdaH(\e_1)$ relation \eqref{E:defn of a} gives  $H_p \chi_j=i\sum_{k=0}^\infty h^{k(1-2\class)}e_{j,k}$
%\cs i think we need define $a_{j,k}$ so that $H_pa_{j,k}=ih^{k(1-2\class)}e_{j,k}$ and so $H_p\chi_j=\sum e_{j,k} =0$\cs
%\ \\
%\cs Also, we need to define $a_{j,k}$ so that the support of $\chi_j$ does what we claim.\cs}

Now, note that by construction $\{\chi_{j}\}$ remains a partition of unity modulo $O(h^\infty)$ and by adding an $h^\infty$ correction to teach term, we construct $\{\chi_j\}$ so that it forms a partition of unity. {We also have by construction that $\chi_j\in C^\infty_c(\TM ;[-C_1h^{1-2\delta},1+C_1h^{1-2\delta}])$ {for some $C_1$ depending only on $(M,p$)} {and finitely many of the constants $\KR_\alpha$}.}
\end{proof}

%%%%%%%%%%%%%%%%%%%%%%%%%%%%%%%%%%%%%%%%%%%%%%%%%%%%%%%%%%%%%%%%%%%%%%%%%%%%%%%%

%%%%%%%%%%%%%%%%%%%%%%%%%%%%%%%%%%%%%%%%%%%%%%%%%%%%%%%%%%%%%%%%%%%%%%%%%%%%%%%%

\subsection{Step 3: Estimate near bicharacteristics}
{Let $h>0$.}
{Let $(x',\tilde{x})$ be Fermi coordinates near {$\tilde{H}=\tilde{H}_h$} with corresponding dual coordinates $(\xi',\tilde{\xi})$. Then, since $H$ is {uniformly} conormally transverse for $p$, {$\tilde{H}$} and on $\Sigma_{_{\tilde{H},p}}$, there exists $j$ so that $H_p\tilde{x}_j\neq 0$. In particular, 
$$
dp,\,\{d\tilde{x}_i\}_{i=1}^k,\,\{d\xi'_i\}_{i=1}^{n-k}\text{ are linearly independent near }\SigH.
$$
 Thus, there exist $y_1,\dots, y_{n-1}\in C^\infty(\TM ;\re)$ so that $(p,\tilde{x},\xi',y)$ are coordinates on $\TM $ near $\Sigma_{_{{\tilde{H},p}}}$ {for which $\Sigma_{_{{\tilde{H},p}}}=\{p=0, \tilde x=0, \xi'=0\}$}. In particular, {there exists $C>0$ depending only on $(M,p, \KR_{_{0}})$ so that}
$$
d((x_0,\xi_0),\Sigma_{_{{\tilde{H},p}}})^2\leq C(p(x_0,\xi_0)^2+|\tilde{x}_0|^2+|\xi'_0|^2).
$$
We define the constant $K>0$ {introduced} in the definition \eqref{e:beta} of $\beta_\delta$ to be large enough so that 
\begin{align}\label{e:Kdef}
&\text{If}\quad d((x_0,\xi_0),\Sigma_{_{{\tilde{H},p}}})\geq {\tfrac{1}{2}}h^\delta,\quad {(x_0',\xi_0')} \in \supp \beta_\delta,\quad \text{and} \quad d(x,{\tilde{H}})\leq \tfrac{1}{K}h^\delta, \notag\\
&\hspace{4cm} \text{then} \;\;
|p(x_0, \xi_0)| \geq \tfrac{1}{3}h^\delta.
\end{align}}
As introduced in Step 1 in the proof of Theorem \ref{t:porcupine}, let $\chi_{_{0}} \in C_{{c}}^\infty(\R;[0,1])$ be a smooth cut-off function with $\chi_{_{0}}(t)=1$ for $t \leq \tfrac{1}{2}$ and $\chi_{_{0}}(t)=0$ for $t \geq {1}$.  Let 
$\beta_\class{(x',\xi')}$ be defined as in \eqref{e:beta}. {In what follows $\tau_1, \e_1, h_1$ are the positive constants given by Proposition \ref{l:nicePartition}.} 

{Our next proposition estimates the main contribution to averages. In particular, we control the average near zero frequency by the $L^2$ mass along bicharacteristics co-normal to the submanifold $H$. One of the main estimates used in the proof of Proposition~\ref{l:manyTubes} is found in Lemma~\ref{l:squirrel}. In particular, $p$ is factored as $e(x,\xi)(\xi_1-a(x,\xi'))$ so that it can be treated using elementary estimates. This idea comes from~\cite{KTZ} where, to the best of the authors' knowledge, it was first used to control $L^\infty$ norms. }
%%%%%%%%%%%%%%%%%%%%%%%%%%%%%%%%%%%%%%%%%%%%%%%%%%%%%%%%%%%%%%%%%%%%%%%%%%%%%%%%
\begin{proposition}
\label{l:manyTubes}
There exist constants $0<\tau_0 \leq \tau_1$, $0<\e_0 \leq \e_1$,  with  $\tau_0=\tau_0(M,p,\Tinj,\FR )$ and $\e_0=\e_0(\tau_0)$, {$R_0=R_0(M,p,k,\KR_{_{0}},\Tinj,\FR )>0$} and a constant
$C_{n,k}$ depending only on $n,k$,  and for each $0<\delta < \tfrac{1}{2}$ there exists $0<h_0\leq h_1$ so that the following holds.

Let   $0<\tau\leq \tau_0$, $0<\e<\e_0$, $ {4}h^\delta\leq R(h)\leq R_0$. Let $\mathfrak{D}_n$ be the constant from Proposition~\ref{l:cover}, {$0<h<h_0$}, and 
$\{ \Lambda_{_{\!\rho_{j}}}^\tau (R(h))\}_{j=1}^{N_h}$ be a $(\mathfrak{D}_n,\tau,R(h))$-good cover for $\SigH$. In addition, let  $\{\chi_j\}_{j=1}^{N_h}$ be the partition of unity built in Proposition~\ref{l:nicePartition}.

 Then, {there exists $C>0$ so that} for all {$N>0$ there is $C_{_{\!N}}>0$} with the following properties. For all   $w=w(x';h)\in S_\class \cap C_c^\infty({\tilde{H}})$, $0<h \leq h_0 $, and $u\in \mathcal{D}'(M)$,
\begin{align*}
{h^{\frac{k-1}{2}}}\|w Op_h(\beta_\class)u\|_{L^1({\tilde{H}})}
&\leq C_{n,k}{\|w\|_{_{\!\infty}}} R(h)^{\frac{n-1}{2}}\sum_{j\in \mathcal I_h(w)}\frac{\|Op_h(\chi_j)u\|_{\LM}}{\tau^{\frac{1}{2}}|H_pr_H(\rho_j)|^{\frac{1}{2}}}\\
&+{{C{h^{-1}}}}{\|w\|_\infty}\|Pu\|_{\Hm}+{C_{_{\!N}}h^N}{\|w\|_\infty}\|u\|_{\LM},
\end{align*}
where $\mathcal I_h(w)=\{j: \Tj\cap \pi^{-1}(\supp w)\neq \emptyset\}$. {Moreover the constants $C, C_{_{\!N}}, {h_0}$ are uniform for $\chi_j$ in bounded subsets of $S_\delta$, {uniform in $\tau,\e_0,\FR $  when these are bounded away from $0$,} {and uniform for $\KR_{_{\alpha}}$ bounded.}}
\end{proposition}

%%%%%%%%%%%%%%%%%%%%%%%%%%%%%%%%%%%%%%%%%%%%%%%%%%%%%%%%%%%%%%%%%%%%%%%%%%%%%%%%
\begin{proof}

 We define $\tau_0>0$, ${\e_0>0}$ to be the constants given by  Lemma \ref{L:optimized} below. {Let  $\chi_{_{0}} \in C_{{c}}^\infty(\R;[0,1])$ be a smooth cut-off function with $\chi_{_{0}}(t)=1$ for $t \leq \tfrac{1}{2}$ and $\chi_{_{0}}(t)=0$ for $t \geq {1}$.}
 We first decompose $\|wOp_h(\beta_\class)u\|_{\LMo}$ with respect to $\{\chi_j\}_{j=1}^{N_h}$. 
 {We write 
\[
Op_h( \beta_\delta)=
\Big[1-\chi_{_0}\Big(\frac{Kd(x,{\tilde{H}})}{h^\delta}\Big)\Big]Op_h(\beta_\delta) 
+\chi_{_0}\Big(\frac{Kd(x,{\tilde{H}})}{h^\delta}\Big)Op_h(\beta_\delta)\sum_{j=1}^{N_h}Op_h(\chi_j)
+Op_h(\chi)
\]
 with
 $$
Op_h(\chi)=\chi_{_0}\Big(\frac{Kd(x,{\tilde{H}})}{h^\delta}\Big)Op_h(\beta_\delta)\Big(1-\sum_{j=1}^{N_h}Op_h(\chi_j)\Big).
$$
First, note that 
$
\big[1-\chi_{_0}\big(\frac{Kd(x,{\tilde{H}})}{h^\delta}\big)\big]Op_h(\beta_\delta) u\big|_{{\tilde{H}}}\equiv 0.
$
Therefore, 
\begin{equation}\label{e:bounda}
\|Op_h( \beta_\delta)u\|_{L^1({\tilde{H}})}
\leq
\Big\|Op_h(\beta_\delta)\sum_{j=1}^{N_h}Op_h(\chi_j) u \Big\|_{L^1({\tilde{H}})}
+\| Op_h(\chi)u \|_{L^1({\tilde{H}})}.
\end{equation}
We first study the $\| Op_h(\chi)u \|_{L^1({\tilde{H}})}$ term.
To do this let $\psi\in C_c^\infty(\TM )$ be so that $|p(x,\xi)|\geq c|\xi|^m$ on $\supp (1-\psi)$. Then, by a standard elliptic parametrix construction (see e.g~\cite[Appendix E]{ZwScat}) together with the semiclassical Sobolev estimates (see e.g.~\cite[{Lemma 6.1}]{Gdefect}) there exist {$C>0$ and $0<h_0\leq h_1$ so that the following holds. For all $N$ there exists $C_{_{\!N}}>0$}  such that for all $0<h\leq h_0$
\begin{align*}
\|Op_h(1-\psi)Op_h(\chi) u\|_{L^2({\tilde{H}})}&\leq Ch^{-\frac{k}{2}}\|Op_h(1-\psi)Op_h(\chi)u\|_{\Hs{\frac{k+1}{2}}}\\
&\leq Ch^{-\frac{k}{2}}\|Pu\|_{\Hm}+{C_{_{\!N}} h^N}\|u\|_{\LM}.
\end{align*}
Together with} Lemma~\ref{l:naive squirrel} (below) applied to $\psi\chi$ {and the fact that $\|Pu\|_{\LM}\leq \|Pu\|_{\Hm}$}
this implies
\begin{equation}\label{e:boundb}
{\|Op_h(\chi)u  \|_{_{\!L^2({\tilde{H}})}}}\leq C{h^{-\frac{k}{2}-\delta}}\|Pu\|_{\Hm}+{C_{_{\!N}} h^N}\|u\|_{\LM}.
\end{equation}
Indeed, to see that Lemma~\ref{l:naive squirrel} applies, let $(x_0,\xi_0)\in \supp \psi\chi$. Then observe that $\supp \chi \subset \Big(\Lambda_{\SigH}^\tau({2}h^\class)\Big)^c$ and hence
$$
 d((x_0,\xi_0),\Sigma_{_{{\tilde{H},p}}})\geq h^\delta.
$$
Next, note that  $d((x_0,\xi_0),N^*\!{\tilde{H}})\leq \tfrac{1}{K}h^\class$ since $(x_0,\xi_0)\in \supp \beta_\delta$. 
{Therefore,  since $d((x_0,\xi_0),\Sigma_{_{{\tilde{H},p}}})\geq h^\class$, $d(x,{\tilde{H}})\leq \tfrac{1}{K}h^\class$, and $(x_0, \xi_0)\in \supp \beta_\delta$, by the definition \eqref{e:Kdef} of $K$ we obtain that} $|p(x_0,\xi_0)|\geq \frac{h^\class}{3}$ {for all $0<h\leq h_0$.} {To see that $|dp|>\tfrac{\FR}{2}>0$ on $\supp \psi \chi$, we observe that $|H_p|>\FR>0$ on $\SigH$.} 
It follows {from \eqref{e:bounda} and \eqref{e:boundb}} that
\begin{multline}\label{E:allspikes0}
\|wOp_h(\beta_\class)u\|_{\LMo}
\leq \!\Big\|\sum_{j=1}^{N_h} wOp_h(\beta_\class)Op(\chi_j)u\Big\|_{\LMo}\!
\\+C {\|w\|_\infty}{h^{-\frac{k}{2}-\delta}}\|Pu\|_{\Hm}\!\!+\! {C_{_{\!N}}h^N\|w\|_\infty }\|u\|_{\LM}.
\end{multline}

By Proposition~\ref{l:cover}, or more precisely its proof, there exist a collection of balls $\{B_i\}_{i=1}^{M_h}$ in ${\tilde{H}}$ of radius $R(h)\leq R_0(n,k,{\KR_{_{0}}})$ and constants $\alpha_{n,k}$ depending only on $n,k$, so that 
$$
{\tilde{H}}\subset \bigcup_{i=1}^{M_h} B_i
$$
and each $x\in {\tilde{H}}$ lies in at most $\alpha_{n,k}$ balls $B_i$. Let $\{\psi_i\}_{i=1}^{M_h}$ be a partition of unity on ${\tilde{H}}$ subordinate to $\{B_i\}_{i=1}^{M_h}$.
Then,  by \eqref{E:allspikes0}, {for all $0<h\leq h_0$,} 
\begin{equation}
\begin{aligned}\label{E:allspikes}
\|wOp_h(\beta_\class)u\|_{L^1({\tilde{H}})}
&\leq \sum_{i=1}^{M_h}\sum_{j=1}^{N_h} \|\psi_i wOp_h(\beta_\class) Op(\chi_j)u\|_{L^1({\tilde{H}})}\\
&\qquad+C{h^{-\frac{k}{2}-\delta}}{\|w\|_{\infty}}\|Pu\|_{\Hm}\!\!+\!{C_{_{\!N}}h^N\|w\|_\infty}\|u\|_{\LM}.
\end{aligned}
\end{equation}
%Thus,  if $b\in S_\class$ is so that $\supp b\subset \Lambda_{\SigH}^T(2h^\class)$,  then
%$$
%\WFh\Big(Op_h(b)[P,Op_h(\chi_j)]\Big)=\emptyset.
%$$
We next note that on ${\tilde{H}}$, the volume of a ball of radius $r$ satisfies
$$
|\vol_{{\tilde{H}}}(B(x,r))-c_{n,k}r^{n-k}|\leq {C_{_{\!\KR_{_{0}}}}}r^{n-k+1}
$$
where {$C_{_{\!\KR_{_{0}}}}>0$} is a constant depending only on ${\KR_{_{0}}}$ {and $c_{n,k}$ is a constant that depends only on $(n,k)$}, (this can be seen by working in geodesic normal coordinates). Therefore, for some $c_{n,k}>0$ and {any} $R(h)\leq R_0=R_0({\KR_{_{0}}})$ 
\begin{equation}\label{E:l1bound}
 \|\psi_i wOp_h(\beta_\class) Op(\chi_j)u\|_{L^1({\tilde{H}})} \leq c_{n,k}R(h)^{\frac{n-k}{2}} \|\psi_i w Op_h(\beta_\class) Op(\chi_j)u\|_{L^2({\tilde{H}})}.
 \end{equation}
 
We next bound $\|\psi_i wOp_h(\beta_\class) Op(\chi_j)u\|_{L^2({\tilde{H}})}$. By Lemma~\ref{L:optimized} below there exist $C_{n,k}>0$ depending only on $(n,k)$, and $C>0$ {so that the following holds. For every $\tilde N>0$ there exists $C_{_{\! \tilde N}}>0$, independent of $(i,j)$,} so that for all $0<h\leq h_0$ 
\begin{multline}\label{E:spike1}
\|\psi_iwOp_h(\beta_\class)Op_h(\chi_j)u\|_{L^2({\tilde{H}})}\\
\leq C_{n,k}{\|w\|_{_{\!\infty}}}h^{\frac{1-k}{2}}\!R(h)^{\frac{k-1}{2}}\!\!\left(\frac{\|Op_h(\chi_j)u\|_{\LM}}{\tau^{\frac{1}{2}}|H_pr_H(\rho_j)|^{\frac{1}{2}}}\!+\!Ch^{-1}\|Op_h({\chi_j})Pu\|_{\LM}\!\right)\\+\!C_{_{\! \tilde N}}h^{\tilde N}\|w\|_\infty\|u\|_{\LM}.
\end{multline}

Also, note that  if $j\notin \mathcal I_h({\psi_i w})$ for some $i\in \{1, \dots, M_h\}$, then 
$$
\Tj \cap \pi^{-1}(\supp {\psi_i w})=\emptyset.
$$ Therefore, since $\supp \chi_j \subset \Tj$ {for all $j$}, {for all $N'$ there exists $C_{_{N'}}>0$ so that the following holds.} For all $i\in \{1, \dots, M_h\}$ and  $j\notin \mathcal I_h({\psi_i w})$
$$
\|\psi_i wOp_h(\beta_\class)Op_h(\chi_j)u\|_{L^2({\tilde{H}})}\leq C_{_{\! N'}}h^{N'}\|w\|_\infty\|u\|_{\LM}.
$$
In particular, since $N_h$ {and $M_h$ grow} like a polynomial power of $h$, {we can choose $N'$ so that}
\begin{equation}\label{E:spike2}
{\sum_{i=1}^{M_h}}\sum_{j\notin I_h({\psi_i w})}\|\psi_i wOp_h(\beta_\class)Op_h(\chi_j)u\|_{L^2({\tilde{H}})}\leq {C_{_{\!N}}h^N\|w\|_\infty}\|u\|_{\LM}.
\end{equation}

Putting \eqref{E:l1bound}, \eqref{E:spike1} and \eqref{E:spike2} into  \eqref{E:allspikes},  we find that for some adjusted $C_{n,k}$ {and $0<h\leq h_0$} 
\begin{align*}
&\|wOp_h(\beta_\class)u\|_{L^1({\tilde{H}})}\\
&\leq C_{n,k}{\|w\|_{_{\!\infty}}}h^{\frac{1-k}{2}}R(h)^{\frac{n-1}{2}}\sum_{i=1}^{M_h}\sum_{j\in I_h(\psi_iw)}\!\!\left(\frac{\|Op_h(\chi_j)u\|_{\LM}}{\tau^{\frac{1}{2}}|H_pr_{_{\!H}}(\rho_j)|^{\frac{1}{2}}}+Ch^{-1}\|Op_h({\chi_j})Pu\|_{\LM}\!\!\right)\\
&\qquad +C{h^{-\frac{k}{2}-\delta}}{\|w\|_\infty}\|Pu\|_{\Hm}+{C_{_{\!N}}h^N\|w\|_\infty}\|u\|_{\LM}.
\end{align*}
We have used that both $M_h$ and $N_h$ grow like a polynomial power of $h$ {to collect all the $C_{_{\! \tilde N}}h^{\tilde N}\|u\|_{\LM}$  error terms in \eqref{E:spike1}}. Furthermore, since the balls $\{B_i\}$ are built so that every point in ${\tilde{H}}$ lies in at most $\alpha_{n,k}$ balls, and each $\psi_i$ is supported on $B_i$, we have 
\begin{align}\label{e:superPorcupine}
&\|wOp_h(\beta_\class)u\|_{L^1({\tilde{H}})} \notag\\
&\leq C_{n,k}{\|w\|_{_{\!\infty}}}h^{\frac{1-k}{2}}R(h)^{\frac{n-1}{2}}\sum_{j\in \mathcal I_h(w)}\left(\frac{\|Op_h(\chi_j)u\|_{\LM}}{\tau^{\frac{1}{2}}|H_pr_H(\rho_j)|^{\frac{1}{2}}}+Ch^{-1}\|Op_h({\chi_j})Pu\|_{\LM}\right) \notag\\
&\qquad +C{{h^{-\frac{k}{2}-\delta}}\|w\|_\infty}\|Pu\|_{\Hm}+{C_{_{\!N}}h^N\|w\|_\infty}\|u\|_{\LM}.
\end{align}
Now, {since ${\chi_j}$ is supported in $\Tj$, and the tubes were built so that every point in $\LambdaH(h^\class)$ lies in at most $\beta_{n,k}$ tubes,  we have}  
$
\sum_{j=1}^{N_h} |{\chi_{j}}|^2\leq \beta_{n,k}.
$
This implies
$$
\sum_{j=1}^{N_h} \|Op_h({\chi_j})Pu\|_{\LM}^2\leq 2\beta_{n,k}\|Pu\|_{\LM}^2.
$$
Next, notice that since $\dim \SigH=n-1$, we have  $|\mathcal I_h(w)|\leq c_{n,k}R(h)^{1-n}\vol(\SigH)$ for some $c_{n,k}>0$ depending only on $n,k$. Therefore,  
\begin{align*}
\sum_{j\in \mathcal I_h(w)} \|Op_h({\chi_j})Pu\|_{\LM}&\leq |\mathcal I_h(w)|^{\frac{1}{2}}\Big(\sum_{j=1}^{N_h} \|Op_h({\chi}_j)Pu\|_{\LM}^2\Big)^{\frac{1}{2}}\\
&\leq c_{n,k}R(h)^{-\frac{n-1}{2}}\vol(\SigH)^{\frac{1}{2}}\|Pu\|_{\LM},
\end{align*} 
for some $c_{n,k}>0$ depending only on $n,k$.
Using this in~\eqref{e:superPorcupine} {together with $\delta<\frac{1}{2}$,} gives
\begin{align*}
\|wOp_h(\beta_\class)u\|_{L^1({\tilde{H}})}
&\leq C_{n,k}{\|w\|_{_{\!\infty}}}h^{\frac{1-k}{2}}R(h)^{\frac{n-1}{2}}\sum_{j\in \mathcal I_h(w)}\frac{\|Op_h(\chi_j)u\|_{\LM}}{\tau^{\frac{1}{2}}|H_pr_H(\rho_j)|^{\frac{1}{2}}}\\
&+C{h^{-\frac{1+k}{2}}}{\|w\|_\infty}\|Pu\|_{\Hm}+{C_{_{\!N}}h^N\|w\|_\infty}\|u\|_{\LM},
\end{align*}
as claimed. Note that the constants $C, C_{_{\!N}}, {h_0}$ are uniform for $\chi_j$ in bounded subsets of $S_\delta$, and are also  uniform in $\tau,\e_0,\FR $  when these are bounded away from $0$. {Furthermore, they depend only on finitely many of the constants $\KR_\alpha$.}

\end{proof}
%%%%%%%%%%%%%%%%%%%%%%%%%%%%%%%%%%%%%%%%%%%%%%%%%%%%%%%%%%%%%%%%%%%%%%%%%%%%%%%%

%%%%%%%%%%%%%%%%%%%%%%%%%%%%%%%%%%%%%%%%%%%%%%%%%%%%%%%%%%%%%%%%%%%%%%%%%%%%%%%%

We now state the following result {which gives elliptic estimates in regions that are $h^\delta$ away from the characteristic variety of $p$.} 

%%%%%%%%%%%%%%%%%%%%%%%%%%%%%%%%%%%%%%%%%%%%%%%%%%%%%%%%%%%%%%%%%%%%%%%%%%%%%%%%
{\begin{lemma}
\label{l:naive squirrel}
Let $0\leq \delta<\frac{1}{2}$, $0<k<n$. {Let ${\Theta:W\subset \re^n\to M}$ be coordinates on $M$.} Let $\chi\in S^{{\comp}}_\delta\cap C_c^\infty(\TM ;[-C_0h^{1-2\delta},1+C_0h^{1-2\delta}])$ be so that there exist $c, h_1>0$  with
\[ 
\supp \chi\subset \{|p|\geq c h^\delta\;,\;{|p|+|dp|>c}\}
\]
for  $0<h\leq h_1$.
{Then, {there exists $C>0$ such that} for all $\tilde \chi \in S_\class\cap C_c^\infty(\TM ;[0,1])$ with $\tilde \chi \equiv 1$ on $\supp \chi$, there exists $0<h_0<h_1$ so that the following holds. For all $N>0$ there exists $C_{_{\!N}}>0$  such that  for $0<h<h_0$}
\[
\|Op_h(\chi)u \|_{{L^\infty_{\bar{x}}L^2_{ x'}}}\leq Ch^{{-\frac{k}{2}}-\class}\|Op_h(\tilde{\chi})P{u}\|_{{L^2_x}} +{C_{_{\!N}}h^N}\|u\|_{L^2_x},
\] 
{where $x=(x', {\bar x})\in \R^{n-k}\times \R^k$ are the coordinates induced by $\Theta$.}
Moreover, $C,C_{_{\!N}}$ are uniform for $\tilde{\chi},\,\chi$ in bounded subsets of $S_\delta$, and {for} $\Theta$ in bounded subsets of $C^\infty$.
 \smallskip
\end{lemma}}
%%%%%%%%%%%%%%%%%%%%%%%%%%%%%%%%%%%%%%%%%%%%%%%%%%%%%%%%%%%%%%%%%%%%%%%%%%%%%%%%

\begin{proof}
 {First, let $\psi\in C_c^\infty(\re)$ with $\psi\equiv 1$ on $[-1,1]$. Then, using the standard elliptic parametrix construction~\cite[Appendix E]{ZwScat} there exists $b_1\in S^{\comp}_\delta$ with $\sup |b_1|\leq 2c^{-1}+C_1h^{1-2\delta}$ such that 
  \begin{equation}
  \label{e:realElliptic}
  Op_h(\chi)Op_h(1-\psi\big(\tfrac{2}{c}p\big) )=Op_h(b_1)Op_h(\tilde{\chi})P+O(h^\infty)_{\Psi^{-\infty}}.
  \end{equation}}
  
  {Next, we show that there exists $b_2\in S^{\comp}_\delta$ with $\sup|b_2|\leq c^{-1}h^{-\delta}+C_1h^{1-3\delta}$ so that 
  \begin{equation}
  \label{e:elliptic}
  Op_h(\chi)Op_h(\psi\big(\tfrac{2}{c}p\big))=Op_h(b_2)Op_h(\tilde{\chi})P+O(h^\infty)_{\Psi^{-\infty}}.
  \end{equation}}
  Using that $ |p|\geq c h^\delta$ on $\supp \chi$ one can carry out an elliptic parametrix construction in the second microlocal calculus associated to $p=0$. Using a partition of unity, since $|dp|>\frac{c}{2}$ on $\supp\chi\cap \supp \psi\big(\tfrac{2}{c}p\big)$ we may assume that {there exist} an $h$-independent neighborhood $V_0$ of $\supp \chi$, $V_1\subset T^*\re^n$ a neighborhood of $0$, {and a symplectomorphism} $\kappa:V_1\to V_0$ so that $\kappa^*p=\xi_1$.  Let $U$ be a microlocally unitary FIO quantizing $\kappa$. Then
\[
\mathbf{P}:={U^*}P{U}=hD_{x_1}+h {Op^L_h}({\bf{r}}),
\]
with ${\bf{r}}\in S^{\comp}(\re^n)$ {and $Op_h^L$ denotes the \emph{left} quantization of $\bf{r}$.}
Moreover, there exist $\mathbf{a},\mathbf{\tilde{a}}\in S_\delta^{\comp}(T^*\re^n)$ so that 
$$
 {Op^L_h}(\mathbf{a})=U^*Op_h(\chi){Op_h(\psi\big(\tfrac{2}{c}p\big))} U
$$
and 
$$
 {Op^L_h}(\mathbf{\tilde{a}})=U^*Op_h(\tilde{\chi}) U\
$$ 
with $\supp\mathbf{a}\subset \{|\xi_1|\geq ch^\delta\}$ and $\mathbf{\tilde{a}}\equiv 1$ on $\supp \mathbf{a}$. 
Now, for ${\bf{b}}\in S^{\comp}_\delta(T^*\re^n)$ supported on $|\xi_1|\geq ch^\delta$, 
$$ |\partial_x^\alpha \partial_{\xi}^\beta (\xi_1^{-1}{\bf{b}})|\leq  C_{\alpha\beta}h^{-(|\beta|+|\alpha|)\delta}|\xi_1|^{-1}.$$

Let ${\bf{b}}_0=\mathbf{a}/\xi_1.$ Then ${\bf{b}}_0\in h^{-\delta}S_{\delta}^{\comp}$ and
$$
\sup |{\bf{b}}_0|\leq c^{-1}h^{-\delta}.
$$ 
Observe that
$$
 {Op^L_h}({\bf{b}}_0){Op_h^L}(\mathbf{\tilde{a}})\mathbf{P}= {Op^L_h}(\mathbf{a})+ {Op^L_h}({\bf{e}}_1)+O(h^\infty)_{{\Psi^{-\infty}}}
$$
with $\supp {\bf{e}}_1\subset \{|\xi_1|\geq ch^\delta\}$ and, since $\mathbf{\tilde{a}}\equiv 1$ on $\supp {\bf{b}}_0$, 
$$
{\bf{e}}_1\sim \sum_{{|\alpha|}\geq 1} \frac{h^{{|\alpha|}}i^{{|\alpha|}}}{\alpha!}D^{{\alpha}}_x({\bf{b}}_0)D_{\xi}^{{\alpha}}(\xi_1)+\sum_{{{|\alpha|}}\geq 0}\frac{h^{{{|\alpha|}}+1}i^{{|\alpha|}}}{k!}D^{{\alpha}}_x({\bf{b}}_0)D_{\xi}^{{\alpha}}({\bf{r}}).
$$
In particular, ${\bf{e}}_1\in h^{1-2\delta}S^{\comp}_\delta$. Then, setting ${\bf{b}}_{\ell}=-{\bf{e}}_{\ell}/\xi_1\in h^{\ell(1-2\delta)-\delta}S_{\delta}^{\comp}$, and
$$
 {Op^L_h}({\bf{e}}_{\ell+1}):= {Op^L_h}({\bf{b}}_{\ell}) {Op^L_h}(\mathbf{\tilde{a}})\mathbf{P} {+} {Op^L_h}({\bf{e}}_{\ell}) +O(h^\infty)_{{\Psi^{-\infty}}}
$$
we have ${\bf{e}}_{\ell+1}\in h^{(\ell+1)(1-2\delta)}S^{\comp}_\delta$ with $\supp {\bf{e}}_{\ell+1}\subset \{|\xi_1|\geq ch^\delta\}$.
In particular, putting ${\bf{b}}\sim \sum_\ell {\bf{b}}_{\ell}$, 
$$
 {Op^L_h}({\bf{b}}) {Op^L_h}(\mathbf{\tilde{a}})\mathbf{P}= {Op^L_h}(\mathbf{a})+O(h^\infty)_{{\Psi}^{-\infty}}.
$$

It follows that
\begin{align*}
U{Op_h^L({\bf{b}})}U^*Op_h(\tilde{\chi})P&=U {Op^L_h}(\mathbf{b})U^*U {Op^L_h}(\mathbf{\tilde{a}})U^*U\mathbf{P}U^*+O(h^\infty)_{{\Psi^{-\infty}}}\\
&=U {Op^L_h}({\bf{b}}) {Op^L_h}(\mathbf{\tilde{a}})\mathbf{P}U^*+O(h^\infty)_{{\Psi^{-\infty}}}\\
&=U  {Op^L_h}(\mathbf{a})U^*+O(h^\infty)_{{\Psi^{-\infty}}}\\
&=Op_h(\chi){Op_h(\psi\big(\tfrac{2}{c}p\big))}+O(h^\infty)_{{\Psi^{-\infty}}}.
\end{align*}
In particular, there exists ${b_2}\in h^{-\delta}S^{\comp}_{\delta}(\TM )$ with {$\sup| {b_2}|\leq c^{-1}h^{-\delta}+{C_1h^{1-3\delta}}$} so that
$$
Op_h({b_2})=U{Op_h^L({\bf{b}})}U^*+O(h^\infty)_{{\Psi^{-\infty}}}.
$$
Therefore, {as claimed in~\eqref{e:elliptic}} that
$$
 Op_h(\chi){Op_h(\psi\big(\tfrac{2}{c}p\big))}=Op_h({b_2})Op_h(\tilde{\chi})P+{O(h^\infty)_{{\Psi^{-\infty}}}},
$$
for all $\chi$ supported in $V_0$ and some suitable ${b_2}$ with $\|Op_h({b_2})\|\leq 2c ^{-1}h^{-\class}$. Next, using that $Op_h(\tilde{\chi})Pu$ is compactly microlocalized, we apply the Sobolev Embedding~\cite[{Lemma 6.1}]{Gdefect} (see also~\cite[Lemma 7.10]{EZB}) in the $\bar{x}$ coordinates. {Writing $b=b_1+b_2$,} we obtain {using~\eqref{e:realElliptic} and~\eqref{e:elliptic}} {that there exists $h_0>0$, and for all $N>0$ there exists $C_{_{\!N}}>0$, such that if $0<h<h_0$, then for every $\bar x$}
\begin{align*}
 \| Op_h(\chi){u}(\bar{x},\cdot)\|_{{L^2_{ x'}}}
&=  \| Op_h(b)Op_h(\tilde{\chi})Pu(\bar{x},\cdot)\|_{L^2_{ x'}}+   {{C_{_{\!N}}h^N}\|u\|_{{L^2_{x}}}}\\
&\leq 2c ^{-1}C_{k}h^{{-\frac{k}{2}}-\class} \|Op_h(\tilde{\chi})Pu\|_{L^2_x}+  {C_{_{\!N}}h^N}\|u\|_{L^2_x}.
\end{align*}
Since this is true for any $\bar{x}$, the claim follows.
\end{proof}

{The following lemma contains the key new ideas used to prove our main theorems. In particular, it converts quantitative localization along a bichacteristic into quantitative gains in averages. This idea is at the heart of the bicharacteristic beam techniques and originated in~\cite{Gdefect}. }
%%%%%%%%%%%%%%%%%%%%%%%%%%%%%%%%%%%%%%%%%%%%%%%%%%%%%%%%%%%%%%%%%%%%%%%%%%%%%%%%
\begin{lemma}\label{L:optimized}
There exist $C_{n,k}>0$, depending only on $n$ and $k$, and positive constants $\tau_0=\tau_0(M,p,\Tinj,\FR,{\{H_h\}_h} )$, $\e_0=\e_0(\tau_0)$, $R_0=R_0(M,p,k,\Tinj,\FR )$ so  that the following holds. Let  $0<\tau\leq \tau_0$, $0\leq \class<\frac{1}{2}$, and $2h^\delta\leq R(h)\leq R_0$.  Let $\gamma$ be a  bicharacteristic through $\SigH$, and $\chi \in S_\class\cap C^\infty_c(\TM ;[-C_1h^{1-2\delta},1+C_1h^{1-2\delta}])$ with $\rho_\gamma:=\gamma \cap \SigH \in \supp \chi$,
\begin{equation}\label{E:support}
\begin{gathered}
\supp (\chi)\subset {\Lambda_{\rho_\gamma}^{\tau+\e_0}}(R(h)),
\end{gathered}
\end{equation} 
and 
\begin{equation}
\label{e:commuting}
\MSh([P,Op_h(\chi)])\cap \LambdaH({\e_0})=\emptyset.
\end{equation}

Then, there are $C>0$ and $h_0>0$ with the following properties. For every $N>0$ there exists $C_{_{\!N}}>0$ such that, if $0<h\leq h_0$, then for  $u\in \mc{D}'(M)$,
\begin{equation}
\label{e:optimized}
\begin{aligned}
&h^{k-1}\|Op_h(\beta_\class) Op_h(\chi)u\|_{L^2({\tilde{H}})}^2\leq
{C_{n,k}}\,\frac{R(h)^{k-1}}{\tau|H_pr_H(\rho_\gamma)|}\|Op_h(\chi) u\|^2_{\LM}\\
&\hspace{6cm}+{C}R(h)^{k-1}h^{-2}\|Op_h(\chi)Pu\|^2_{\LM}\\
&\hspace{6cm}+{C_{_{\!N}}}h^N {\|u\|^2_{\LM}},
\end{aligned}
\end{equation}
The constants {${\tau_0},C,C_{_{\!N}},h_0$} are uniform for $\chi$ in  bounded subsets of $S_\class$, uniform for $\tau>0$ and $\FR$ uniformly bounded away from zero, {and only depend on $\{H_h\}_h$ through finitely many of the constants $\KR_{_{\alpha}}$ in \eqref{e:curvature}}. 
\end{lemma}
%%%%%%%%%%%%%%%%%%%%%%%%%%%%%%%%%%%%%%%%%%%%%%%%%%%%%%%%%%%%%%%%%%%%%%%%%%%%%%%%

\begin{proof}

The proof of this result relies heavily on Lemma \ref{l:squirrel} below. {Let $\Theta:W \subset \R^n \to M$ be coordinates on $M$.} {Let $h>0$.} Note that we may adjust coordinates so that 
$\tilde{H}={\tilde{H}_h}\subset \{x_1=0\}$, 
$dx_1|_{x_1=0}\in N^*\!\tilde{H}$,
${\tfrac{1}{2}}H_pr_H{\leq}\partial_{\xi_1}p$,
{and so that the $C^k$ norm of the coordinate map $\Theta$ is bounded by finitely many of the constants $\KR_\alpha$}. Therefore, since {$ |\partial_{\xi_1}p(\rho_{\gamma})| \geq {\tfrac{1}{2}}\FR$} {by \eqref{e:FR}}, we may apply Lemma \ref{l:squirrel} with  $\mathfrak{I}:= {\tfrac{1}{2}}\FR $. Let $r_0,\tilde{\tau}_0, C_0$, depending only on $(M,{p},\FR,{\Theta})$, be the constants from Lemma~\ref{l:squirrel}. {Note that they are uniform for $\Theta$ in bounded sets of $C^k$. Therefore, they depend on ${\{H_h\}_h}$ through finitely many of the constants $\KR_\alpha$.}  Next, let $r_1=r_1(M,p,\FR,{\Theta})$ be small enough so that for all $\rho\in \SigH$, 
\begin{equation}
\label{E:control}
\frac{\inf_{B(\rho,r_1)}|H_pr_H|}{\sup_{B(\rho,r_1)}|H_pr_H|}\geq \frac{1}{2}.
\end{equation}
Let $r=\frac{1}{2}\min\{r_1,r_0\}$ and let $\{{\rho_i}\}_{i=1}^K\subset \SigH$ be a maximal $r$ separated set. Then for all $q\in \SigH$, there exists $i$ so that $d(q,{\rho_i})<r$ and in particular, $B(q,r)\subset B({\rho_i},2r)\subset V_{{\rho_i}}$ where $V_{{\rho_i}}$ is the subset from Lemma~\ref{l:squirrel} associated to ${\rho_i}$. 

{Fix $\rho_0 \in \{{\rho_i}\}_{i=1}^K$.}
Without loss of generality assume that $d(\rho_\gamma,{\rho_0})<r$. Next, let $0<\tilde{\tau}_1<\frac{\Tinj}{2}$, $R_0>0$, $\e_0>0$ small enough (depending only on $(M,P,\FR,\Tinj)$) so that 
$\Lambda_{\rho_\gamma}^{\tilde{\tau}_1+\e_0}(R_0) \subset {V_{\rho_0}}$. Next, by letting  
\begin{equation}\label{E:tau_0}
\tau_0=\min\{\tilde \tau_0, \tilde \tau_1\}
\end{equation}
we have
\[
\supp (\chi)\subset \Lambda_{\rho_\gamma}^{\tau+\e_0}(R(h)) \subset {V_{\rho_0}},
\]
for all $0<\tau<\tau_0$ and $h$ small enough. This will allow us to apply Lemma \ref{l:squirrel} to our $\chi$.
 
We work in coordinates so that $\partial_{\xi_1}p (\rho_\gamma) \neq 0$, {which we can assume since $\gamma$ is a bicharacteristic through $\SigH$ and $\rho_\gamma=\gamma \cap \SigH$}. In what follows we abuse notation slightly and redefine $\bar{x}$ as the normal coordinates to ${\tilde{H}}$ that are not $x_1$. With this notation $x=(x_1, \bar{x}, x')$. 

Given a function $v_h \in C^\infty(M)$ we may bound $\|v_h\|_{\LM}$  using  {the version of the Sobolev Embedding Theorem given in} \cite[Lemma 6.1]{Gdefect} which gives, after setting $k=\ell$, that for all $\alpha>0$ there exists $C_{k}>0$ depending only on $k$ so that 
\begin{equation}
\label{e:centipede}
\|v_h(x_1,\bar{x}, \cdot)\|^2_{L^2_{x'}}\leq C_{k}  h^{1-k} \left( \alpha^{k-1}\|v_h(x_1,\cdot)\|^2_{L_{\bar{x},x'}^2}+ \alpha^{-1-k} {\sum_{i=2}^{k}}\|(hD_{x_i})^k v_h(x_1,\cdot)\|^2_{L_{\bar{x},x'}^2}\right).
\end{equation}
We proceed to choose $v_h$ so that 
\begin{equation}\label{E:trick}
\|Op_h(\beta_\class) (Op_h(\chi )u)(x_1,\bar{x},\cdot)\|_{_{L^2_{x'}}} = \|v_h(x_1,\bar{x},\cdot)\|_{_{{L^2_{x'}}}},
\end{equation}
 and in such a way that the terms in \eqref{e:centipede} can be controlled efficiently.
{Let $0<\tau<\tau_0$, and set $\tau_{\rho_0}:=\tau |\partial_{\xi_1}p(\rho_0)|$.}

  Since $\gamma$ is a bicharacteristic through $\SigH$, we may define a function $a=a(x_1)$ so that $\xi-a(x_1)$ vanishes along $\gamma$. This is possible since we are working in coordinates so that $\partial_{\xi_1}p (\rho_\gamma) \neq 0$, and hence $\gamma$ may be locally written (near $\rho_\gamma$) as $\gamma(x_1)=(x(x_1),a(x_1))$ for $a$ and $x$ smooth.

  Define 
\[\kappa(x,\xi)={\chi_{_{0}}\Big(\frac{|(x_1,\bar{x})|}{{\e_0^2}}\Big)}{\chi_{_{0}}\Big(\frac{3|x_1|}{{\tau_{\rho_0}}}\Big)}\; {\beta_\class(x',\xi')},\]
{where $\e_0<1$ is so that the coordinates are well defined if $|(x_1,\bar{x})|< \e_0$.}  
Let
\[
v_h:=e^{-\frac{i}{h}\langle \bar{x}\,,\,\bar{a}(x_1)\rangle}Op_h(\kappa) Op_h(\chi) u ,
\]
where $\bar{a}(x_1)=(a_{2}(x_1),\dots,a_{k}(x_1))$ is so that $a(x_1)=(a_1(x_1), \bar a(x_1))$. The reason for working with this function $v_h$ is that not only \eqref{E:trick} is satisfied, but also
\[
(hD_{x_i})^k  v_h=  {e^{-\frac{i}{h}\langle \bar{x}\,,\,\bar{a}(x_1)\rangle}} (hD_{x_i}-a_{i})^k (Op_h(\kappa)Op_h( \chi ) u),
\]
for $i=2,\dots,k$, and this will allow us to obtain a gain in the $L^2$-norm bound once we use that,  by Lemma~\ref{l:flow}, for $(\tau_0,\e_0)$ small enough (depending only on $p$), 
\begin{equation}\label{E:R}
\sup_{{\Lambda_{\rho_\gamma}^{\tau_0+\e_0}}(R(h))} \max_i |\xi_i-a_{i}(x_1)| \leq 3R(h).
\end{equation}

We bound the terms in \eqref{e:centipede} by applying Lemma~\ref{l:squirrel} with $\kappa$ and {$\chi$}. 
% In order to apply the lemma we define  $b\in S_0\cap C_c^\infty(M;[0,1])$ with  
%$$
%b\equiv 1\;\; \text{on}\;\;\ |x_1|\leq  \qquad \text{and}\qquad \supp b\subset |x_1|<.
%$$
%Then by~\eqref{E:support},
%%\begin{equation}\label{E:support2}
%%\supp (\chi )\cap \supp b\;\; \subset   \Lambda_{\gamma}^T( R(h)).
%%\end{equation} 
We first bound the non-derivative term on the RHS of \eqref{e:centipede}.

{By~Lemma \ref{l:squirrel} we have that ${\inf_{V_{\rho_0}}}|\partial_{\xi_1}p|\geq \frac{3}{4}|\partial_{\xi_1}p(\rho_0)|$  on $\Lambda_{\rho_\gamma}^{\tau+\e_0}(R(h))$. This implies
\begin{equation}
\label{e:quince}
\Big(\Lambda_{\rho_\gamma}^{\tau+\e_0}(R(h))\cap (\LambdaH(\e_0))^c\Big)\subset \{|x_1| {\geq} \tfrac{3}{4}\tau_{\rho_0}\}.
\end{equation}}
Let $b\in C_c^\infty( \re;[0,1])$ with $b\equiv 1 $ on $\{x_1:\, |x_1|\leq \tau_{\rho_0}/2\}$, $\supp b\subset \{x_1:\, |x_1|<3\tau_{\rho_0}/4\}$. 
{By ~\eqref{E:support} and~\eqref{e:commuting} we have 
 $\MSh([P,Op_h(\chi)]) \subset (\Lambda_{\rho_\gamma}^{\tau+\e_0}(R(h))\cap (\LambdaH(\e_0))^c)$.} Therefore, by~\eqref{e:quince},
\begin{equation}
\label{e:blackberry}
\WFh(b)\cap \MSh([P,Op_h(\chi)])=\emptyset.
\end{equation}

Throughout the rest of the proof we will write $C,C_{_{\!N}}$ for constants that are uniform as claimed. We also note that when bounding $\|Op_h(a)u\|_{\LM}$ by $2\sup|a|\|u\|_{\LM}$, $h$ need only be taken small enough depending on finitely many seminorms of $a$ in $S_\delta$.
Let $C_0=C_0(M,P,\FR )$ as above and $\tau_0$ as in \eqref{E:tau_0}.
 Applying Lemma~\ref{l:squirrel} with $\kappa$, $\chi$, $b$, $q=1$, and using that  $b\equiv 1$ on $|x_1|\leq \tau_{\rho_0}/2$, $\| Op_h(\kappa)\|\leq 2$  and $0<\tau<\tau_0$, we have that {there exists $h_0>0$ such that for all $0<h<h_0$}
 \begin{equation}
\begin{aligned}\label{E:v}
{\|v_h(x_1,\cdot)\|}_{L_{\bar{x},x'}^2}&\leq  {8}\tau_{\rho_0}^{-\frac{1}{2}}\| b\,Op_h(\chi )u\|_{\LM} +{2C_0\tau_{\rho_0}^{\frac{1}{2}}}h^{-1}\|b\,POp_h(\chi ) u\|_{\LM} \\
&\qquad+{C_{_{\!N}}h^N} \|u\|_{\LM}.
\end{aligned}
\end{equation}

Next, note that  
$$
b\,POp_h(\chi)= b\,Op_h(\chi)P+b\,[P,Op_h(\chi)].
$$
 Therefore, since $|b|\leq 1$,
\begin{equation*}
\|b\,POp_h(\chi)u\|_\LM\leq 2\|Op_h(\chi)Pu\|_\LM+\|{b\,[P,Op_h(\chi)]}u\|_\LM.
\end{equation*}

Using the previous bound, equation \eqref{E:v} turns into
\begin{equation}
\begin{aligned}\label{E:v2}
{\|v_h(x_1,\cdot)\|}_{L_{\bar{x},x'}^2}
&\leq  {{16}\tau_{\rho_0}^{-\frac{1}{2}}}\|Op_h(\chi )u\|_{\LM}+{4}C_0\tau^{\frac{1}{2}}_{\rho_0}h^{-1}\|Op_h(\chi )P u\|_{\LM} \\
&\qquad+{2}C_0\tau^{\frac{1}{2}}_{\rho_0}h^{-1}\|{b\,[P,Op_h(\chi)]}u\|_{\LM}  +{{C_{_{\!N}}h^N} \|u\|_{\LM}}.
\end{aligned}
\end{equation}

We proceed to bound the derivative terms in \eqref{e:centipede}. For this, we first note that  $\|(hD_{x_i})^k v_h(x_1,\cdot)\|_{L_{\bar{x},x'}}=\|Q_i Op_h(\kappa)Op_h( \chi ) u(x_1,\cdot)\|_{L_{\bar{x},x'}} $ after setting
\begin{equation}\label{e:Q_i}
Q_i:=(hD_{x_i}-a_i)^k,
\end{equation}
{for $i=2, \dots, k$}.
Writing $Q_i=Op_h(q_i)$ we get  ${q_i=(\xi_i-a_i)^k}$ and $Q_i$ commutes with $Op_h(\kappa)$ modulo $O(h)$. {Note that there are no remainder terms since $a_i$ is a function of only $x_1$.} Then,  Lemma \ref{l:squirrel} gives that there exists $C_0>0$, independent of $\tau$,  and some $C,{C_{_{\!N}}}>0$ so that
\begin{equation}
\begin{aligned}\label{E:D} 
{\|(hD_{x_i})^k v_h(x_1,\cdot)\|}_{L_{\bar{x},x'}^2}
&\leq {{8}\tau_{\rho_0}^{-\frac{1}{2}}}\| b\,{{Q_i}Op_h(\chi )}u\|_{\LM}  +{{2}C_0\tau^{\frac{1}{2}}_{\rho_0}}h^{-1}\|b\,P{{Q_i}Op_h(\chi )}u\|_{\LM}  \\
& +\|[{Op_h}(\kappa),Q_i]Op_h(\chi){u(x_1,\cdot)}\|_{{L^2_{\bar{x},x'}}} +{C_{_{\!N}}h^N} \|u\|_{\LM},
\end{aligned}
\end{equation}
{for all $0<h<h_0$ where $h_0$ was possibly adjusted}.
We proceed to find efficient bounds for all the  terms in \eqref{E:D}. 
{Throughout the rest of the proof we use $C_0$ for a positive constant that depends only on {$P$ and finitely may $S_\delta$ seminorms of $(q,\chi)$}, possibly bigger than that above. We also write  $C_k$ for a positive constant that depends only on $k$. These constants may increase from line to line.}

First, let $\tilde{\chi}\in S_\delta\cap C_c^\infty(T^*M;[0,1])$ with $\tilde{\chi}\equiv 1$ on $\supp \chi$ and $\supp \tilde{\chi}\subset \Lambda_{\rho_\gamma}^{\tau+\e_0}(R(h)).$ Then note that by \eqref{E:R} {and \eqref{e:Q_i}} {there exists $C_{_{\!N}}>0$ such that}
\begin{equation}\label{E:D1}
\begin{aligned}\|b\,Q_iOp_h(\chi)u\|_{\LM}&\leq \|b\,Q_iOp_h(\tilde{\chi})Op_h(\chi)u\|_{\LM}+{C_{_{\!N}}h^N}\|u\|_{\LM}\\
&\leq {C_k}R(h)^k\|Op_h(\chi)u\|_{\LM}+{C_{_{\!N}}h^N}\|u\|_{\LM},
\end{aligned}
\end{equation}
{for all $0<h<h_0$ for $h_0$ small enough.}

Second, using that  
\begin{equation*}
b\,P{Q_i}Op_h(\chi )=b\,Q_i Op_h(\chi)P+b\,[P,Q_i]Op_h(\chi)+b\,Q_i[P,Op_h(\chi)],
\end{equation*}
we claim that {there exists $C_{_{\!N}}>0$ such that}
\begin{equation}
\begin{aligned}\label{E:D2}
\|b\,P{Q_i}Op_h(\chi )u\|_{\LM} &\leq {C_k}R(h)^k\|Op_h(\chi)Pu\|_{\LM}+C_0h{R(h)^{k}}\|Op_h(\chi)u\|_{\LM}\\
&\qquad+\|{b\,Q_i[P,Op_h(\chi)]}u\|_{\LM} +{C_{_{\!N}}h^N}\|u\|_{\LM}.
\end{aligned}
\end{equation}
Indeed,  the estimate on $b\,[P,Q_i]Op_h(\chi)$ was obtained as follows. We observe that 
$$
H_pq_i={k} (\xi_i-a_i)^{{k-1}}H_p(\xi_i-a_i).
$$
and since $H_p(\xi_i-a_i)$ vanishes on $\gamma$, $H_pq_i$ vanishes to order {$k$} on $\gamma$. Therefore, { using $\tilde \chi$ as in \eqref{E:D}, on $\supp \tilde{\chi}$} we have $|H_pq_i|\leq C_0R(h)^k$ and {there exists $C_{_{\!N}}>0$ such that}
\begin{align*}
&\|{b\,}[P,Q_i]Op_h(\chi)u\|_{\LM}\leq C_0hR(h)^k\|Op_h(\chi)u\|_{\LM}\\
&\qquad\qquad\qquad\qquad\;\;+\|([P,Q_i]-\tfrac{h}{i}Op_h(H_pq_i)){Op_h(\tilde \chi)Op_h(\chi)}u\|_{\LM} +{C_{_{\!N}}h^N}\|u\|_{\LM}.
\end{align*}
Finally, observe that $([P,Q_i]-\tfrac{h}{i}Op_h(H_pq_i)){Op_h(\tilde{\chi})}\in h^2R(h)^{k-2}S_{\class}$ and hence the bound follows since $R(h)\geq {2 h^\class}$ and $\class <\frac{1}{2}$.

%Next, we bound the third term  
%in \eqref{E:D} by 
%\begin{equation}\label{E:D3}
%\|{Q_i}Op_h(\tilde{\chi}_{{1}})Pu\|_{\LM}\leq { C_k}R(h)^k\|Op_h(\tilde{\chi}_{{1}})Pu\|_{\LM}+{C_{_{\!N}}h^N}\|u\|_{\LM}.
%\end{equation}

Finally, to bound the fourth term in~\eqref{E:D} note that {by \cite[Lemma 6.1]{Gdefect}}
$$
\|[Op_h(\kappa),Q_i]Op_h(\chi)u(x_1,\cdot)\|_{_{{L^2_{\bar{x},x'}}}} \leq  {C_{M,p,R_0}}h^{-\frac{1}{2}}\|[Op_h(\kappa),Q_i]Op_h(\chi)u\|_{\LM}.
$$
Then, observe that $[Op_h(\kappa),Q_i]Op_h(\tilde \chi) \in hR(h)^{k-1}S_{\delta} $ since for $i=2, \dots, k$  we have $\partial_{x_j}q_i=0$ for $j\neq 1$, $\partial_{\xi_1}\kappa=0$, $\partial_{\xi_j}q_i=0$ for all $j \neq i$, and $\partial_{x_i}\kappa\in S_\delta$ {because $\beta_\delta$ is a tangential symbol}. We then obtain {that there exists $C_{_{\!N}}>0$ such that}
\begin{equation}\label{E:D4}
\|[Op_h(\kappa),Q_i]Op_h(\chi)u(x_1,\cdot)\|_{_{{L^2_{\bar{x},x'}}}}\leq Ch^{\frac{1}{2}}R(h)^{k-1}\|Op_h(\chi)u\|_{\LM}+{C_{_{\!N}}h^N}\|u\|_{\LM}.
\end{equation}
Combining \eqref{E:D1}, \eqref{E:D2}, and \eqref{E:D4} into \eqref{E:D} it follows that 
\begin{equation}
\begin{aligned}\label{E:D5} 
R(h)^{-k}\|(hD_{x_i})^\ell v_h(x_1,\cdot)\|_{L_{\bar{x},x'}^2}
&\leq  \left({{C_k}\tau_{\rho_0}^{-\frac{1}{2}}} + C_0{\tau^{\frac{1}{2}}_{\rho_0}}+Ch^{\frac{1}{2}}R(h)^{-1}\right)  \|Op_h(\chi)u\|_{\LM}\\
&+{C_k C_0}\tau^{\frac{1}{2}}_{\rho_0}h^{-1}\|Op_h(\chi)Pu\|_{\LM}\\
& +{C_0\tau^{\frac{1}{2}}_{\rho_0}}h^{-1}\|{b\,Q_i[P,Op_h(\chi)]}u\|_{\LM}  +{C_{_{\!N}}h^N} \|u\|_{\LM},
\end{aligned}
\end{equation}
{for some $C>0$, $C_{_{\!N}}>0$, and for all $0<h<h_0$ with $h_0$ small enough.}

By~\eqref{e:blackberry} {we also know that there exists $C_{_{\!N}}>0$ and $h_0>0$ so that  for all $0<h<h_0$}
\begin{equation}
\label{E:D7}
\begin{aligned}
\|b\,[P,Op_h(\chi)]u\|_{\LM}+\|b\,Q_i[P,Op_h(\chi)]u\|_{\LM}&{\leq} {C_{_{\!N}}h^N}\|u\|_{\LM}.
\end{aligned}
\end{equation}

Feeding~\eqref{E:D7} into \eqref{E:v2} and \eqref{E:D5}, {and combining them in to \eqref{e:centipede}}, we have
\begin{align*}
R(h)^{1-k}h^{k-1}\|v_h(x_1,\bar{x},\cdot)\|^2_{L^2_{x'}}
&\leq C_k  \left( \|v_h(x_1,\cdot)\|^2_{L_{\bar{x},x'}^2}+ R(h)^{-2k} {\sum_{i=2}^{k}}\|(hD_{x_i})^k v_h(x_1,\cdot)\|^2_{L_{\bar{x},x'}^2}\right).\\
&\leq  C_k\left(\tau_{\rho_0}^{-1} + C_0\tau_{\rho_0}+ ChR(h)^{-2}\right)  \|Op_h(\chi)u\|^2_{\LM}\\
&+ Ch^{-2}\|Op_h(\chi)Pu\|_{\LM}^2+{C_{_{\!N}}h^N} \|u\|_{\LM}.  
\end{align*}
Taking $\tau_0\leq C_0^{-1}(\sup_{\SigH}|H_pr_H|)^{-1}$ and $h_0$ small enough so that $ChR(h)^{-2}\leq \tau_{\rho_0}^{-1}$ proves the desired result because of \eqref{E:trick}.
Also, note that,  since $\rho_\gamma \in {V_{\rho_0}}$, in view of \eqref{E:control}, we have 
\[\frac{1}{2}|\partial_{\xi_1}p(\rho_0)| \leq |\partial_{\xi_1}p(\rho_\gamma)| \leq 2|\partial_{\xi_1}p(\rho_0)|.\]
We may therefore rewrite the bound for $\|v_h\|^2_{L^2(H)}$ in terms of ${|H_pr_H(\rho_\gamma)|}$
%
%Finally, observe that since $\supp \tilde{\chi}_1\subset \{\tilde{\chi}\equiv 1\}$, {there exist $h_0>0$ and $C_{_{\!N}}>0$ so that for all $0<h<h_0$}
%\begin{gather*}
%%\|Op_h(\tilde{\chi}_1)u\|_{\LM}\leq {2}\|Op_h(\tilde{\chi})u\|_{\LM}+ {C_{_{\!N}}h^N}\|u\|_{\LM},\\
%\|Op_h(\tilde{\chi}_1)Pu\|_{\LM}\leq {2}\|Op_h(\tilde{\chi})Pu\|_{\LM}+ {C_{_{\!N}}h^N}\|u\|_{\LM},
%\end{gather*}
which completes the proof.

\end{proof}
%%%%%%%%%%%%%%%%%%%%%%%%%%%%%%%%%%%%%%%%%%%%%%%%%%%%%%%%%%%%%%%%%%%%%%%
In what follows we work with points $x\in \re^n$ and $(x,\xi)\in T^*\re^n$. We will isolate one position coordinate $x_1$ and write $(x,\xi)=(x_1,\tilde x, \xi_1, \tilde \xi)$. {This lemma is based on~\cite[Lemma 4.3]{Gdefect} which in turn draws on the factorization ideas from~\cite{KTZ}.}
%%%%%%%%%%%%%%%%%%%%%%%%%%%%%%%%%%%%%%%%%%%%%%%%%%%%%%%%%%%%%%%%%%%%%%%%%%%%%%%%
\begin{lemma}
\label{l:squirrel}
Let {$\Theta:W\subset\re^n\to M$ be coordinates on $M$,} $\rho_0\in T^*\R^n$ {and $\mathfrak{I}>0$} be so that 
 \[
|\partial_{\xi_1}p(\rho_0)|\geq {\mathfrak{I} }>0.
 \] 
 Then, there exist $\tau_0>0$, $C_0>0$, {$r_0>0$} depending only on $(M,p,\mathfrak{I}, {\Theta} )$  and $V_0 \subset {T^*\R^n}$ neighborhood of $\rho_0$, so that {$B(\rho_0,r_0)\subset V_0$}, 
 \begin{equation}
\label{e:nontrivialFlow}
{\frac{3}{4}}|\partial_{\xi_1}p(\rho_0)|\leq \inf_{V_0}|\partial_{\xi_1}p|\leq \sup_{V_0}|\partial_{\xi_1}p|\leq {\frac{4}{3}}|\partial_{\xi_1}p(\rho_0)|,
\end{equation}
and the following holds.\medskip
 
\noindent Let $0\leq \class<\frac{1}{2}$ and {$0<\tau<\tau_0$}. 
{Let $I_\tau=\{x_1: -\tfrac{\tau_{\rho_0}}{3}\leq x_1 \leq \tfrac{\tau_{\rho_0}}{3}\}$ with $\tau_{\rho_0}:=\tau |\partial_{\xi_1}p(\rho_0)|$,}
and  
$$
\kappa=\kappa(x_1, \tilde{x},\tilde{\xi})\in S_\class\cap C_c^\infty\Big({I_\tau \, \times\, }T^*\R^{n-1}\Big).
$$
 Let $\chi\in S_\delta\cap C_c^\infty(V_0;{[-2,2]})$ and   $q=q(x_1)\in C^\infty(\R; S^\infty(T^*\R^{n-1}))$.
Then, there is $C>0$ such that for all {$N>0$}, there is ${C_{_{\!N}}}>0$ and $h_0>0$ so that for all $0<h\leq h_0$,  {and all $x_1$,}
\begin{align*} 
\label{UPSHOT 1}
\|{Op_h(q)}Op_h(\kappa )Op_h(\chi){u({x_1},\cdot)}\|_{{L^2_{ \tilde{x}}}} 
&\leq { 4\tau_{\rho_0}^{-\frac{1}{2}}}\|Op_h(\kappa)\|\|Op_h( q) Op_h( \chi)u\|_{L^2_{x}(|x_1|<\tau_{\rho_0}/2)}\\
&+C_0\tau_{\rho_0}^{\frac{1}{2}}h^{-1} {\|Op_h(\kappa) \|}\|  POp_h(q)Op_h(\chi)u\|_{L^2_{x}(|x_1|<\tau_{\rho_0}/2)}\\
&+{\|[Op_h(\kappa),Op_h(q)]Op_h(\chi){u}({x_1},\cdot)\|}_{L^2_{{\tilde{x}}}}\\
&+{C_{_{\!N}}h^N}\|{u}\|_{L^2_x}.
\end{align*}
 Also, all constants are uniform when $\chi,\kappa,{q}$ are taken in bounded subsets of $S_\class$, {$\Theta$ is taken in bounded subset of $C^k$}, and when {$\mathfrak{I},\,\tau$ are taken uniformly bounded away from 0}. 
\end{lemma}

%%%%%%%%%%%%%%%%%%%%%%%%%%%%%%%%%%%%%%%%%%%%%%%%%%%%%%%%%%%%%%%%%%%%%%%%%%%%%%%%

\begin{proof}
%If $|p(\rho_0)|>{\frac{\mathfrak{I} }{2}}$, we may choose $V_0\subset \TM $  open neighborhood of $\rho_0$ so that $|p(q)|>{\frac{\mathfrak{I} }{4}}$ for $q\in V_0$. {Moreover, there is $r_0=r_0(\mathfrak{I} ,M,p)$ so that $B(\rho_0,r_0)\subset V_0$.}  If $p(\rho_0)<{\frac{\mathfrak{I} }{2}}$,  then by assumption $|\partial_{\xi_1}p(\rho_0)|>{\frac{\mathfrak{I} }{2}}$. Therefore,   
{T}here exists an open neighborhood $V_0$ of $\rho_0$ {so that $|\partial_{\xi_1}p|>\frac{\mathfrak{I}}{2}$ on $V_0$.} {Therefore, we may assume that there is} $e\in C^\infty(T^*\R^n)$ elliptic on $V_0$, and $a=a(x_1,\tilde{x},\tilde{\xi})\in C^\infty(\R\times S^0(T^*\R^{n-1}))$ so that for all $\psi \in C_c^\infty(V_0)$
$$
p(x,\xi)\psi(x,\xi)=e(x,\xi)(\xi_1-a(x_1,\tilde{x},\tilde{\xi}))\psi(x,\xi),
$$
with {$e$ satisfying that for every $\alpha, \beta$,}
\begin{equation}
\label{e:bd}
\begin{gathered}
{\|e^{-1}\|_\infty\leq C_{{1}}=C_{{1}}(M,P,\mathfrak{I}),}\\
{\|\partial_{x}^\alpha\partial_\xi^\beta e(x,\xi)\|_{\infty }\leq C=C(M,P,\mathfrak{I},\alpha,\beta,\Theta)}
\end{gathered}
\end{equation}
where $C(M,P,\mathfrak{I},\alpha,\beta,\Theta)$ depends on $\Theta$ through finitely many $C^k$ norms.
{Moreover, there exists $r_0=r_0(M,p,{\mathfrak{I}})$ so that $B(\rho_0,r_0)\subset V_0$.}

Using this factorization, we see that there exists $R\in S^0(T^*\R^n)$ so that for all $\psi\in S_\delta(V_0)$,
$$
POp_h(\psi)=Op_h(e){(hD_{x_1}-Op_h(a))}Op_h(\psi)+hOp_h(R)Op_h(\psi)+{R_\infty}.
$$
{where we write $R_\infty$ for an $O(h^\infty)_{\Psi^{-\infty}}$ operator that may change from line to line but whose seminorms are bounded by those of $P,\psi,e,e^{-1}$.}
Moreover, there exists an element  $a_1\in hC^\infty(\R\times S^0(T^*\R^{n-1}))$ so that for each fixed $x_1$ the operator $Op_h(a(x_1)+a_1(x_1)):L^2_{\tilde{x}}\to L^2_{\tilde{x}}$ is self-adjoint. Abusing notation slightly, we relabel $a+a_1$ as $a$ and {$Op_h(R)-Op_h(e)Op_h(a_1)$} as $Op_h(R)$. Then, for all $\psi\in S_\delta(V_0)$
$$
POp_h(\psi)=Op_h(e)(hD_{x_1}-Op_h(a))Op_h(\psi)+hOp_h(R)Op_h(\psi)+{R_\infty}.
$$
Therefore, letting $Op_h(e)^{-1}$ denote a microlocal parametrix for {$Op_h(e)$} on $V_0$, we have for all $\psi\in S_\delta(V_0)$,
\begin{equation}
\label{e:factor}
(hD_{x_1}-Op_h(a))Op_h(\psi)=Op_h(e)^{-1}POp_h(\psi)+hOp_h(R_0)Op_h(\psi)+{R_\infty}
\end{equation}
where  {$R_0$ is such that} {$Op_h(R_0)=-Op_h(e)^{-1}Op_h(R).$} {From the symbolic calculus together with~\eqref{e:bd} we see that {for every $\alpha, \beta$}
\begin{equation}
\label{e:Rbd}
{\|\partial_{x}^\alpha\partial_\xi^\beta R_0(x,\xi)\|_{\infty }\leq C=C(M,P,\mathfrak{I} ,\alpha,\beta,{\Theta}) },
\end{equation}}
{where $C$ depends on $\Theta$ through finitely many $C^k$ norms.}
Shrinking $V_0$ {(in a way depending only on $(M,p,\mathfrak{I})$ {and the $C^2$ norm of $\Theta$})}, if necessary, we may also assume that 
\begin{equation*}
{\frac{3}{4}}|\partial_{\xi_1}p(\rho_0)|\leq \inf_{V_0}|\partial_{\xi_1}p|\leq \sup_{V_0}|\partial_{\xi_1}p|\leq {\frac{4}{3}}|\partial_{\xi_1}p(\rho_0)|.
\end{equation*}

Define 
\begin{equation}\label{e:w}
w:=Op_h(q)Op_h(\chi )u,
\end{equation}
{with $Op_h(\psi)=Op_h(q)Op_h(\chi)$} we have by~\eqref{e:factor} that
$$
(hD_{x_1}-Op_h(a)) w=f,
$$
for
\begin{equation}\label{e:R_0}
f:=[Op_h(e)^{-1}POp_h(q)Op_h(\chi )+h Op_h(R_0)Op_h(q)Op_h(\chi )]u+{R_\infty} u.
\end{equation}

Defining {the operator} {$U(x_1,t)$ by
\[
(hD_{x_1}-{Op_h(a)})U(x_1,t)=0,\qquad\qquad
U(t,t)=Id,\]
%\[
%A(t,{x_1},\tilde{x},hD_{\tilde{x}}):=-\int_{{x_1}}^{t}a({s},\tilde{x},hD_{\tilde{x}})d{s},
%\] 
 we obtain that for all ${x_1},t \in \R$
$$
w({x_1},\tilde{x})=\U{x_1}{t}w(t,\tilde{x})-\frac{i}{h}\int_{{x_1}}^{t}\U{x_1}{s}f({s},\tilde{x})d{s}.
$$}
Let $\e=\e(\tau)$ be defined as
  \begin{equation}\label{e:ep}
  \e:=\frac{\tau_{\rho_0}}{3}=\frac{\tau |\partial_{\xi_1}p(\rho_0)|}{3} ,
  \end{equation}
and let $\Phi\in C_c^\infty(\re;[0,{3\e^{-1}}])$ with $\supp \Phi \subset [0,\e]$ and $\int_\R \Phi=1$. Then, integrating in $t$,
\begin{equation}
\label{e:est0}
w({x_1},\tilde{x})=\int_{\R}\!\Phi (t)\U{x_1}{t}w(t,\tilde{x})dt-\frac{i}{ h}\int_\R\! \Phi(t) \int_{{x_1}}^{t}\U{x_1}{s}f({s},\tilde{x})d{s}dt.
\end{equation}

Let $\tau_0$ satisfy
\begin{equation}\label{e:condition on tau_0}
\tau_0< {\sqrt{\frac{3}{2}}}{|\partial_{\xi_1}p(\rho_0)|^{-1}\|Op_h(R_0)\|^{-1}},
 \end{equation}
 where $Op_h(R_0)$ is as in \eqref{e:factor}. Note that by{~\eqref{e:Rbd}} $\tau_0$ only depends on {$(M,P,\mathfrak{I},{\Theta})$}.

%Given $0<\tau<\tau_0$, 
% $b\in  S_\class\cap C_c^\infty(T^*\R^n;[0,1])$ {as in the statement of the lemma}, consider $b_0\in  S_\class\cap C_c^\infty(T^*\R^n;[0,1])$ with $\supp b_0 \subset \{b \equiv1\}$ and
%\[\left\{ \rho\in T^*\R^n: \; d\left(\rho\,,\,\Lambda_{\supp \kappa}^\tau \cap \{p=0\} \cap V_0\right) < {\e_0} \right\} \subset \{b_0 \equiv 1\}.\] 
%Next, applying propagation of singularities (see below for the proof), we claim that {there exists $h_0>0$ so that for all $0<h<h_0$ and} {for all ${x_1} \in \R$}
%\begin{equation}
%\label{e:est0}
%\begin{aligned} 
%Op_h(\kappa) w({x_1},\tilde{x})&=\int_\R\!\Phi(t)Op_h(\kappa)  {\U{x_1}{t}} Op_h(b_0)w(t,\tilde{x})dt\\
%& -\frac{i}{h}  \int_\R \! \Phi(t)\int_{{x_1}}^{t}Op_h(\kappa){\U{x_1}{s}}Op_h(b_0)f({s},\tilde{x})d{s}dt\\
%&+{R_h({x_1}, \tilde x)},
%\end{aligned}
%\end{equation}
%with \[\|R_h({x_1}, \tilde x)\|_{L^\infty_{{x_1}}L^2_{\tilde{x}}}\leq Ch^{-1}\|Op_h(q)Op_h(\tilde{\chi})Pu\|_{L^2_x}+C_{_{\!N}}h^N\|u\|_{L^2}.\]
{From now on, we write 
$$C=C(M,P,\mathfrak{I} ,\e_0,\tau, \chi,q,\kappa,{\Theta}),\quad\text{and}\quad C_{_{\!N}}=C_{_{\!N}}(M,P,N,\tau,\mathfrak{I} ,\e_0,\chi,q,\kappa,{\Theta})$$
for constants depending on finitely many seminorms of the given parameters.}
To bound the first term in \eqref{e:est0} we apply Cauchy-Schwarz and use that ${\U{x_1}{t}}$ is a {unitary operator acting on $L^2_{\tilde x}$} to get
\[
\left \|\int_\R\!\Phi(t)Op_h(\kappa){\U{x_1}{t}}w(t,\tilde{x})dt \right\|_{{L^\infty_{x_1}}L^2_{\tilde x}}
 \!\!\leq \|\Phi\|_2\, \|Op_h(\kappa) \|\|w\|_{L^2_{t,\tilde x}(|t|\leq \e)}.
 \]
To bound the second term {in \eqref{e:est0}} we apply Minkowski's integral inequality, use that the support of $\Phi$ is contained in $[0,\e]$, {and that $\supp \kappa \subset \{|x_1|<\e\}$ to get}
\begin{align*}
&\left \|  \int_\R \! \Phi(t)\int_{{x_1}}^tOp_h(\kappa){\U{x_1}{s}}f({s},\tilde{x})d{s}dt\right\|_{{L^\infty_{x_1}}L^2_{\tilde x}}\\
&\leq \left\|\int_\R  \Phi(t) \left(\int_{\R^{n-1}}\left(\int_\R \mathbf 1_{[{-\e},\e]}({s})Op_h(\kappa){\U{x_1}{s}}f({s},\tilde{x})d{s}\right)^2 d\tilde x\right)^\frac{1}{2}dt\right\|_{{L^\infty_{x_1}}}\\ \ \medskip
%&\leq \int_\R  \Phi(t) \left(\int_{\R^{n-1}}\| \mathbf 1_{[0,\e]}\|_{L^2_{x_1}}^2 \left\|Op_h(\kappa)e^{-\frac{i}{h}A(x_1,t,\tilde{x},hD_{\tilde{x}})}Op_h(b)f(x_1,\tilde{x})\right\|^2_{L^2_{x_1}} d\tilde x\right)^\frac{1}{2}dt\\
&\leq \| \mathbf 1_{[{-\e},\e]}({s})\|_{L^2_{{s}}}  \left\|Op_h(\kappa)\|\|f\right\|_{L^2_{{s}, \tilde x}(|s|\leq \e)}. 
\end{align*}
Feeding  these two bounds into \eqref{e:est0}, and using that {$\Phi(t) \leq 3\e^{-1}$} and $\int_\R \Phi(t)dt =1$ give {$\|\Phi\|_{L^2(\R)}\leq  \sqrt{3}\e^{-\frac{1}{2}}$},  we obtain
\begin{align}\label{e:fly}
\|Op_h(\kappa) w(x_1,\cdot)\|_{L^2_{\tilde{x}}}
\!\leq {\sqrt{3}}\e^{-\frac{1}{2}}  \|Op_h(\kappa) \|\|w\|_{L^2_{x}(|x_1|\leq \e)}
+{\sqrt{2}}\e^{\frac{1}{2}}h^{-1} \! \left\|Op_h(\kappa)\|\|f\right\|_{L^2_{x}(|x_1|\leq \e)} .
\end{align}
Finally, note that according to \eqref{e:R_0}
\begin{align*}
\|f\|_{L^2_x(|x_1|\leq \e)}
&\leq \|Op_h(e)^{-1}POp_h(q)Op_h(\chi )u\|_{L^2_x(|x_1|\leq \e)}\\
&\quad + h \|Op_h(R_0)Op_h(q)Op_h(\chi )u\|_{L^2_x(|x_1|\leq \e)}+{C_{_{\!N}}h^N}\|u\|_{L^2_x}\\
&\leq C_0\|POp_h(q)Op_h(\chi )u\|_{L^2_x(|x_1|\leq 3\e/2)}\\
&\quad + h \|Op_h(R_0)\| \| {Op_h({b})Op_h(q)}Op_h(\chi )u\|_{L^2_x(|x_1|\leq 3\e/2)}+{C_{_{\!N}}h^N}\|u\|_{L^2_x}.
\end{align*}
{Using~\eqref{e:bd}, we see that $C_0>0$ depends only $(M,P,\mathfrak{I} )$.}
Therefore, since 
$$Op_h(q)Op_h(\kappa)Op_h(\chi)=Op_h(\kappa){Op_h(q)}Op_h(\chi)+[Op_h(q),Op_h(\kappa)]{Op_h( \chi)},$$ 
we may combine  definition \eqref{e:w} of $w$ with \eqref{e:fly} to obtain
\begin{align*}
\|Op_h(q)Op_h(\kappa)Op_h(\chi)u({x_1},\cdot)\|_{L^2_{\tilde{x}}}
&\leq  {\sqrt{3}}\e^{-\frac{1}{2}}{ \|Op_h(\kappa) \|}\|Op_h(q)Op_h(\chi)u\|_{L^2_{x}(|x_1|\leq \e)}\\
&+C_0h^{-1}\e^{\frac{1}{2}} { \|Op_h(\kappa) \|}\|POp_h(q)Op_h(\chi )u\|_{L^2_{x}(|x_1|\leq 3\e/2)}\\
&+ {\sqrt{2}}\e^{\frac{1}{2}} {\|Op_h(R_0)\|  \|Op_h(\kappa) \|}   \|{Op_h(q)}Op_h(\chi )u\|_{L^2_{x}(|x_1|\leq 3\e/2)}\\
&+C_{_{\!N}}h^N\|u\|_{L^2_x}+\|[Op_h(q),Op_h(\kappa)]{Op_h( \chi)}u(x_1,\cdot)\|_{L^2_{\tilde x}}.
\end{align*}
To finish the proof   we combine the first and third terms in the bound above using that {$\sqrt{3}\ep^{-\frac{1}{2}}=3\tau_{\rho_0}^{-\frac{1}{2}}$  and that  \eqref{e:condition on tau_0} gives  ${\sqrt{2}}\e^{\frac{1}{2}} \|Op_h(R_0)\| \leq  \tau_{\rho_0}^{-\frac{1}{2}}$}.\\

\end{proof}

%%%%%%%%%%%%%%%%%%%%%%%%%%%%%%%%%%%%%%%%%%%%%%%%%%%%%%%%%%%%%%%%%%%%%%%
%%%%%%%%%%%%%%%%%%%%%%%%%%%%%%%%%%%%%%%%%%%%%%%%%%%%%%%%%%%%%%%%%%%%%%%
%%%%%%%%%%%%%%%%%%%%%%%%%%%%%%%%%%%%%%%%%%%%%%%%%%%%%%%%%%%%%%%%%%%%%%%
%%%%%%%%%%%%%%%%%%%%%%%%%%%%%%%%%%%%%%%%%%%%%%%%%%%%%%%%%%%%%%%%%%%%%%%

\section{Non-looping Propagation Estimates: Proof of Theorem~\ref{t:coverToEstimate}}
\label{s:Egorov}

%%%%%%%%%%%%%%%%%%%%%%%%%%%%%%%%%%%%%%%%%%%%%%%%%%%%%%%%%%%%%%%%%%%%%%%
%%%%%%%%%%%%%%%%%%%%%%%%%%%%%%%%%%%%%%%%%%%%%%%%%%%%%%%%%%%%%%%%%%%%%%%
%%%%%%%%%%%%%%%%%%%%%%%%%%%%%%%%%%%%%%%%%%%%%%%%%%%%%%%%%%%%%%%%%%%%%%%
%%%%%%%%%%%%%%%%%%%%%%%%%%%%%%%%%%%%%%%%%%%%%%%%%%%%%%%%%%%%%%%%%%%%%%%
%%%%%%%%%%%%%%%%%%%%%%%%%%%%%%%%%%%%%%%%%%%%%%%%%%%%%%%%%%%%%%%%%%%%%%%%%%%%%%%%

The main result in this section is the proof of Theorem~\ref{t:coverToEstimate} which we present in what follows. {The proof is based on an application of Egorov's theorem (see Lemma~\ref{l:recur-a}) which in turn uses that cutoffs with disjoint support act almost orthogonally.}\\

\noindent{\bf Proof of Theorem~\ref{t:coverToEstimate}.}
By Theorem~\ref{t:porcupine} there exist $\tau_0$,  $R_0$, and $C_{n,k}>0$  so that if  $0<\tau\leq \tau_0$,  $0\leq \class<\frac{1}{2}$, $N>0$,  and ${8}h^\delta\leq R(h)<R_0$,
then for {$\{\Lambda^\tau_{_{\rho_j}}(R(h))\}_j$} a $({\mathfrak D},\tau,R(h))$-good cover of $\SigH$, and {$\{\chi_j\}_j$} a $\delta$-partition associated to the cover,  there exist $C>0$, {$h_0>0$}, 
so that for all $w=w(x';h)\in S_\class \cap C_c^\infty({\tilde{H}})$ there {is} $C_{_{\!N}}>0$  with the property that for any $0<h<h_0$ and $u\in \mc{D}'(M)$,
\begin{align}\label{e:mosquito}
h^{\frac{k-1}{2}}\Big|\int_{\tilde{H}}wud\sigma_{{\tilde{H}}}\Big|
&\leq \frac{C_{n,k} {\|w\|_{_{\!\infty}}}}{\tau^{\frac{1}{2}}\FR^{\frac{1}{2}}}R(h)^{\frac{n-1}{2}}\sum_{j\in \mathcal J_h(w)}{\|Op_h(\chi_j)u\|_{\LM}}\\
&+Ch^{-1}{\|w\|_\infty}\|Pu\|_{\Hm} +C_{_{\!N}}h^N\big(\|u\|_{\LM}+{\|Pu\|_{\Hm}}\big).\notag
\end{align}
Next, suppose there exist $\mc{B}\subset \{1,\dots, N_h\}$ and a finite collection  $\{\mc{G}_\ell\}_{\ell \in \mathcal L} \subset \{1,\dots, N_h\}$ {satisfying}  $\mathcal J_h(w)\;\subset\;  \mc{B} \cup \bigcup_{\ell \in \mathcal L}\mc{G}_\ell$, and {with $\{\mc{G}_\ell\}_{\ell \in \mathcal L}$ having }the non self looping properties described in the statement of the theorem. {Furthermore, since we are working with a $({\mathfrak D},\tau,R(h))$-good cover, we split each $\mc{G}_\ell$ into $\mathfrak D$ families $\{\mc{G}_{\ell, i}\}_{i=1}^{\mathfrak{D}}$ of disjoint tubes.}

Note that  
\begin{align*}
\sum_{j\in \mathcal J_h(w)}\|Op_h(\chi_j)u\|_{\LM}&\leq \sum_{\ell\in \mathcal L}{\sum_{i=1}^{\mathfrak{D}} \sum_{j\in \mc{G}_{\ell,i}}}\|Op_h(\chi_j)u\|_{\LM}+\sum_{j\in \mc{B}}\|Op_h(\chi_j)u\|_{\LM}.
\end{align*}
Since 
$$
{\bigcup_{j\in\mc{G}_{\ell}}\Lambda_{\rho_j}^\tau(R(h))\qquad\text{ is}\;\;[t_\ell(h),T_\ell(h)]\text{ non-self looping}},
$$
{and the tubes in $\mc{G}_{\ell,i}$ are disjoint,} we may apply  Lemma~\ref{l:recur-a} below {to $\mc{G}=\mc{G}_{\ell,i}$} and  $(t_j,T_j)=(t_\ell, T_\ell)$ for all {$j \in \mc{G}_{\ell,i}$} {together with Cauchy-Schwarz to get} 

\begin{align*}
{\sum_{j \in \mc{G}_{\ell,i}}}\|Op_h(\chi_j)u\|_{\LM}&\leq \Big(\frac{t_\ell |\mc{G}_\ell|}{T_\ell}\Big)^{\frac{1}{2}}\Big(\sum_{j\in { \mc{G}_{\ell,i}}}\frac{\|Op_h(\chi_j)u\|^2_{\LM}T_\ell}{ t_\ell}\Big)^{\frac{1}{2}}\\
&\leq 2\Big(\frac{t_\ell |\mc{G}_\ell|}{T_\ell}\Big)^{\frac{1}{2}}\Big(\|u\|_\LM^2+ \frac{T_{\ell}^2}{h^2}\, \|Pu\|^2_\LM\Big)^{\frac{1}{2}}.
\end{align*}

On the other hand, {using Cauchy-Schwarz and the fact that there are $\mathfrak{D}$ families of disjoint tubes},
$$
\sum_{j\in \mc{B}}\|Op_h(\chi_j)u\|_{\LM}\leq {2 \mathfrak{D}}|\mc{B}|^{\frac{1}{2}}\|u\|_{\LM}.
$$
Therefore, after adjusting $C_{n,k}$ {in \eqref{e:mosquito}},
\begin{align*}
&h^{\frac{k-1}{2}}\Big|\int_Hwu\, d\sigma_H\Big|\\
&\leq \frac{C_{n,k}{\mathfrak{D}}\|w\|_{_{\!\infty}}{R(h)^{\frac{n-1}{2}}}}{\tau^{\frac{1}{2}}\FR^{\frac{1}{2}}}\Big[\sum_{\ell \in \mathcal L} \Big(\frac{t_\ell |\mc{G}_\ell|}{T_\ell}\Big)^{\frac{1}{2}}\Big(\|u\|_\LM^2+ \frac{T_{\ell}^2}{h^2}\, \|Pu\|^2_\LM\Big)^{\frac{1}{2}}+|\mc{B}|^{\frac{1}{2}}\|u\|_{\LM}\Big]\\
&\qquad+Ch^{-1}{\|w\|_\infty}\|Pu\|_{\Hm} +C_{_{\!N}}\big(\|u\|_{\LM}+{\|Pu\|_{\Hm}}\big)\\
&\leq  \frac{C_{n,k}{\mathfrak{D}}\|w\|_{_{\!\infty}}{R(h)^{\frac{n-1}{2}}}}{\tau^{\frac{1}{2}}\FR^{\frac{1}{2}}}\Big[\sum_{\ell\in \mathcal L} \Big(\frac{t_\ell |\mc{G}_\ell |}{T_\ell}\Big)^{\frac{1}{2}}\|u\|_\LM+ \sum_{\ell\in L} \Big(\frac{|\mc{G}_\ell |t_\ell T_{\ell}}{h^2}\Big)^{\frac{1}{2}}\, \|Pu\|_\LM+|\mc{B}|^{\frac{1}{2}}\|u\|_{\LM}\Big]\\
&\qquad+Ch^{-1}{\|w\|_\infty}\|Pu\|_{\Hm} +C_{_{\!N}}\big(\|u\|_{\LM}+{\|Pu\|_{\Hm}}\big).
\end{align*}
\qed

%%%%%%%%%%%%%%%%%%%%%%%%%%%%%%%%%%%%%%%%%%%%%%%%%%%%%%%%%%%%%%%%%%%%%%%%%%%%%%%%

The next lemma relies on Egorov's theorem to the Ehrenfest time (see for example~\cite[Proposition 3.8]{DyGu14}, \cite{EZB}).

%%%%%%%%%%%%%%%%%%%%%%%%%%%%%%%%%%%%%%%%%%%%%%%%%%%%%%%%%%%%%%%%%%%%%%%%%%%%%%%%
\begin{lemma}
\label{l:recur-a}
Assume that $P$ is self adjoint. Let  $0\leq \class_0<\frac{1}{2}$, $0<2\e_0<1-2\class_0$, {and let $\mc{G}$ be a set of indices with  $|\mc{G}|\leq h^{-N}$ for some $N>0$}. For each $\ell \in {\mc{G}}$ let  $0\leq\class_\ell\leq \class_0$,   $0<\alpha_\ell<{1}-2\class_\ell-2\e_0$, and $\chi_\ell \in S_{\class_\ell}(\TM ) \cap C^\infty_c(\TM ;[-C_1h^{1-2\delta_0},1+C_1h^{1-2\delta_0}])$. 
 In addition,  for each $\ell \in {\mc{G}}$ let  {$t_\ell(h)>0$ and  $0<T_{\ell}(h)\leq 2\alpha_\ell\, T_e(h)$} be so that
\begin{equation}
\label{e:supp2}
\bigcup_{k\in {\mc{G}}}\supp \chi_k \cap \varphi_{-t}(\supp \chi_\ell)=\emptyset
\end{equation}
for all $t \in   [{t_\ell(h)}, T_{\ell}(h)]$ or $t\in {[-T_\ell(h),-t_\ell(h)]}$,
and suppose that  
\begin{equation}
\label{e:supp1}
\bigcup_{k\neq \ell}\supp \chi_k \cap \supp \chi_\ell=\emptyset.
\end{equation} 
Then, %there exists a universal constant $C_0>0$ (independent of $T_\pm(h), {t_0},h, \chi,P,\rho,\e$) 
there exists a  constant  $h_0>0$  so that for $0<h<h_0$ 
$$
\sum_{\ell\in {\mc{G}}}\frac{\|Op_h(\chi_\ell)u\|_\LM^2T_{\ell}(h)}{t_\ell(h)}\leq 4\|u\|_\LM^2+ 4\max_{\ell \in {\mc{G}}}\frac{T_{\ell}(h)^2}{h^2}\, \|Pu\|^2_\LM .
$$
Moreover, the constant $h_0$ can be chosen to be uniform for $\chi_\ell$ in bounded subsets of $S_\class(\TM )$ {and $N<N_0$}.
\end{lemma}
%%%%%%%%%%%%%%%%%%%%%%%%%%%%%%%%%%%%%%%%%%%%%%%%%%%%%%%%%%%%%%%%%%%%%%%%%%%%%%%%

\begin{proof}
Throughout this proof it will be convenient to write $\|\cdot\|$ for $\|\cdot\|_\LM$. Define $\tilde \chi$ by 
$$
Op_h(\tilde{\chi})=\sum_{\ell\in {\mc{G}}} \sum_{k=\frac{-T_{\ell}}{2t_\ell}}^{\frac{T_{\ell}}{2t_\ell}} e^{\frac{ik t_\ell P}{h}}Op_h(\chi_\ell )e^{-\frac{ik t_\ell P}{h}}.
$$
First, we claim that there exists $h_0>0$ so that for all $0<h<h_0$
\begin{equation}\label{e:three halfs}
  \|Op_h(\tilde{\chi})u\|^2\leq \frac{3}{2}{\|u\|^2}.
  \end{equation}
Indeed, Egorov's Theorem \cite[Proposition 3.9]{DyGu14} gives that there exists $C_\chi>0$ and $h_0>0$  
so that for every $k$ 
\begin{equation}\label{E:eg-a}
\begin{gathered} {e^{\frac{ik t_\ell P}{h}}Op_h(\chi_\ell )e^{-\frac{ik t_\ell P}{h}}=Op_h(\chi_{k,\ell})+O(h^\infty)_{\Psi^{-\infty}},\qquad \chi_{k,\ell}= \chi_\ell \circ \varphi_{_{\!k t_\ell}}  + r_{k,\ell}(h),}
 \end{gathered}
\end{equation}
where $r_{k,\ell}\in h^{1-d_{k,\ell}(h)-2\class_\ell}S_{d_{k,\ell}(h)/2+\class_\ell}$, \; $\supp r_{k,\ell}\subset \supp \chi_\ell\circ{\varphi_{kt_\ell}},$
\[
|r_{k,\ell}(h)| \leq C_\chi h^{1-d_{k,\ell}(h)-2\class_\ell} \qquad \text{ and } \qquad  {d_{k,\ell}(h)} \leq |k| \frac{t_\ell}{ T_e(h)},
\]
 for all $0<h<h_0$. 
  Note that since $\{\chi_\ell\}_{\ell\in {\mc{G}}} \mapsto \tilde \chi$ is a continuous map from 
  $${\prod_{\ell\in {\mc{G}}} S_{\delta_\ell}(\TM)\to S_{\frac{1}{2}-\e_0}(\TM),}$$ 
  the constant $C_\chi$ can be chosen to be uniform for $\{\chi_\ell\}_{\ell\in {\mc{G}}}$ in bounded subsets of {$ \Pi_\ell S_{\class_\ell}(\TM )$}, and that then the same is true for $h_0$. 
 
Now, let {$\ell,m \in \mc{G}$ with $\ell\neq m$} and assume without loss that $T_{\ell}\leq T_{m}$. Then, {using~\eqref{e:supp2} and~\eqref{e:supp1}}, we have for $\frac{-T_{\ell}(h)}{2t_\ell}\leq {k}\leq \frac{T_{\ell}(h)}{2t_\ell}$, $\frac{-T_{m}(h)}{2t_m}\leq {j}\leq \frac{T_{m}(h)}{2t_m}$, 
$$
\varphi_{_{-kt_\ell}}(\supp \chi_\ell)\cap \varphi_{_{-j t_m}}(\supp \chi_m)=\supp \chi_\ell\cap \varphi_{_{kt_\ell-j t_m}}(\supp \chi_m)=\emptyset.
$$ 
In addition, {using~\eqref{e:supp2}}, we have if $\ell=m$, then for $\frac{-T_{\ell}(h)}{2t_\ell}\leq {k<j}\leq \frac{T_{\ell}(h)}{2t_\ell}$,
$$
\varphi_{_{-kt_\ell}}(\supp \chi_\ell)\cap \varphi_{_{-j t_\ell}}(\supp \chi_m)=\supp \chi_\ell\cap \varphi_{_{(k-j)t_\ell}}(\supp \chi_m)=\emptyset.
$$
Thus, it follows from \eqref{E:eg-a} that 
 \[
 \tilde \chi = \sum_{\ell\in {\mc{G}}} \sum_{k=-\frac{T_\ell}{2t_\ell}}^{\frac{T_\ell}{2t_\ell}}  \chi_\ell \circ \varphi_{\!_{k t_\ell }} + r(h). 
 \]
with $|r(h)|\leq C_\chi h^{2\e_0}$  for all $0<h<h_0$,  and $C_{\chi}, h_0$ can be chosen uniform for $\{\chi_\ell\}_{\ell=1}^J$ in bounded subsets of $S_{\class_0}$.  {We have used  that the support of the $r_{k,\ell}$'s are disjoint, {together with the fact that $2\e_0<1-\alpha_\ell-2\delta_\ell$ implies $2\e_0<1-d_{k,\ell}(h)-2\delta_\ell$}, to get the bound on $r(h)$.}
This implies that
\begin{equation}\label{E:tilde chi-a}
\tilde \chi \in S_{\frac{1}{2}-\e_0}
\qquad \text{and }\qquad 
-C_\chi h^{2\e_0}\leq \tilde{\chi}\leq 1+C_\chi h^{2\e_0},
\end{equation}  
for all $0<h<h_0$.% To conclude this we have used that {$d_{k,\ell}<\alpha_\ell$ and $\alpha_\ell+2\delta_\ell<1-2\e_0.$}

Note that by the sharp G\aa rding inequality \eqref{E:tilde chi-a} yields
\[
\big\langle \left( 1+C_\chi  h^{2\e_0}-Op_h(\tilde{\chi})^*Op_h(\tilde{\chi})\right)u,u\big\rangle \geq -C_\chi h^{2\e_0}\|u\|_{L^2}^2,
\]
 which in turn gives
 \begin{equation}\label{E:u bound-a}
  \|Op_h(\tilde{\chi})u\|^2\leq (1+2C_\chi h^{2\e_0}){\|u\|^2}
  \end{equation}
  for all $0<h<h_0$.
  %\marginpar{{[Y]Note that this breaks in strong localization. We can't have both $\frac{\alpha}{2}<1-\r$ and $\frac{\alpha}{2}<\frac{1}{2}$}}
  Also, note that since $\e_0>0$, we may shrink $h_0$  so that \eqref{E:u bound-a} gives
 \begin{equation}\label{E:u bound2-a}
  \|Op_h(\tilde{\chi})u\|^2\leq {\frac{{3}}{2}}{\|u\|^2},
  \end{equation}
for $0<h<h_0$ as claimed in \eqref{e:three halfs}.
 
Next, note that  since the supports of the $\chi_m \circ \varphi_{_{\!j t_m}}$ and $\chi_\ell\circ\varphi_{_{\!kt_\ell}}$ are disjoint for $(j,m)\neq (k,\ell)$, Egorov's Theorem also gives
\begin{equation}\label{E:egorov-a}
 \left \langle e^{\frac{ij t_mP}{h}}Op_h(\chi_m )e^{-\frac{ij t_mP}{h}}u \; ,\;  e^{\frac{ikt_\ell P}{h}}Op_h(\chi_\ell)e^{-\frac{ikt_\ell P}{h}}u \right\rangle =O_\chi(h^\infty)\|u\|^2,
 \end{equation}
where  the constant in $O_\chi(h^N)$ depends only on the $|\alpha|\leq C_{\!N}\,n$ seminorms of $\chi$,  where $ C_{\!N}$ is a universal constant.
It then follows from \eqref{E:u bound2-a} and \eqref{E:egorov-a} that
\begin{equation}
\label{E:decomposition-a}
\begin{aligned}
{\frac{3}{2}\|u\|^2}&\geq \sum_{\ell\in {\mc{G}}}\sum_{k=-\frac{T_\ell}{2t_\ell}}^{\frac{T_\ell}{2t_\ell}}\left \|e^{\frac{ikt_\ell P}{h}}Op_h(\chi_\ell)e^{-\frac{ikt_\ell P}{h}}u\right\|^2
+O_\chi(h^\infty \max_\ell|T_\ell|))\|u\|^2,
\end{aligned}
\end{equation}
as long as we work with $0\leq h\leq h_0$ and $h_0$ small enough so that   $r(h)$ can be absorbed by ${\frac{3}{2}\|u\|^2}$.

On the other hand, since the propagators $e^{\frac{ikt_\ell P}{h}}$ are unitary operators, 
\begin{equation}
\label{E:bound-a}
\begin{aligned}
\left \|e^{\frac{ikt_\ell P}{h}}Op_h(\chi_\ell )e^{-\frac{ikt_\ell P}{h}}u\right\|^2&=  \left \|Op_h(\chi_\ell )e^{-\frac{ikt_\ell P}{h}}u\right\|^2\\\\
&= \|Op_h(\chi_\ell )u\|^2-I_{k,\ell } - I\!\!I_{k,\ell }
\end{aligned}
\end{equation}
where
\begin{equation*}
\begin{aligned}
I_{k,\ell }&=
\left\langle Op_h(\chi_\ell )[u-e^{-\frac{ikt_\ell P}{h}}u],Op_h(\chi_\ell )u\right\rangle,\\
I\!\!I_{k,\ell }&=\left\langle Op_h(\chi_\ell )e^{-\frac{ikt_\ell P}{h}}u,Op_h(\chi_\ell )[u-e^{-\frac{ikt_\ell P}{h}}u]\right\rangle.
\end{aligned}
\end{equation*}
 
It follows from \eqref{E:bound-a} that 
\begin{equation}
\label{E:decomposition2-a}
\sum_\ell \!\!\!\sum_{k=-\frac{T_\ell }{2t_\ell }}^{\frac{T_\ell }{2t_\ell }}\!\! \left\|e^{ikt_\ell P/h}Op_h(\chi_\ell )e^{-ikt_\ell P/h}u \right\|^2
\!\!\!\!\!=\!\sum_\ell \frac{T_\ell }{t_\ell }\|Op_h(\chi_\ell )u\|^2-\sum_\ell \!\!\sum_{k=-\frac{T_\ell }{2t_\ell }}^{\frac{T_\ell }{2t_\ell }}\!\! I_{k,\ell } +I\!\!I_{k,\ell }.
\end{equation}

%Define $A_{k,\ell }:= \frac{i}{h}\int_0^{kt_\ell }\left\langle e^{\frac{isP}{h}}Op_h(\chi_\ell )e^{-\frac{isP}{h}}Pu,e^{\frac{isP}{h}}Op_h(\chi_\ell )e^{-\frac{isP}{h}}u\right\rangle ds$  and  
Observe that
\begin{equation*}
I_{k,\ell }
= \frac{i}{h}\int_0^{kt_\ell }\left\langle Op_h(\chi_\ell )e^{-\frac{isP}{h}}Pu,Op_h(\chi_\ell )u\right\rangle ds
=A_{k,\ell } +B_{k,\ell }, 
\end{equation*}
where
\begin{align*}
&A_{k,\ell }:= \frac{i}{h}\int_0^{kt_\ell }\left\langle e^{\frac{isP}{h}}Op_h(\chi_\ell )e^{-\frac{isP}{h}}Pu,e^{\frac{isP}{h}}Op_h(\chi_\ell )e^{-\frac{isP}{h}}u\right\rangle ds\\
&B_{k,\ell }:=\frac{i}{h}\int_0^{kt_\ell }\left\langle e^{\frac{isP}{h}}Op_h(\chi_\ell )e^{-\frac{isP}{h}}Pu,e^{\frac{isP}{h}}Op_h(\chi_\ell )(u-e^{-\frac{isP}{h}}u)\right\rangle ds
\end{align*}

To deal with the $A_{k,\ell }$ terms note that 
\begin{align*}
&\sum_{k,\ell }A_{k,\ell }\leq \frac{1}{h}\sum_{k,\ell }\int_0^{kt_\ell }\|e^{\frac{isP}{h}}Op_h(\chi_\ell )e^{-\frac{isP}{h}}Pu\| \|e^{\frac{isP}{h}}Op_h(\chi_\ell )e^{-\frac{isP}{h}}u\| ds\\
&\leq \frac{1}{h}\left(\sum_{\ell ,k}\int_0^{kt_\ell }\|e^{\frac{isP}{h}}Op_h(\chi_\ell )e^{-\frac{isP}{h}}Pu\|^2ds\right)^{\frac{1}{2}} \left(\sum_{\ell ,k}\int_0^{kt_\ell } \|e^{\frac{isP}{h}}Op_h(\chi_\ell )e^{-\frac{isP}{h}}u\|^2 ds\right)^{\frac{1}{2}}.
\end{align*}
In addition, observe that for $v\in L^2$,
\begin{equation}
\label{e:orthog}
\begin{aligned}
\sum_{\ell ,k}\int_0^{kt_\ell }\|e^{\frac{isP}{h}}Op_h(\chi_\ell )e^{-\frac{isP}{h}}v\|^2ds&\leq \left\langle Lv, v\right\rangle,
\end{aligned}
\end{equation}
with $L:=\sum_{\ell ,k}\int_0^{kt_\ell } e^{\frac{isP}{h}}Op_h(\chi_\ell )^*Op_h(\chi_\ell )e^{-\frac{isP}{h}}ds$.
Also, another application of Egorov's theorem gives
\begin{equation*}
L=Op_h\left(\sum_{\ell ,k}\int_0^{kt_\ell } |\chi_\ell |^2\circ\varphi_s +\tilde r_{k,\ell }(s,h)ds\right)+O(h^\infty)_{\Psi^{-\infty}}
 \end{equation*}
where $\tilde r_{k,\ell }(s,h)\in h^{1-d_{k,\ell}(h)-2\delta_\ell }S_{d_{k,\ell}/2+\delta_\ell }$ with $ \supp \tilde r_{k,\ell }(s,h)\subset \supp \chi_\ell \circ\varphi_{s}$ and 
$$
|\tilde r_{k,\ell }(s,h)|\leq C_\chi h^{1-d_{k,\ell}(h)-2\delta_\ell }.
$$
Next, we claim that \eqref{e:supp2} gives
\begin{equation}
\label{e:bud}
\Big|\int_0^{kt_\ell } |\chi_\ell |^2\circ\varphi_s +\tilde r_{k,\ell }(s,h)ds\Big|\leq t_\ell (1+ C_\chi h^{1-d_{k,\ell}(h)-2\delta_\ell }).
\end{equation}
{To see this, let $\rho \in T^*\!M$, $s,t\in [-\frac{T_\ell }{2},\frac{T_\ell }{2}]$, be so that $\varphi_s(\rho)\in \supp \chi_\ell $ and $\varphi_t(\rho)\in \supp \chi_\ell $.
Suppose $s \geq t$ and note that
$$\varphi_s(\rho) \in \varphi_{s-t}(\supp \chi_\ell )\cap \supp\chi_\ell .$$
Therefore,  since $0\leq s-t\leq T_\ell ,$ we obtain $0\leq s-t\leq t_\ell $ from~\eqref{e:supp2}. This proves the claim.}

In addition, we claim that combining \eqref{e:bud} with \eqref{e:supp1} gives
\begin{equation}
\label{e:Bud2}
\Big|\sum_{\ell ,k}\int_0^{kt_\ell } |\chi_\ell |^2\circ\varphi_s +\tilde r_{k,\ell }(s,h)ds\Big|\leq \max_\ell  T_\ell (h)(1+C_\chi h^{1-\e_0}).
\end{equation}
To see this, first observe that $\#\big\{ k \in [-\tfrac{T_\ell }{2t_\ell }, \tfrac{T_\ell }{2t_\ell }]\big\} \leq T_\ell /t_\ell $. Together with~\eqref{e:bud} {this} implies
\begin{equation}\label{e:first}
\Big| \sum_k \int_0^{kt_\ell }|\chi_\ell |^2\circ\varphi_s +\tilde r_{k,\ell }(s,h)ds\Big|\leq T_\ell (1+C_\chi h^{1-\e_0}).
\end{equation}
Second, note that
$$
\supp\Big(\sum_k \int_0^{kt_\ell }|\chi_\ell |^2\circ\varphi_s +\tilde r_{k,\ell }(s,h)ds\Big)\subset \bigcup_{s=-T_\ell /2}^{T_\ell /2}\varphi_{-s}(\supp\chi_\ell ).$$
Therefore, by~\eqref{e:supp1} for $\ell \neq j$
\begin{equation}\label{e:second}
\supp\Big(\sum_k \int_0^{kt_\ell }|\chi_\ell |^2\circ\varphi_s +\tilde r_{k,\ell }(s,h)ds\Big)\cap \supp\Big(\sum_k \int_0^{kt_j}|\chi_j|^2\circ\varphi_s +\tilde r_{k,\ell }(s,h)ds\Big)=\emptyset.
\end{equation}
Combining \eqref{e:first} with \eqref{e:second} we {obtain} \eqref{e:Bud2} as claimed.

Using \eqref{e:orthog} and \eqref{e:Bud2} together with the same argument we used for $\tilde{\chi}$,  for $h_0$ small enough (uniform for $\chi_\ell $ in bounded subsets of $S_{\delta_\ell }$)
$$
\sum_{\ell ,k}\int_0^{kt_\ell }\|e^{\frac{isP}{h}}Op_h(\chi_\ell )e^{-\frac{isP}{h}}v\|^2ds\leq 2{\max_\ell  T_\ell (h)}\|v\|^2.
$$
In particular,
$$\Big|\sum_{\ell ,k} A_{k,\ell }\Big|\leq 2\frac{\max_\ell  T_\ell (h)}{h}\|Pu\|\|u\|.$$

We next turn to dealing with $B_{k,\ell }$. Note that
\begin{equation*}
B_{k,\ell }={\frac{1}{h^2}}\int_0^{kt_\ell }\int_0^s\left\langle  e^{\frac{i(t-s)P}{h}}e^{\frac{isP}{h}}Op_h(\chi_\ell )e^{-\frac{isP}{h}}Pu,e^{\frac{itP}{h}}Op_h(\chi_\ell )e^{-\frac{itP}{h}}Pu\right\rangle \;dt ds.
\end{equation*}
Therefore, by a similar argument this time using
$$
{\Big|\int_0^{kt_\ell }\int_0^{kt_\ell } |\chi_\ell |^2\circ\varphi_s +\tilde r_{k,\ell }(s,h)dtds\Big|}\leq kt_\ell ^2(1+ C_\chi h^{1-d_{k,\ell}(h)-2\delta_\ell }),
$$
we obtain
\begin{equation}
\begin{aligned}
\Big|\sum_{\ell ,k} B_{k,\ell }\Big|&\leq \frac{1}{h^2}\sum_{\ell ,k}\int_0^{kt_\ell }\int_0^s\| e^{\frac{isP}{h}}Op_h(\chi_\ell )e^{-\frac{isP}{h}}Pu\|\|e^{\frac{itP}{h}}Op_h(\chi_\ell )e^{-\frac{itP}{h}}Pu\|dt ds\\
&\leq \frac{1}{h^2}\sum_{\ell ,k}\int_0^{kt_\ell }\int_0^{kt_\ell }\| e^{\frac{isP}{h}}Op_h(\chi_\ell )e^{-\frac{isP}{h}}Pu\|^2ds dt\\
&\leq  2\frac{\max_\ell  T_\ell ^2(h)}{h^2}\|Pu\|^2.
\end{aligned}
\end{equation}

We have therefore shown that 
\begin{equation}\label{e:I}
\Big|\sum_{\ell ,k} I_{k,\ell }\Big|\leq 2\frac{\max_\ell  T_\ell (h)}{h}\|Pu\|\|u\|+ 2\frac{\max_\ell  T_\ell ^2(h)}{h^2}\|Pu\|^2.
\end{equation}

Next, note that
\begin{equation*}
\begin{aligned}
I\!\!I_{k,\ell }&=\left\langle Op_h(\chi_\ell )e^{\frac{-ikt_\ell P}{h}}u,Op_h(\chi_\ell )[u-e^{-\frac{ikt_\ell P}{h}}u]\right\rangle\\
%&= \left\langle e^{\frac{ikt_\ell P}{h}}Op_h(\chi_\ell )e^{\frac{-ikt_\ell P}{h}}u,e^{\frac{ikt_\ell P}{h}}Op_h(\chi_\ell )[u-e^{-\frac{ikt_\ell P}{h}}u]\right\rangle\\
&= \frac{i}{h}\int_0^{kt_\ell }\left\langle e^{\frac{ikt_\ell P}{h}}Op_h(\chi_\ell )e^{\frac{-ikt_\ell P}{h}}u,e^{\frac{ikt_\ell P}{h}}Op_h(\chi_\ell )e^{-\frac{isP}{h}}Pu\right\rangle ds\\
%&= \frac{i}{h}\int_0^{kt_\ell }\left\langle e^{\frac{ikt_\ell P}{h}}Op_h(\chi_\ell )e^{\frac{-ikt_\ell P}{h}}u,e^{\frac{i(k-s)P}{h}}e^{\frac{ist_\ell P}{h}}Op_h(\chi_\ell )e^{-\frac{isP}{h}}Pu\right\rangle ds\\
&\leq \frac{1}{h}\int_0^{kt_\ell }\|e^{\frac{ikt_\ell P}{h}}Op_h(\chi_\ell )e^{\frac{-ikt_\ell P}{h}}u\|\|{e^{\frac{ik{(t_\ell-s)} P}{h}}}{e^{\frac{iksP}{h}}}Op_h(\chi_\ell )e^{-\frac{isP}{h}}Pu\|ds.
\end{aligned}
\end{equation*}
Then, by {unitarity of $e^{-\frac{i{t_\ell-s}P}{h}}$ and}~\eqref{e:orthog},
\begin{equation}\label{e:II}
\Big|\sum_{\ell ,k} I\!\!I_{k,\ell }\Big|\leq 2\frac{\max_\ell T_\ell }{h}\|Pu\|\| u\|.
\end{equation}
In particular, from \eqref{e:I} and \eqref{e:II} we have
\begin{equation}
\label{e:R_k-a}\Big|\sum_{\ell ,k}I_{k,\ell } +I\!\!I_{k,\ell } \Big|\leq  4\frac{\max_\ell T_\ell }{h}\|Pu\|\| u\| + 2\frac{\max_\ell T_\ell ^2}{h^2}\|Pu\|^2 \leq  2\|u\|^2+4\frac{\max_\ell T_\ell ^2}{h^2}\|Pu\|^2.
\end{equation}
By possibly shrinking $h_0$ we may assume  that the error term in $\eqref{E:decomposition-a}$ is smaller than $\frac{1}{2}\|u\|^2$ for $0<h<h_0$. We conclude from \eqref{E:decomposition-a} together with \eqref{E:bound-a},{~\eqref{E:decomposition2-a}} and \eqref{e:R_k-a} that 
\begin{align}
{2}\|u\|^2
&\geq \sum_\ell \frac{T_\ell (h)}{t_\ell }\|Op_h(\chi_\ell )u\|^2 -2\|u\|^2-4\frac{\max_\ell T_\ell ^2}{h^2}\|Pu\|^2 \label{E:ant-a}.
\end{align}
Therefore, \eqref{E:ant-a} gives
$$
\sum_{\ell\in {\mc{G}}}  \frac{\|Op_h(\chi_\ell )u\|^2T_\ell (h)}{t_\ell }\leq \Big(4\|u\|^2+ 4\frac{\max_\ell T_\ell ^2}{h^2}\|Pu\|^2 \Big)$$
for $0<h<h_0$.
As noted right after \eqref{E:eg-a} the constant $h_0$ can be chosen to be uniform for $\chi_\ell $ in compact subsets of $ S_{\class_0}(\TM )$.
\end{proof}
%%%%%%%%%%%%%%%%%%%%%%%%%%%%%%%%%%%%%%%%%%%%%%%%%%%%%%%%%%%%%%%%%%%%%

\addcontentsline{toc}{section}{\quad\;\,\bf{Dynamical analysis}}

%%%%%%%%%%%%%%%%%%%%%%%%%%%%%%%%%%%%%%%%%%%%%%%%%%%%%%%%%%%%%%
%%%%%%%%%%%%%%%%%%%%%%%%%%%%%%%%%%%%%%%%%%%%%%%%%%%%%%%%%%%%%%
%%%%%%%%%%%%%%%%%%%%%%%%%%%%%%%%%%%%%%%%%%%%%%%%%%%%%%%%%%%%%%
%%%%%%%%%%%%%%%%%%%%%%%%%%%%%%%%%%%%%%%%%%%%%%%%%%%%%%%%%%%%%%

\section{Quantitative improvements in integrable geometries}
\label{s:spheres}
%%%%%%%%%%%%%%%%%%%%%%%%%%%%%%%%%%%%%%%%%%%%%%%%%%%%%%%%%%%%%
%%%%%%%%%%%%%%%%%%%%%%%%%%%%%%%%%%%%%%%%%%%%%%%%%%%%%%%%%%%%%
%%%%%%%%%%%%%%%%%%%%%%%%%%%%%%%%%%%%%%%%%%%%%%%%%%%%%%%%%%%%%
%%%%%%%%%%%%%%%%%%%%%%%%%%%%%%%%%%%%%%%%%%%%%%%%%%%%%%%%%%%%%

{In this section, we focus on the special case of spheres of revolution $M=[0,2\pi]_\theta\times [0,\pi]_r$ with Hamiltonian
$$
p(\theta,r,\xi_{\theta},\xi_r)=\xi_r^2+\tfrac{1}{\alpha(r)^2}\, {\xi_\theta^2}+V(r),
$$ 
and operate under the assumptions of Theorem~\ref{t:sphere}.}

{In this setting,} one can explicitly describe the Liouville tori intersected with $\{p=0\}$ as
$$
\mathbb{T}_{\xi_\theta}=\Big\{(\theta,r,\xi_r):\; \xi_r^2=V(r)-\tfrac{1}{\alpha(r)^2}{\xi_\theta^2}\Big\}.
$$
In particular, 
$$
\mathbb{T}_{\xi_\theta}\cap S^*_{(\theta_0,r_0)}M=\Big\{ \xi_r=\pm\sqrt{V(r_0)-\tfrac{1}{\alpha(r_0)^2}{\xi_\theta^2}}\Big\},
$$
and for any $\delta>0$ there is $c>0$ so that if $r_0\in [\delta,2\pi-\delta]$ the two intersections are separated by at least 
\begin{equation}
    \label{e:canyon}
    c\sqrt{\alpha(r_0)\sqrt{{V}(r_0)}-\xi_\theta}.
\end{equation} 
 Let $R_1>0$ and define \[A_{\pm,R_1}:=\{(\theta, r,\xi_\theta, \xi_r) \in T^*M: \; \pm \xi_r\geq R_1\}.\] 

{Theorem~\ref{t:sphere} is a consequence of the following Lemma which constructs non-looping covers together with Theorem~\ref{t:coverToEstimate}.}
\begin{lemma}
\label{l:sphere}
Let the above assumptions hold. Fix $\delta>0$ and let $\big\{\Lambda_{_{\rho_j}}^\tau(R)\big\}_{j=1}^{N_R}$ be as in Proposition~\ref{l:cover}. Then there exists $\beta>0$ so that if $r_0\in [\delta,2\pi-\delta]$, $H=\{x\}=\{(r_0,\theta_0)\}$ the following holds. For all $0<\tau<\tau_0$, $ \alpha_1>0$, $0<R\ll 1$, and $0<T<cR^{\alpha_1-1}$, there exists ${\mc{B}}\subset \{1,\dots, N_R\}$ so that  for $R_1=R^{\alpha_1}$
$$
|{\mc{B}}| \leq \beta T^3 R^{1-\alpha_1}+R^{-\alpha_1}
$$
and for $j\notin {\mc{B}}$ with $\Lambda_{_{\rho_j}}^\tau(R)\cap \Lambda_{_{A_{\pm,R_1}\cap \SigH}}^\tau(R)\neq \emptyset,$ 
$$
d\Big (\Lambda_{_{A_{\pm,R_1}\cap \SigH}}^\tau(R),\bigcup_{t\in [1,T]}\varphi_t(\Lambda_{_{\rho_j}}^\tau(R))\Big)\geq 2R
$$
In particular, 
$$
\bigcup_{j\notin{\mc{B}}}\Lambda_{_{\rho_j}}^\tau(R)\text{ is }[1,T]\text{ non-self looping}.
$$
\end{lemma}

\begin{proof}
We start by removing tubes covering the intersection of an $R^{1-{\alpha_1}}$ neighborhood of $\xi_\theta=\sqrt{V(r_0)}\alpha(r_0)$ with $\SigH$. This requires $R^{-{\alpha_1}}$ tubes of radius $R$. In particular, this covers an $R^{1-{\alpha_1}}$ neighborhood of the singular torus and we may restrict our attention to $A_{\pm,R_1}$.

We claim that  there is $C>0$ so that if $\rho_1,\rho_2$ are at least $\alpha$ away from the singular torus, then
\begin{equation}\label{e:distances}
|\Theta(\rho_1)-\Theta(\rho_2)|+|I(\rho_1)-I(\rho_2)|\leq C\alpha^{-1}d(\rho_1,\rho_2). 
\end{equation}
Indeed, by (e.g. \cite[eqn.~(3.37)]{To09},~\cite[Theorem 3.12]{Vu06},~\cite[Theorem. Page 9]{El90}) there are Birkhoff normal form symplectic coordinates in a neighborhood of the stable bicharacteristic $\gamma_s$ so that $\rho=(t,x,\tau,\xi)\in S^1\times \re\times \re^2$ with  $\gamma_s$ given by ${ \{(t,0,0,0):\, t\in S^1\}}$  so that 
$$p(t,x,\tau,\xi)=\tau+f(x^2+\xi^2,\tau),$$
$f \in C^\infty((-\delta, \delta)^2 ; \re)$ for some $\delta>0$ and  $f(u,v)=\alpha(v)u+O(v^2)+O_v(u^2)$ for some $\alpha \in C^\infty((-\delta, \delta) ; \re)$.

In particular, we may work with action-angle coordinates $(\Theta, I)$  given by 
$$
 I_1=\tau,\qquad I_2=\frac{1}{2}(x^2+\xi^2)\qquad
x=\sqrt{2I_2}\cos(\Theta_2),\qquad \xi=\sqrt{2I_2}\sin(\Theta_2).
$$
In these coordinates $p(\Theta, I)=I_1 + f(2 I_2, I_1)$, the action coordinate function $I_2(x,\xi)$ measures the squared distance from $(x,\xi)$ to the singular torus, 
and we have
$$|\partial_{_{I,\Theta}} \rho|\leq C/\sqrt{2I_2}=C\alpha^{-1}.$$
This yields \eqref{e:distances} as claimed.

Next, suppose 
$$
d(\rho,\SigH\cap A_{\pm, R_1})<2R,\qquad d(\varphi_t(\rho),\SigH\cap A_{\pm,R_1})<2R.
$$
There exists $\tilde{\rho}\in \SigH\cap A_{\pm,R_1}$ with $d(\rho,\tilde{\rho})<2R$. Therefore,  for some $C>0$,
$$
d(\varphi_t(\tilde{\rho}),\varphi_t(\rho))<CRt
$$
and hence, for  $t\leq T$, 
$$d(\varphi_t(\tilde{\rho}),\SigH\cap A_{\pm,R_1})< (CT+1)R.$$
Now, for $RT\ll R^{{\alpha_1}} $, by~\eqref{e:canyon} since $\rho$ is at least $R^{1-\alpha_1}$ away from the singular torus, the only intersection of $\mathbb{T}_{I_0(\tilde{\rho})}$ with 
$$
\{q:\; d(q,\SigH\cap A_{\pm,R_1})<(CT+1)R\}
$$
happens at $q$ with $d(q,\tilde{\rho})<(CT+1)R$. In particular, 
$$
d(\varphi_t(\tilde{\rho}),\tilde{\rho})<(CT+1)R,
$$
and hence by~\eqref{e:distances}
$$
d(t\partial_I p(I_0),2\pi \mathbb{Z}^2)<CTRR^{-1+{\alpha_1}}.
$$
That is, $\tilde{\rho}$ is $CTR^{\alpha_1}$ close to a rational torus of period $t$. Thus, the same is true for the original $\rho$ with possibly a different constant.

Now, the points that are $CTR^{\alpha_1}$ close to the intersection of $\SigH\cap A_{\pm}$ with $\mathbb{T}_{I_0}$ can be covered by $CTR^{1-{\alpha_1}}$ tubes. Moreover, since $p$ is isoenergetically non-degenerate, there is $c>0$ so that the rational tori of period $\leq T$, are separated by $cT^{-2}$. Hence, there are at most $CT^2$ such tori and we require $CT^3R^{1-{\alpha_1}}$ tubes.
\end{proof}
%%%%%%%%%%%%%%%%%%%%%%%%%%%%%%%%%%%%%%%%%%%%%%%%%%%%%%%%%%%%%%%%%%%%%%%%%%%%%%%

%%%%%%%%%%%%%%%%%%%%%%%%%%%%%%%%%%%%%%%%%%%%%%%%%%%%%%%%%%%%%%%%%%%%%%%%%%%%%%%
%\begin{theorem}
%\label{t:sphere}
%Let $\alpha$ and $V$ satisfy the assumptions above, %$M=[0,2\pi]_\theta\times[0,\pi]_r$ and 
%$$
%g=dr^2+\alpha(r)^2d\theta^2.
%$$
%Then, for
%\begin{equation}
%    \label{e:sphereHamilton}
%P=-h^2\Delta_g+V(r)+hQ
%\end{equation}
%with $Q\in \Psi^2(M)$ self-adjoint, and $K\subset [0,2\pi]\times (0,\pi)$ compact, there exists $C>0$ with the following properties. For all $L>0$ there exists $h_0>0$ so that for $0<h<h_0$, and $u \in \mathcal D'(M)$
%$$
%\|u\|_{L^\infty(K)}\leq Ch^{-\frac{1}{2}}\Bigg(\frac{\|u\|_{\LM}}{L\sqrt{\log h^{-1}}}+\frac{L\sqrt{\log h^{-1}}\|Pu\|_{\Hs{-\frac{1}{2}}}}{h}\Bigg).
%$$
%In particular, if $\|Pu\|_{\Hs{-\frac{1}{2}}}=o\Big(\frac{h\|u\|_{\LM}}{\log h^{-1}}\Big)$, then
%$$
%\|u\|_{L^\infty(K)}=o\Bigg(\frac{h^{-\frac{1}{2}}}{\sqrt{\log h^{-1}}}\|u\|_{\LM}\Bigg).
%$$
%\end{theorem}
%%%%%%%%%%%%%%%%%%%%%%%%%%%%%%%%%%%%%%%%%%%%%%%%%%%%%%%%%%%%%%%%%%%%%%%%%%%%%%%

\begin{proof}[{Proof of Theorem~\ref{t:sphere}}]
Fix $L>0$, $r_0\in[\delta,2\pi-\delta]$, $\theta_0\in [0,\pi]$ and $\alpha_1=\frac{1}{2}$. Then for $0<R\ll1$ and $0<T<R^{-\frac{1}{2}}$, we may apply Lemma~\ref{l:sphere}. Let $\{\Lambda_{\rho_j}^\tau(R)\}_{j=1}^{N_R}$ be the cover of $\SigH$ given by Proposition~\ref{l:cover}. Then, there are $\mc{G},\mc{B}\subset \{1,\dots, N_R\}$ so that 
\begin{gather*}
|\mc{B}|\leq (\beta T^3+1)R^{-\frac{1}{2}},\qquad \{1,\dots, N_R\}\subset \mc{G}\cup \mc{B}\\
\bigcup_{j\in \mc{G}}\Lambda_{\rho_j}^\tau(R)\text{ is }[1,T]\text{ non-self looping}.
\end{gather*}
Fix $0<\e<\delta<\frac{1}{2}$, let $R=h^\e$ and $T=L^2\log h^{-1}$.  We next apply Theorem~\ref{t:coverToEstimate} {with $P$ as in~\eqref{e:sphereHamilton}}, $\mathcal G_\ell= \mc{G}$,  $T_\ell=T$ and $t_\ell= 1 $  for all $\ell$. Then,  there exist $C>0$ independent of $L$, for any $N>0$, $C_{_{\!N}}>0$, and $h_0>0$, so that for all $0<h<h_0$

\begin{align*}
&h^{\frac{1}{2}}\|u\|_{L^\infty(B((r_0,\theta_0),h^\delta))}\\
&\leq C h^{\frac{\e}{2}}\Bigg(\Big[(\log h^{-1})^{\frac{3}{2}}h^{-\frac{\e}{4}}+ \frac{h^{-\frac{\e}{2}}}{L\sqrt{\log h^{-1}}}\Big]\|u\|_{\LM}
+\frac{h^{-\frac{\e}{2}}L\sqrt{\log h^{-1}}}{h}\|Pu\|_{\LM}\Bigg)\\ &\hspace{5.5cm}+Ch^{-1}\|Pu\|_{\Hs{-\frac{1}{2}}} +C_{_{\!N}}h^N\big(\|u\|_{\LM}+\|Pu\|_{\Hs{-\frac{1}{2}}}\big)\\
&\leq C \Bigg(\beta\Big[(\log h^{-1})^{\frac{3}{2}}{h^{\frac{\e}{4}}}+ \frac{1}{L\sqrt{\log h^{-1}}}\Big]\|u\|_{\LM}+\frac{L\sqrt{\log h^{-1}}}{h}\|Pu\|_{\Hs{-\frac{1}{2}}}\Bigg).
\end{align*}
\end{proof}

\section{Change of the Hamiltonian}\label{S:change of hamiltonian}

When studying quasimodes for the Laplacian, it will be convenient to replace the operator $P_0:=-h^2\Delta_g-1$ by an operator whose dynamics agree with those of $p=|\xi|_g-1$.
\begin{lemma}
\label{l:gecko}
There exists $P\in\Psi^0(M)$ with real, classically eliptic symbol $p$ such that {$\{p=0\}=\SM$,} $p=|\xi|_g-1$ in a neighborhood of $\SM$ and there exist $Q\in \Psi^{-2}(M)$, $E\in h^\infty\Psi^{-\infty}(M)$ satisfying
$$
P=QP_0+E.
$$
In particular, for all $s\in \re$ there exists a constant $C_s>0$ depending only on $s$ so that for all $N>0$, there exist $C_{N,s}=C(N,s, M,g)>0$ and $h_0=h_0(N,s, M,g)>0$ so that for $0<h<h_0$ and $u\in \mc{D}'(M)$,
$$
\|Pu\|_{_{\sob{s}}}\leq C_s\|P_0u\|_{_{\sob{s-2}}}+C_{N,s}h^N\|u\|_{_{\sob{-N}}}.
$$
\end{lemma}

\begin{proof}
Let $\psi_1\in C_c^\infty(\re;[0,1])$ with $\supp \psi_1 \subset (-\frac{1}{2},\frac{1}{2})$ and  $\psi_1\equiv 1$ on $[-\frac{1}{4},\frac{1}{4}]$. Next, let $\psi_2\in C_c^\infty(\re; {[0,1]})$ with $\psi_2\equiv 1$ on $[-4,-\frac{1}{2}]\cup [\frac{1}{2},4]$ so that $\psi:=\psi_1+\psi_2$ has $\psi\equiv 1$ on $[-4,4]$. Define
\[\tilde P=\tilde{P}_1+\tilde{P}_2+\tilde{P}_3\]
with
\begin{equation}
\label{e:Trex}
{\tilde{P}}_1:=\tfrac{1}{2}\psi_1(-h^2\Delta_g),\qquad
{\tilde{P}}_2:=\psi_2(-h^2\Delta_g)\sqrt{-h^2\Delta_g}, \qquad
{\tilde{P}}_3:=2(1-\psi(-h^2\Delta_g)).
\end{equation}

Note that by the functional calculus~\cite[Theorem 14.9]{EZB} $\tilde{P}\in \Psi(M)$ with symbol
$$
\tilde{p}:=\tfrac{1}{2}\psi_1(|\xi|^2_g)+\psi_2(|\xi|_g^2)|\xi|_g+2(1-\psi(|\xi|_g^2))
$$
In particular, $\tilde{p}=|\xi|_g$ in a neighborhood of $\SM$.

Next, observe that 
\begin{align*}
(\tilde{P}+I)(\tilde{P}-I)
&=P_0+h^2\Delta_g+ \tilde{P}^2\\
&=P_0-(I-\psi_2^2(-h^2\Delta_g))(-h^2\Delta_g)+\tilde{P}_1^2+\tilde{P}_3^2+2\tilde{P}_1\tilde{P}_2+2\tilde{P}_2\tilde{P}_3{+2\tilde{P}_1\tilde{P}_3}
\end{align*}
Now, there exists $c>0$ so that 
$$\WF_h(\tilde{P}_1)\cup \WF_h(\tilde{P}_3)\cup \WF_h(I-\psi_2^2(-h^2\Delta_g))\subset \{|\sigma(P_0)|>c\langle \xi \rangle^2\}.$$
In particular, by the elliptic parametrix construction (see e.g.~\cite[Appendix E.2]{ZwScat}) there is $Q_1\in \Psi^{-2}(M)$ so that 
$$
(\tilde{P}+I)(\tilde{P}-I)=Q_1P_0 +O(h^\infty)_{\Psi^{-\infty}}.
$$
Now, $\sigma(\tilde{P}+I)>1$ therefore, $(\tilde{P}+I)^{-1}\in \Psi(M)$ and we have that 
$$
\tilde{P}-I=(\tilde{P}+I)^{-1}Q_1P_0+O(h^\infty)_{\Psi^{-\infty}}
$$
which completes the proof of the lemma after letting $Q=(\tilde{P}+I)^{-1}Q_1$ and $P=\tilde{P}-I$.

\end{proof}
Applying Theorem~\ref{t:coverToEstimate} to $P$ from Lemma~\ref{l:gecko}, where $P_0:=-h^2\Delta_g-1$, and then estimating $Pu$ by Lemma~\ref{l:gecko}, we obtain the following theorem.
\begin{theorem}
{Let $\{H_h\}_h \subset M$ be a {regular} family of submanifolds of codimension $k$ that is uniformly conormally transverse for $p$.}
{Let $\{\tilde H_h\}_h$ be a family of submanifolds of codimension $k$ satisfying \eqref{e:conormalClose}.}
 Let $0<\delta<\frac{1}{2}$, $N>0$ and ${\{w_h\}_h}$ with {$w_h\in S_\class \cap C_c^\infty(\tilde H_h)$}. There exist positive constants $\tau_0=\tau_0(M,g,\Tinj ,{\{H_h\}_h})$, {$R_0=R_0(M,g, {\KR_{_{0}}},{k,\Tinj })$,} $C_{n,k}$ depending only on $n$ and $k$, and $h_0=h_0(M,g,{\delta, {\{H_h\}_h}})$ and for each $0<\tau\leq \tau_0$  there exist, $C=C(M,g,\tau,\delta,, {\{H_h\}_h})>0$ and $C_{_{\!N}}=C_{_{\!N}}(M,g, N,\tau,\delta,{\{w_h\}_h},{\{H_h\}_h})>0$,   so that the following holds.

Let ${8}h^\delta\leq R(h){<R_0}$,{ $0\leq\alpha< 1-2{\limsup_{h\to 0}\frac{\log R(h)}{\log h}}$,} and suppose $\{\Lambda_{_{\rho_j}}^\tau(R(h))\}_{j=1}^{N_h}$  is a {$(\mathfrak{D},\tau, R(h))$ cover of $\SNH$ for some $\mathfrak{D}>0$}.

In addition, suppose there exist $\mc{B}\subset \{1,\dots, N_h\}$ and a finite collection  $\{\mc{G}_\ell\}_{\ell \in \mathcal L} \subset \{1,\dots, N_h\}$ with 
$$
\mathcal J_h(w_h)\;\subset\;  \mc{B} \cup \bigcup_{\ell \in \mathcal L}\mc{G}_\ell,
$$
where $\mathcal J_h(w_h)$ is defined in \eqref{e:i},
 and so that for every $\ell \in \mathcal L$ there exist  $t_\ell=t_\ell(h)>0$ and ${T_\ell=T_\ell(h)}\leq {2} \alpha T_e(h)$  so that
$$
\bigcup_{j\in \mc{G}_\ell}\Lambda_{_{\rho_j}}^\tau(R(h))\;\;\text{ is }\;\;[t_\ell,T_{\ell}]\text{ non-self looping for }\varphi_t:=\exp(tH_{|\xi|_g}).
$$
Then, for $u\in \mc{D}'(M)$ and $0<h<h_0$,
\begin{align*}
h^{\frac{k-1}{2}}\Big|\int_{\tilde H_h} w_h u\, d\sigma_{\tilde H_h}\Big|
&\leq\frac{C_{n,k}{\mathfrak{D}}\|w_h\|_{_{\!\infty}}R(h)^{\frac{n-1}{2}}}{\tau^{\frac{1}{2}}}
\Bigg(|\mc{B}|^{\frac{1}{2}}+\sum_{\ell \in \mathcal L }\frac{(|\mc{G}_\ell|t_\ell)^{\frac{1}{2}}}{T^{\frac{1}{2}}_\ell}\Bigg)\|u\|_{\LM} \\
&+\frac{C_{n,k}{\mathfrak{D}}\|w_h\|_{_{\!\infty}}R(h)^{\frac{n-1}{2}}}{\tau^{\frac{1}{2}}} \sum_{\ell \in \mathcal L}\frac{(|\mc{G}_\ell|t_\ell T_\ell)^{\frac{1}{2}}}{h}\;\|P_0u\|_{\LM}\! \\
&+Ch^{-1}{\|w_h\|_\infty}\|P_0u\|_{\Hl}+C_{_{\!N}}h^N\big(\|u\|_{\LM}+{\|P_0u\|_{\Hl}}\big).
\end{align*}
Here, the constant $C_{_{\!N}}$ depends on $\{w_h\}_h$  only through finitely many $S_\delta$ seminorms of $w_h$. {The constants ${\tau_0},C,C_{_{\!N}},h_0$ depend on $\{H_h\}_h$ only through finitely many of the constants $\KR_{_{\alpha}}$ in \eqref{e:curvature}.}
\end{theorem}

%%%%%%%%%%%%%%%%%%%%%%%%%%%%%%%%%%%%%%%%%%%%%%%%%%%%%%%%%%%%%%
%%%%%%%%%%%%%%%%%%%%%%%%%%%%%%%%%%%%%%%%%%%%%%%%%%%%%%%%%%%%%%
%%%%%%%%%%%%%%%%%%%%%%%%%%%%%%%%%%%%%%%%%%%%%%%%%%%%%%%%%%%%%%
%%%%%%%%%%%%%%%%%%%%%%%%%%%%%%%%%%%%%%%%%%%%%%%%%%%%%%%%%%%%%%
\appendix
\section{}
%%%%%%%%%%%%%%%%%%%%%%%%%%%%%%%%%%%%%%%%%%%%%%%%%%%%%%%%%%%%%%
%%%%%%%%%%%%%%%%%%%%%%%%%%%%%%%%%%%%%%%%%%%%%%%%%%%%%%%%%%%%%%
%%%%%%%%%%%%%%%%%%%%%%%%%%%%%%%%%%%%%%%%%%%%%%%%%%%%%%%%%%%%%%
\label{s:prelim}

%%%%%%%%%%%%%%%%%%%%%%%%%%%%%%%%%%%%%%%%%%%%%%%%%%%%%%%%%%%%%%%%%%%%%%%%%%%%%%%%
{\subsection{Index of Notation}\label{s:index}\ \smallskip

In general we denote points in $\TM$ by $\rho$. When position and momentum need to be distinguished we write $\rho=(x,\xi)$ for $x\in M$ and $\xi \in T_x^*M$.  Sets of indices are denoted in calligraphic font (e.g $\mathcal I$).  Next, we list symbols that are used repeatedly in the text along with the location where they  are first defined.
\pagebreak
\begin{multicols}{3}
\noindent $\mc{C}^{r,t}_x$          \tabto{1.6cm}    \eqref{e:conjugate set}\\ 
$\Sigma_{H,p}$          \tabto{1.6cm}    \eqref{e:SigH}\\
$\varphi_t$            \tabto{1.6cm}    \eqref{e:varphi}\\
$\mc{K}_{_{\!\alpha}}$            \tabto{1.6cm}  \eqref{e:curvature}\\
$r_{H}$               \tabto{1.6cm}    \eqref{e:HprH}\\
$K_p$                \tabto{1.6cm}   \eqref{e:classellipt}\\\
$\FR$               \tabto{1.6cm}
\eqref{e:FR}
\vfill\null
\columnbreak 
\noindent $\mc{H}_{\Sigma}$       \tabto{1.6cm}    \eqref{e:Hsig}\\
$\Tinj$                 \tabto{1.6cm}    \eqref{e:Tinj}\\
$\Lambda_{_{\!A}}^\tau(r)$ \tabto{1.6cm}  \eqref{e:tube}\\
$\Lambda_{\rho}^\tau(r)$ \tabto{1.6cm}  \eqref{e:tube2}\\
$\mc{J}_h(w)$          \tabto{1.6cm}   \eqref{e:i}  \\
$T_e(h)$                    \tabto{1.6cm}    \eqref{e:Lmax}\\
$\Lambda_{\max}$                   \tabto{1.6cm}    \eqref{e:Lmax}

\vfill\null
\columnbreak 
\noindent $\beta_\delta$     \tabto{1.6cm}   \eqref{e:beta}\\
$\mathfrak{D}_n$    \tabto{1.6cm}   Prop. \ref{l:cover}\\
$\Psi_\delta^k$            \tabto{1.6cm}     \eqref{e:Sdelta}\\
$S_{\delta}^{k}$           \tabto{1.6cm}     \eqref{e:Sdelta}\\
$H_{\scl}^k$           \tabto{1.6cm}      \eqref{e:sobolev}\\
$\MSh$              \tabto{1.6cm}      Def. \ref{d:MSh}
\end{multicols}
}
\noindent {For the definition of $[t,T]$\text{ non-self looping}, see~\eqref{e:nonsl}. For that of $(\mathfrak{D},\tau,r)$ good covers, see Definition~\ref{d:good cover}.}

\subsection{Notation from semiclassical analysis}
\label{s:SemiNote}

We refer the reader to~\cite{EZB} or~\cite[Appendix E]{ZwScat} for a complete treatment of semiclassical analysis, but recall some of the relevant notation here. We say $a\in C^\infty(\TM )$ is a symbol of order $m$ and class $0\leq \delta<\frac{1}{2}$, writing $a\in S^m_\delta(\TM )$ if there exists $C_{\alpha\beta}>0$ so that
\begin{equation}\label{e:Sdelta}
|\partial_x^\alpha \partial_\xi^\beta a(x,\xi)|\leq C_{\alpha\beta}h^{-\delta(|\alpha|+|\beta|)}\langle \xi\rangle^{m-|\beta|},\qquad \langle \xi\rangle:=(1+|\xi|_g^2)^{1/2}.
\end{equation}
{Note that we implicitly allow $a$ to also depend on $h$, but omit it from the notation.}
We then define $S^\infty_\delta(\TM ):=\bigcup_m S^m_\delta(\TM )$. We sometimes write $S^m(\TM )$ for $S_0^m(\TM )$. We also sometimes write $S_\delta$ for $S_\delta^m$. Next, we say that $a\in S^{\comp}_\delta(\TM)$ if $a$ is supported in an $h$-independent compact {sub}set {of $\TM$}. 

Next, there is a quantization procedure $Op_h:S^m_\delta \to \mc{L}(C^\infty(M),\mc{D}'(M))$ and we say $A\in \Psi^m_\delta(M)$ if there exists $a\in S^m_\delta(\TM )$ so that $Op_h(a)-A=O(h^\infty)_{\Psi^{-\infty}}$ where we say an operator is $O(h^k)_{\Psi^{-\infty}}$ if for all $N>0$ there exists $C_{_{\!N}}>0$ so that
$$
\|Au\|_{H^N(M)}\leq C_{_{\!N}}h^k\|u\|_{H^{-N}(M)}{,}
$$
and say {an} operator, $A$, is $O(h^\infty)_{\Psi^{-\infty}}$ if for all $N>0$ there exists $C_{_{\!N}}>0$ so that
$$
\|Au\|_{H^N(M)}\leq C_{_{\!N}}h^N\|u\|_{H^{-N}(M)}.
$$

For $a\in S_\delta^{m_1}(\TM )$ and $b\in S_\delta^{m_2}(\TM )$, we have that 
\begin{equation}\label{e:exp}
Op_h(a)Op_h(b)=Op_h(c),\qquad c(x,\xi)\sim \sum_j h^jL_{2j}(a(x,\xi)b(y,\eta))\Big|_{\substack{x=y\\\xi=\eta}} 
\end{equation}
where $L_{2j}$ is a differential operator of order $j$ in $(x,\xi)$ and order $j$ in $(y,\eta)$.

There is a symbol map $\sigma:\Psi^m_\delta(M)\to S^m_\delta(\TM )/h^{1-2\delta}S^{m-1}_\delta(\TM )$ so that 
\begin{gather*}
\sigma(Op_h(a))=a,\qquad \sigma(Op_h(a)^*)=\bar{a},\\
\sigma(Op_h(a)Op_h(b))=ab,\qquad \sigma([Op_h(a),Op_h(b)])=-ih\{a,b\},
\end{gather*}
and 
$$
0\longrightarrow h^{1-2\delta}\Psi_\delta^{m-1}(M)\longrightarrow \Psi_\delta^m(M)\overset{\sigma}{\longrightarrow} S^m_\delta(M)/h^{1-2\delta}S^{m-1}_\delta(M)\longrightarrow 0
$$
is exact.

The main consequence of~\eqref{e:exp} that we will use is that if $p\in S^m(M)$ and $a\in S^k_\delta(\TM )$, then
$$
[Op_h(p),Op_h(a)]=\frac{h}{i}Op_h(H_pa)+h^{2-2\delta}Op_h(r)
$$
with $r\in S^{m+k-2}_\delta(\TM )$. 

We define the semiclassical Sobolev spaces $\sob{s}$ by
\begin{equation}\label{e:sobolev}
\sob{s}:=\{u\in \mc{D}'(M)\mid \|u\|_{_{\sob{s}}}<\infty\},\quad \|u\|_{_{\sob{s}}}:=\|Op_h(\langle \xi\rangle^s)u\|_{L^2(M)}.
\end{equation}

\subsection{Background on Microsupports and Egorov's Theorem}\label{s:micro}

\begin{definition}
\label{d:MSh}
For a pseudodifferential operator $A\in \Psi^{\comp}_{\delta}(M)$, we say that $A$ is microsupported in a family of sets  $\{V(h)\}_h$ and write $ \MSh(A)\subset V(h)$ if 
$$
A=Op_h(a)+O(h^\infty)_{\Psi^{-\infty}}
$$
{and for all $\alpha,N$, there exists $C_{\alpha, N}>0$} so that 
$$
\sup_{(x,\xi)\in \TM \setminus V(h)}|\partial_{x,\xi}^\alpha a(x,\xi)|\leq C_{\alpha, N}h^N.
$$
For $B(h)\subset \TM $, will also write $\MSh(A)\cap B(h)=\emptyset$ for $\MSh(A)\subset (B(h))^c$.
\end{definition}
{Note that the notation $ \MSh(A)\subset V(h)$ is a shortening for $ \MSh(A)\subset \{V(h)\}_h$.}

%%%%%%%%%%%%%%%%%%%%%%%%%%%%%%%%%%%%%%%%%%%%%%%%%%%%%%%%%%%%%%%%%%%%%%%%%%%%%%%%
\begin{lemma}
Let $0\leq \delta <\frac{1}{2}$ and $\delta'>\delta$, $c>0$. Suppose that $A\in \Psi^{\comp}_{\delta}(M)$ and that $\MSh(A)\subset V(h)$. Then, 
$$
\MSh(A)\subset \Big\{ (x,\xi)\, \big| \, d\big( (x,\xi), V(h)^c\big)\leq ch^{\delta'}\Big\}.
$$
\end{lemma}
%%%%%%%%%%%%%%%%%%%%%%%%%%%%%%%%%%%%%%%%%%%%%%%%%%%%%%%%%%%%%%%%%%%%%%%%%%%%%%%%

\begin{proof}
Let $A=Op_h({a})+O(h^\infty)_{\Psi^{-\infty}}$. Suppose that 
$$
2r(h):=d\big(\rho_1, V(h)^c\big)\leq ch^{\delta'}
$$
and let $\rho_0 \in V(h)^c$ with $d(\rho_1,\rho_0)\leq r(h)$. 
Then, for any $N>0$, 
\begin{align*}
|\partial^\alpha a(\rho_1)|&\leq \sum_{|\beta|\leq N-1}|\partial^{\alpha+\beta}a(\rho_0)|r(h)^{|\beta|}+C_{|{\alpha}|+N}\sup_{|k|\leq |\alpha|+N,\TM }|\partial^k a|r(h)^N\\
&\leq \sum_{|\beta|\leq N-1}\sup_{V^c}|\partial^{\alpha+\beta}a(\rho)|r(h)^{|\beta|}+C_{\alpha N}h^{-N\delta}r(h)^N\\
&\leq C_{\alpha N M}h^M+C_{\alpha N} h^{-N\delta}r(h)^N
\end{align*}
So, letting $N\geq M(\delta'-\delta)^{-1}$, 
$$
|\partial^\alpha a(\rho_1)|\leq C_{\alpha M}h^{M}. 
$$
\end{proof}
%%%%%%%%%%%%%%%%%%%%%%%%%%%%%%%%%%%%%%%%%%%%%%%%%%%%%%%%%%%%%%%%%%%%%%%%%%%%%%%%

%%%%%%%%%%%%%%%%%%%%%%%%%%%%%%%%%%%%%%%%%%%%%%%%%%%%%%%%%%%%%%%%%%%%%%%%%%%%%%%%
\begin{lemma}
Let $0\leq \delta<\frac{1}{2}$ and $A,B\in \Psi_\delta^{\comp}(M)$. Suppose that $\MSh(A)\subset V(h)$ and $\MSh(B)\subset W(h)$.
\begin{enumerate}
\item The statement $\MSh(A)\subset V(h)$ is well defined. In particular, it does not depend on the choice of quantization procedure.
\item $\MSh(AB)\subset V(h)\cap W(h)$
\item $\MSh(A^*)\subset V(h)$
\item If $V(h)=\emptyset$, then $\WFh(A)=\emptyset$. 
\item If $A=Op_h(a)+O(h^\infty)_{\Psi^{-\infty}}$, then $\MSh(a)\subset \supp a$. 
\end{enumerate}
\end{lemma}
%%%%%%%%%%%%%%%%%%%%%%%%%%%%%%%%%%%%%%%%%%%%%%%%%%%%%%%%%%%%%%%%%%%%%%%%%%%%%%%%

\begin{proof}
The proofs of 1-3 are nearly identical, relying on the asymptotic expansion for, respectively, the change of quantization, composition, and adjoint so we write the proof for only (2).
Write 
$$
A=Op_h(a)+O(h^\infty)_{\Psi^{-\infty}}, \qquad B=Op_h(b)+O(h^\infty)_{\Psi^{-\infty}}.
$$
Then, 
$$
Op_h(a)Op_h(b)=Op_h(a\#b) +O(h^\infty)_{\Psi^{-\infty}}
$$
where 
$$
a\#b(x,\xi)\sim \sum_{j}h^jL_{2j} a(x,\xi)b(y,\eta)\Big|_{\substack{x=y\\\xi=\eta}}
$$
and $L_{2j}$ are differential operators of order $2j$. Suppose that $\MSh(A)\subset V$. Then, for any $N>0$.
$$
\sup_{V^c}|\partial^\alpha a|\leq C_{\alpha N}h^N.
$$
So, choosing $M>(N+\delta|\alpha|)(1-2\delta)^{-1}$, 
\begin{align*}
|\partial^\alpha a\#b|&\leq \Bigg|\partial^\alpha \sum_{j{<M}}h^jL_{2j} a(x,\xi)b(y,\eta)\Big|_{\substack{x=y\\\xi=\eta}}\Bigg|+ C_{\alpha M}h^{M(1-2\delta)-|\alpha|\delta}\leq C_{\alpha N}h^N
\end{align*}
In particular, 
$$
\sup_{V^c}|\partial^\alpha a\# b|\leq C_{\alpha N}h^N.
$$
An identical argument shows 
$$
\sup_{W^c}|\partial^\alpha a\# b|\leq C_{\alpha N}h^N.
$$

(4) follows from the definition since if $V(h)=\emptyset,$ $a\in h^\infty S_\delta$. 

(5) follows easily from the definition.
\end{proof}
\begin{lemma}
\label{l:flow}
{Let $\varphi_t:=\exp(tH_p)$ and $\Sigma \subset \TM$ compact.} There exists $\delta>0$ small enough and $C_1>0$ so that uniformly for $t\in [0,\delta ]$, {and $(x_i,\xi_i)\in \Sigma$.}
\begin{multline}
\label{e:flow-1} 
\frac{1}{2}d\big((x_1,\xi_1),(x_2,\xi_2)\big)-C_1d\big((x_1,\xi_1),(x_2,\xi_2)\big)^2\leq d\big(\varphi_t(x_1,\xi_2),\varphi_t(x_2,\xi_1)\big)\\
\leq 2d\big((x_1,\xi_1),(x_2,\xi_2)\big)+C_1d\big((x_1,\xi_1),(x_2,\xi_2)\big)^2
\end{multline}
where $d$ is the distance induced by the Sasaki metric.
Furthermore if $\varphi_t(x_i,\xi_i)=(x_i(t),\xi_i(t))$,
\begin{equation}
\label{e:flow-2} 
d_M(x_1(t),x_2(t))\leq d_M(x_1,x_2)+C_1d\big((x_1,\xi_1),(x_2,\xi_2)\big)\delta
\end{equation}
where $d_M$ is the distance induced by the metric on $M$.
\end{lemma}
%%%%%%%%%%%%%%%%%%%%%%%%%%%%%%%%%%%%%%%%%%%%%%%%%%%%%%%%%%%%%%%%%%%%%%%%%%%%%%%%

%%%%%%%%%%%%%%%%%%%%%%%%%%%%%%%%%%%%%%%%%%%%%%%%%%%%%%%%%%%%%%%%%%%%%%%%%%%%%%%%
\begin{proof}
By Taylor's theorem
\begin{multline*}
\varphi_t(x_1,\xi_1)-\varphi_t(x_2,\xi_2)=d_x\varphi_t(x_2,\xi_2)(x_1-x_2)+d_\xi \varphi_t(x_2,\xi_2)(\xi_1-\xi_2)\\
+O_{C^\infty}(\sup_{q\in \Sigma} |d^2\varphi_t(q)|(|\xi_1-\xi_2|^2+|x_1-x_2|^2)
\end{multline*}
Now,
$$
\varphi_t(x,\xi)=(x,\xi)+(\partial_\xi p(x,\xi) t,-\partial_xp(x,\xi) t)+O(t^2)
$$
so
\begin{gather*}
d_\xi \varphi_t(x,\xi)=(0,I)+t(\partial^2_\xi p ,-\partial^2_{\xi x}p )+O(t^2)\\
d_x \varphi_t(x,\xi)=(I,0)+t(\partial^2_{x\xi} p ,-\partial^2_{x}p )+O(t^2).
\end{gather*}
In particular,
\begin{align*}
\varphi_t(x_1,\xi_1)-\varphi_t(x_2,\xi_2)&= ((0,I)+O(t))(\xi_1-\xi_2)+((I,0)+O(t))(x_1-x_2)\\
&\qquad+O((\xi_1-\xi_2)^2+(x_1-x_2)^2)
\end{align*}
and choosing $\delta>0$ small enough gives the result.
\end{proof}
%%%%%%%%%%%%%%%%%%%%%%%%%%%%%%%%%%%%%%%%%%%%%%%%%%%%%%%%%%%%%%%%%%%%%%%%%%%%%%%%
%%%%%%%%%%%%%%%%%%%%%%%%%%%%%%%%%%%%%%%%%%%%%%%%%%%%%%%%%%%%%%%%%%%%%%%%%%%%%%%%
%%%%%%%%%%%%%%%%%%%%%%%%%%%%%%%%%%%%%%%%%%%%%%%%%%%%%%%%%%%%%%%%%%%%%%%%%%%%%%%%
\section{Proofs of Lemmas~\ref{l:loopsAreFun} and~\ref{l:dejaVu}}
%%%%%%%%%%%%%%%%%%%%%%%%%%%%%%%%%%%%%%%%%%%%%%%%%%%%%%%%%%%%%%%%%%%%%%%%%%%%%%%%
%%%%%%%%%%%%%%%%%%%%%%%%%%%%%%%%%%%%%%%%%%%%%%%%%%%%%%%%%%%%%%%%%%%%%%%%%%%%%%%%
%%%%%%%%%%%%%%%%%%%%%%%%%%%%%%%%%%%%%%%%%%%%%%%%%%%%%%%%%%%%%%%%%%%%%%%%%%%%%%%%
{
\label{a:oldrelations}
\begin{lemma}
\label{l:padMe}
Let $t,T>0$ and suppose that $G\subset S^*_xM$ is a closed set that is $[t,T]$ non-self looping. Then there is $R>0$ such that $B_{T^*M}(G,R)$ is $[t,T]$ non-self looping.
\end{lemma}
\begin{proof}
We will assume that $\varphi_s(G)\cap G=\emptyset$ for $s\in[t,T]$, the case of $s\in[-T,-t]$ being similar.
Let $q\in G$. We claim there is $R_q>0$ such that 
$$
\bigcup_{s\in[t,T]}\varphi_t(B_{T^*M}(q,R_q))\cap B_{T^*M}(G,R_q)=\emptyset.
$$ 
Suppose not. Then there are $q_n \to q$ and  $s_n\in[t,T]$ such that $d(\varphi_{s_n}(q_n),G)\to 0$. Extracting subsequences, we may assume $s_n\to s\in[t,T]$ and $\varphi_{s_n}(q_n)\to \rho \in G$. But then $\varphi_{s}(q)=\rho$ and, in particular, $G$ is not $[t,T]$ non-self looping. 

Now, 
$
G\subset \bigcup_{q\in G} B(q,R_q)
$ 
and hence, by compactness, there are $q_i$, $i=1,\dots N$, such that 
$
G\subset \bigcup_{i=1}^N B(q_i,R_{q_i}).
$
In particular, there is $0<R<\min_{i}R_{q_i}$ such that 
$
B(G,R)\subset \bigcup_{i=1}^N B(q_i,R_{q_i}).
$
This implies that $B(G,R)$ is $[t,T]$ non-self looping.
\end{proof}

\begin{lemma}
\label{l:coverMe}
Let $\tau,\mathfrak{D}, t,T>0$, $R(h)\geq 8h^\delta$, and $\{\Lambda_{\rho_j}^\tau(R(h))\}_{j\in \mc{G}}$ be a $(\mathfrak{D},\tau,R(h))$ good cover of $S^*_xM$. Suppose that $G\subset S^*_xM$ is closed and $[t,T]$ non-self looping. Then, for all $\e>0$, there is $R>0$ small enough such that for $R(h)<R$, 
$$
\mc{G}:=\{j\in \mc{J}\mid \Lambda_{\rho_j}^\tau(R(h))\cap B_{S^*_xM}(G,R)\neq \emptyset\}
$$
satisfies
\begin{equation}
\label{e:nslMe}
\bigcup_{j\in \mc{G}}\Lambda_{\rho_j}^\tau(R(h))\text{ is }[\max(t,3\tau),\max(t,3\tau,T)]\text{ non-self looping}
\end{equation}
and
\begin{equation}
\label{e:volMe}
|\mc{G}|\leq \mathfrak{D}R(h)^{1-n}(\vol_{S^*_xM}(G)+\e).
\end{equation}
\end{lemma}
\begin{proof}
By Lemma~\ref{l:padMe}, there is $R_0>0$ such that $B(G,R_0)$ is $[t,T]$ non-self looping. Furthermore, since $G$ is closed, there is $R_1>0$ such that 
$$
\vol_{S^*_xM}(B(G,R_1))<\vol_{S^*_xM}(G)+\e.
$$
Therefore, putting $R=\min(R_0/4,R_1/4)$, for $R(h)\leq R$, and $j\in \mc{G}$,
$$
\bigcup_{j\in\mc{G}}\Lambda_{\rho_j}^\tau(R(h))\cap S^*_xM\subset B_{T^*M}(G,\min(R_0,R_1)).
$$
In particular,~\eqref{e:nslMe} and~\eqref{e:volMe} hold.
\end{proof}

\begin{proof}[Proof of Lemma~\ref{l:loopsAreFun}]
Suppose that $x$ non-self focal. Let $\mc{L}_x^T:=T_+^{-1}([0,T])$ and note that for all $T>0$, $\mc{L}_x^T$ is closed. Thus, by Lemma~\ref{l:coverMe} for all $T>0$ there is $R_0=R_0(T)>0$ such that for $R(h)\leq R_0$, with
$\tilde{\mc{B}}:=\{ j\mid \Lambda_{\rho_j}^\tau(R(h))\cap B_{S^*_xM}(\mc{L}_x^T,R_0)\},$
one has $$
|\tilde{\mc{B}}|\leq \frac{R(h)^{1-n}}{T}.
$$

Next, since $G:=S^*_xM\setminus B(\mc{L}_{x}^T, R_0)$ is closed and $[\inj M/2,T]$ non-self looping, there is $R_1=R_1(T)>0$ such for $R(h)\leq R_1$ and 
$$
\mc{G}=\{ j\mid \Lambda_{\rho_j}^\tau(R(h))\cap B(G,R_1)\},
$$
equation~\eqref{e:nslMe} holds with $t=\inj M/2$ and $T=T$. Putting $R(T):=\min(R_1(T),R_2(T))$, 
$
\mc{B}:=\tilde{\mc{B}}\setminus \mc{G},
$
and defining
$$
h_0(T)=\inf\{h>0\,\mid\,  R(h)>R(T)\}, \qquad
T(h)=\sup\{T>0\mid h_0(T)>h\},
$$
we have shown that $x$ is $(\inj M/2, T(h))$ non-looping. 
\end{proof}

\begin{proof}[Proof of Lemma~\ref{l:dejaVu}]
Let $\mc{R}^{\pm,\delta,S}_x$ be the set of points $\rho \in S^*_xM$ for which there exists $0<\pm t\leq S$ such that $ \varphi_t(\rho)\in S^*_xM$ and   $d(\varphi_t(\rho),\rho)\leq \delta$. Then, 
$$
\mc{R}_x=\bigcap_{\delta>0}\bigcup_{S>0}\mc{R}_x^{\delta,S}, \qquad \mc{R}_x^{\delta, S}
:=\bigcap_{\pm}\mc{R}^{\pm,\delta,S}_x.
$$
Note that  $\mc{R}_x^{\delta, S}$  is closed for all $\delta, S$, and that  for all $\e>0$ there is $\delta>0$ such that for all $S>0$
$$
\vol_{S^*_xM}(\mc{R}_x^{S,\delta})\leq \vol_{S^*_xM}(\mc{R}_x)+\e.
$$

Now, assume that $x$ is non-recurrent. Then for all $\e>0$, there is $\delta=\delta(\e)>0$ such that for all $S>0$
$$
\vol_{S^*_xM}(\mc{R}_x^{S,\delta})\leq \e.
$$
Let $\{\rho_i\}_{i=1}^{N(\delta)}\subset S^*_xM$ be such that $S^*_xM\subset \cup_i B(\rho_i, \delta/4)$ and $N(\delta)\leq C\delta^{1-n}$.  

Letting $G_0:=\mc{R}_x^{S,\delta}$, by Lemma~\ref{l:coverMe} there is $R_0=R_0(\e,S)>0$ such that for $R(h)\leq R_0$, defining $\tilde{\mc{G}}_0:=\{ j\mid \Lambda_{\rho_j}^\tau(R(h))\cap B_{S^*_xM}(G_i,R_0)\}$, we have
\begin{gather*}
%\tilde{\mc{G}}_0:=\{ j\mid \Lambda_{\rho_j}^\tau(R(h))\cap B_{S^*_xM}(G_i,R_0)\},\\
|\tilde{\mc{G}}_0|\leq \mathfrak{D}R(h)^{1-n}\e,\qquad %|\tilde{\mc{G}}_i|\leq R(h)^{1-n}\mathfrak{D}\delta^{n-1},\,i\geq 1
\end{gather*}
Next, let $G_i:= \overline{B_{S^*_xM}(\rho_i,\delta/4)}\setminus B_{S^*_xM}(\mc{R}_x^{T,\delta},R_0)$ so that $G_i$ is closed and $[\inj M/2,S]$ non-self looping.  By Lemma~\ref{l:coverMe}, there are $R_i=R_i(\e,S)>0$ such that for $R(h)\leq \min_iR_i$, if we set $\tilde{\mc{G}}_i:=\{ j\mid \Lambda_{\rho_j}^\tau(R(h))\cap B_{S^*_xM}(G_i,R_i)\}$, then
\begin{gather*}
%\tilde{\mc{G}}_i:=\{ j\mid \Lambda_{\rho_j}^\tau(R(h))\cap B_{S^*_xM}(G_i,R_i)\},\\
%|\tilde{\mc{G}}_0|\leq \mathfrak{D}R(h)^{1-n}\e,
|\tilde{\mc{G}}_i|\leq R(h)^{1-n}\mathfrak{D}\delta^{n-1},\,i\geq 1
\end{gather*}
and for $i\geq 1$,
$$
\bigcup_{j\in \tilde{\mc{G}}_i}\Lambda_{\rho_j}^\tau(R(h))\text{ is }[\inj M/2,S]\text{ non-self looping}.
$$
Then, we have 
$$
\sum_{i=0}^{N}\sqrt{\frac{|\tilde{\mc{G}}_i|R(h)^{n-1}\inj M}{2S}}\leq N(\delta)\delta^{\frac{n-1}{2}} \sqrt{\frac{\mathfrak{D}\inj M}{2S}}+\sqrt{\mathfrak{D}\e}.
$$
Now, for $\e:=\frac{1}{4\mathfrak{D}T}$ let $\delta:=\delta(\e)$ and set $S:=2N^2(\delta)\delta^{n-1}\mathfrak{D}\inj M $. Working with $R_i=R_i(\ep, S)=R_i(T)$ as defined before, we have
$$
\sum_{i=0}^{N}\sqrt{\frac{|\tilde{\mc{G}}_i|R(h)^{n-1}\inj M}{2S}}\leq \sqrt{\frac{1}{T}}.
$$
Defining
$$
h_0(T)=\inf\{h>0\,\mid\,  R(h)>\min_iR_i(T)\}, \qquad
T(h)=\sup\{T>0\mid h_0(T)>h\},
$$
we have shown that $x$ is $(\inj M/2, T(h))$ non-recurrent. 
\end{proof}
}

\bibliography{biblio}
\bibliographystyle{alpha}

\end{document}